\documentclass[oneside, a4paper,reqno]{amsart}

\usepackage{pdfsync}

\usepackage[applemac]{inputenc} 

\usepackage{stmaryrd}
\usepackage{mathrsfs}
\usepackage{hyperref}
\usepackage{amsmath, amsthm, amscd, amssymb, latexsym, eucal}
\usepackage[all]{xy}

 \calclayout \makeatletter
 \calclayout \makeatletter
\def\serieslogo@{} \def\@setcopyright{} \makeatother
\makeatletter
\renewcommand*\env@matrix[1][c]{\hskip -\arraycolsep
  \let\@ifnextchar\new@ifnextchar
  \array{*\c@MaxMatrixCols #1}}
\makeatother

\usepackage{multienum}

\usepackage[colorinlistoftodos]{todonotes}

\usepackage{hyperref}
\usepackage{color}
\usepackage[nospace,noadjust]{cite}

\hypersetup{colorlinks=true,
     breaklinks=true,
     linkcolor=blue,
     citecolor=red,    
     bookmarks=true,
     pageanchor=true}   
\numberwithin{equation}{section}
\newtheorem{thm}{Theorem}[section]
\newtheorem{cor}[thm]{Corollary}
\newtheorem{lem}[thm]{Lemma}
\newtheorem{prop}[thm]{Proposition}

\theoremstyle{definition}
\newtheorem{defn}[thm]{Definition}
\newtheorem{rem}[thm]{Remark}
\newtheorem{exam}[thm]{Example}

\newtheorem*{ackn}{Acknowledgments}

\newcommand{\lxr}{\longrightarrow}

\newcommand{\A}{\mathscr A}
\newcommand{\B}{\mathscr B}
\newcommand{\C}{\mathscr C}

\newcommand{\Hcal}{\mathcal H}

\newcommand{\M}{\mathcal M}

\newcommand{\R}{\mathcal R}
\newcommand{\T}{\mathcal T}
\newcommand{\U}{\mathcal U}
\newcommand{\V}{\mathcal V}
\newcommand{\W}{\mathcal W}
\newcommand{\X}{\mathcal X}
\newcommand{\Y}{\mathcal Y}
\newcommand{\Z}{\mathcal Z}

\newcommand{\Acal}{\mathscr A}

\newcommand{\Ccal}{\mathscr C}
\newcommand{\Xcal}{\mathcal X}
\newcommand{\Ycal}{\mathcal Y}
\newcommand{\Zcal}{\mathcal Z}

\newcommand{\cone}{\mathsf{cone}}

\newcommand{\mi}{\mathsf{i}}

\newcommand{\iso}{\cong}

\DeclareMathOperator*{\Ker}{\mathsf{Ker}}
 \DeclareMathOperator*{\Image}{\mathsf{Im}}
\DeclareMathOperator*{\Coker}{\mathsf{Coker}}

\DeclareMathOperator*{\id}{\mathsf{id}}

\DeclareMathOperator*{\Mod}{\mathsf{Mod}-\!}

\DeclareMathOperator*{\End}{\mathsf{End}}
 \DeclareMathOperator*{\smod}{\mathsf{mod}-\!}

\DeclareMathOperator*{\Inj}{\mathsf{Inj}-\!}
\DeclareMathOperator*{\Proj}{\mathsf{Proj}-\!}

\DeclareMathOperator*{\Gen}{\mathsf{Gen}}

\DeclareMathOperator*{\Add}{\mathsf{Add}}
 \DeclareMathOperator*{\Prod}{\mathsf{Prod}}

\DeclareMathOperator{\Hom}{\mathsf{Hom}}

\DeclareMathOperator*{\Ext}{\mathsf{Ext}}
\DeclareMathOperator*{\Tor}{\mathsf{Tor}}

\DeclareMathOperator*{\aisle}{\mathsf{aisle}}

\DeclareMathOperator*{\real}{\mathsf{real}}

\newcommand{\iden}{\operatorname{Id}\nolimits}

\newsavebox{\proofbox}
\savebox{\proofbox}{\begin{picture}(7,7)%
  \put(0,0){\framebox(7,7){}}\end{picture}}

\newcommand{\zvek}[2]{{\left(\begin{smallmatrix} {#1} & {#2} \end{smallmatrix}\right)}}
\newcommand{\svek}[2]{{\left(\begin{smallmatrix} {#1} \\ {#2} \end{smallmatrix}\right)}}
\newcommand{\tzmat}[4]{{\big[\begin{smallmatrix} {#1} & {#2} \\ {#3} & {#4} \end{smallmatrix}\big]}}
\newcommand{\define}[1]{{\textbf{#1}}}
\usepackage{enumitem}

\begin{document}
\title{Realisation functors in tilting theory}
\author{Chrysostomos Psaroudakis, Jorge Vit\'oria}
\address{Institute of Algebra and Number Theory, University of Stuttgart, Pfaffenwaldring 57, 70569 Stuttgart, Germany} \email{chrysostomos.psaroudakis@mathematik.uni-stuttgart.de}
\address{Department of Mathematics, City, University of London, Northampton Square, London EC1V 0HB,\newline United Kingdom}\email{jorge.vitoria@city.ac.uk}
\thanks{The first named author was supported by the Norwegian Research Council (NFR 221893) under the project \textit{Triangulated categories in Algebra}. The second named author was initially supported by the German Research Council through the SFB 701 in Bielefeld, then by a Marie Curie Intra-European Fellowship within the 7th European Community Framework Programme (PIEF-GA-2012-327376) and, in the final stages of this project, by the Engineering and Physical Sciences Research Council of the United Kingdom, grant number EP/N016505/1.}

\maketitle
\begin{abstract}
Derived equivalences and t-structures are closely related. We use realisation functors associated to t-structures in triangulated categories to establish a derived Morita theory for abelian categories with a projective generator or an injective cogenerator. For this purpose we develop a theory of (non-compact, or large) tilting and cotilting objects that generalises the preceding notions in the literature. Within the scope of derived Morita theory for rings we show that, under some assumptions, the realisation functor is a derived tensor product. This fact allows us to approach a problem by Rickard on the shape of derived equivalences. Finally, we apply the techniques of this new derived Morita theory to show that a recollement of derived categories is a derived version of a recollement of abelian categories if and only if there are tilting or cotilting t-structures glueing to a tilting or a cotilting t-structure. As a further application, we answer a question by Xi on a standard form for recollements of derived module categories for finite dimensional hereditary algebras. 
\end{abstract}

\setcounter{tocdepth}{1} \tableofcontents

\section{Introduction}
\subsection{Motivation} Derived categories, equivalences between them and the associated derived invariants are central objects of study in modern representation theory and algebraic geometry. In representation theory, the results of Rickard (\cite{Rick}) and Keller (\cite{Keller:dg}) on a derived Morita theory for rings show that compact tilting complexes guarantee the existence of derived equivalences and vice-versa. Some derived equivalences can even be described as a derived tensor product with a complex of bimodules (\cite{Rick2}). In algebraic geometry, derived equivalences between coherent sheaves of smooth projective varieties all have a standard form: Fourier-Mukai transforms (\cite{Or}). In both settings, there is a concern with the \textit{existence} and the \textit{shape} of derived equivalences. In this paper we propose a unifying approach to the study of derived equivalences of abelian categories: they should be regarded as realisation functors of certain t-structures. In doing so, we are in particular able to establish a derived Morita theory for abelian categories with a projective generator or an injective cogenerator. This is done in terms of a \textit{non-compact} or \textit{large} tilting theory, extending and, in some sense, further clarifying the classical cases mentioned above.

In representation theory, there are some motivating predecessors of the \textit{non-compact} tilting theory that we develop in this paper: large tilting and cotilting modules over rings (\cite{CT,AC,CGM,Stovicek}) and large silting complexes (\cite{Wei,AMV1}). Such non-compact counterparts of the classical theory were largely motivated by the search of properties that are difficult to obtain in the compact world, namely within the realm of approximation theory. However, contrary to the compact case (\cite{Hap,Rick,Keller:dg}), these non-compact objects lack a certain \textit{derived flavour}: their endomorphism rings are not derived equivalent to the original ring; they are usually \textit{too big} (\cite{Baz, BMT, CX1}). Here is where our approach to derived equivalences, inspired by that of \cite{Stovicek}, comes to rescue: instead of considering endomorphism rings, one should consider the hearts of the naturally associated t-structures. The corresponding realisation functors then yield derived equivalences.

\subsection{The main results in context} 
A large class of t-structures can be generated from the concept of silting object, which was first defined in \cite{KV}. In the bounded derived category of finitely generated modules over a finite dimensional algebra, compact silting objects classify bounded t-structures whose heart is a module category  (\cite{KN, KY}). Compact silting objects were also considered in abstract triangulated categories (\cite{AI,BR,HKM,IyJoYa}) and, more recently, non-compact silting objects and their associated t-structures were studied in derived module categories (\cite{Wei, AMV1}). We introduce a common generalisation of these notions for arbitrary triangulated categories (see also \cite{NSZ} for parallel work by Nicol\'as, Saorin and Zvonareva on this topic). We also introduce the dual notion of a cosilting object, the 2-term version of which is independently dealt with in \cite{BreazPop}. 

Among silting and cosilting objects, tilting and cotilting objects play a special role: they are the ones providing derived equivalences. Indeed, we show (see Proposition \ref{der equival}) that realisation functors associated with silting or cosilting t-structures are fully faithful (and, thus, equivalences with their essential images) if and only if the t-structures are in fact tilting or cotilting. As a consequence of this fact we are then able to establish a derived Morita theory for abelian categories with a projective generator or an injective cogenerator. We refer to Definitions  \ref{defn restrict extend} and \ref{bounded silting}, as well as to the examples thereafter, for the meaning of \textit{restrictable equivalence} and \textit{bounded (co)tilting} (these are conditions that allow restrictions to the setting of bounded derived categories).\medskip

\noindent\textbf{Theorem A} (\ref{eq tilt cotilti shape}) \textit{Let $\A$ and $\B$ be abelian categories such that $\mathsf{D}(\A)$ is TR5 (respectively, TR5*) and $\B$ has a projective generator (respectively, an injective cogenerator). Consider the following statements.
\begin{enumerate}
\item There is a restrictable triangle equivalence $\Phi\colon \mathsf{D}(\B)\lxr \mathsf{D}(\A)$.
\item There is a bounded tilting (respectively, cotilting) object $M$ in $\mathsf{D}(\A)$ such that $\Hcal_M\cong \B$.
\item There is a triangle equivalence $\phi\colon \mathsf{D}^b(\B)\lxr \mathsf{D}^b(\A)$.
\end{enumerate}
Then we have \textnormal{(i)}$\Longrightarrow$\textnormal{(ii)}$\Longrightarrow$\textnormal{(iii)}. Moreover, if $\B$ has a projective generator and $\A=\Mod{R}$, for a ring $R$, then we also have \textnormal{(iii)}$\Longrightarrow$\textnormal{(ii)}.
}\medskip

Note that, in particular, Theorem A provides a derived Morita theory for Grothendieck abelian categories, thus covering derived equivalences between not only module categories but also categories of quasicoherent sheaves or certain functor categories, for example. Again, this result stresses that in order to have a derived equivalence arising from a possibly non-compact tilting object, one should look to its heart rather than to its endomorphism ring. In the compact case it so happens that both provide the same information about the derived category but as shown in \cite{Baz} and \cite{CX1}, for example, the endomorphism ring of a large tilting module will, in general, provide a recollement rather than a derived equivalence.

It is often easier to know about the existence of a derived equivalence rather than a concrete expression (or \textit{shape}, for short) for such a functor. Realisation functors satisfy certain \textit{naturality} properties (see Theorem \ref{lift implies diagram}, recalling  \cite[Lemma A7.1]{B}) that contribute to the problem of comparing different equivalence functors and establishing a \textit{standard form}. Although this problem was solved in algebraic geometry (every equivalence between derived categories of coherent sheaves between two smooth projective varieties is a Fourier-Mukai transform - see \cite{Or}), in representation theory it is wide open. For algebras which are projective over a commutative ring, Rickard showed that for every equivalence between derived module categories, there is one of \textit{standard type}, i.e. one that is the derived tensor product with a complex of bimodules (\cite{Rick2}). It remains a question whether every derived equivalence is of standard type, as conjectured by Rickard. There are indications that this should hold. 
Recently, it was shown in \cite{XWC} and \cite{XWCYY} that derived equivalences between certain finite dimensional algebras (including piecewise hereditary algebras and some Frobenius algebras of radical square zero) are indeed of standard type.
In this paper, we prove that all these derived equivalences are, in essence, realisation functors associated to tilting or cotilting objects. Although realisation functors are not unique, we provide new criteria for an equivalence to be of standard type and we show that some realisation functors are of standard type.\medskip

\noindent\textbf{Theorem B} (\ref{standard f-lifts}) \textit{
Let $A$ and $B$ be $\mathbb{K}$-algebras which are projective over a commutative ring $\mathbb{K}$. Let $T$ be a compact tilting object in $\mathsf{D}(A)$ such that $\End_{\mathsf{D}(A)}(T)\cong B$. Then the functor $\mathsf{real}_T$ is an equivalence of standard type. Moreover, a triangle equivalence $\phi\colon \mathsf{D}^b(B)\lxr \mathsf{D}^b(A)$ is of standard type if and only if $\phi$ admits an f-lifting $\Phi\colon \mathsf{DF}^b(B)\lxr \mathsf{DF}^b(A)$ to the filtered bounded derived categories.}\medskip

Finally, we discuss recollements and equivalences between them using, once again, realisation functors. Recollements of abelian or triangulated categories are specially well-behaved decompositions (in particular, short exact sequences) of the underlying category. Recollements of abelian categories are well-understood (\cite{FP}), especially if all terms are categories of modules over a ring (\cite{PsaroudVitoria}). The same cannot be said about recollements of derived categories, even in the case where all categories are derived module categories. A natural question (formulated by Xi for derived module categories) is whether every recollement of derived categories is \textit{equivalent} to a derived version of a recollement of abelian categories. This is not true in general as shown by a counterexample in \cite{AKLY3}. In this paper, we use realisation functors to provide a criterion for such an equivalence of recollements to exist in terms of the glueing of tilting t-structures. A different criterion for recollements of derived categories of rings has been independently obtained in \cite{AKLY3}. In the case of a recollement by derived module categories for algebras which are projective over a commutative ring, we prove the following result, which can be thought of as a statement about glueing derived equivalences.\medskip

\noindent\textbf{Theorem C} (\ref{rec strat k-alg}) 
\textit{Let $A$, $B$ and $C$ be $\mathbb{K}$-algebras over a commutative ring $\mathbb{K}$ and assume that $A$ is projective as a $\mathbb{K}$-module. Suppose there is a recollement $\R$ of the form
\begin{equation}\nonumber
\R\colon\ \ \ \ \ \ \xymatrix{\mathsf{D}(B)  \ar[rrr]^{i_*} &&& \mathsf{D}(A)  \ar[rrr]^{j^*}  \ar @/_1.5pc/[lll]_{i^*}  \ar @/^1.5pc/[lll]^{i^!} &&& \mathsf{D}(C). \ar @/_1.5pc/[lll]_{j_!} \ar @/^1.5pc/[lll]^{j_*} }
\end{equation}
The following statements are equivalent.
\begin{enumerate}
\item There is a $\mathbb{K}$-algebra $S$, projective over $\mathbb{K}$, and an idempotent element $e$ of $S$ such that the canonical ring epimorphism $S\lxr S/SeS$ is homological and the associated recollement of $\mathsf{D}(S)$ by derived module categories is equivalent to $\R$.
\item There are compact tilting objects $V$ in $\mathsf{D}(A)$, $U$ in $\mathsf{D}(B)$ and  $W$ in $\mathsf{D}(C)$ such that the associated tilting t-structures in $\mathsf{D}(B)$ and $\mathsf{D}(C)$ glue along $\R$ to the associated tilting t-structure in $\mathsf{D}(A)$ and such that the $\mathbb{K}$-algebra $\End_{\mathsf{D}(A)}(V)$ is projective over $\mathbb{K}$.
\end{enumerate} }

\subsection{Structure of the paper}
This paper is organised, roughly, in a sequential way. Sections 3 and 4 are independent of each other, but they are both essential for Section 5. Section 6 uses results from all preceding sections. In order to facilitate the understanding of the later sections (where our main results lie) we include in the beginning of each section an informal overview of its results, for the reader that might wish to skip some of the earlier material.

We begin in \textbf{Section 2} with some preliminaries on t-structures, recollements and the relation between the two: glueing. These are the well-known concepts that we will use throughout the paper. \textbf{Section 3} discusses some technical but necessary issues regarding the construction of realisation functors, combining the approach of \cite{BBD} as well as that presented in \cite[Appendix]{B}. We explore at length all the necessary properties for the later sections, including some proofs (or sketches of proof) of older results not available or hard to find in the literature. The results in this section are then used throughout sections 5 and 6, where we apply these properties to study equivalences between derived categories of abelian categories or, more generally, between recollements of derived categories. We also show that the realisation functor of the standard t-structure is, as expected, the identity functor. In \textbf{Section 4} we develop our generalised notion of silting and cosilting objects in triangulated categories and we study the properties of the associated hearts. We introduce the notion of bounded silting and bounded cosilting objects, preparing ground for a discussion regarding the relation between derived equivalences at the bounded level and at the unbounded level. This is a recurrent issue throughout the paper, related with the fact that a realisation functor has as domain a bounded derived category. In \textbf{Section 5}, we focus on derived equivalences between certain types of abelian categories, both on their existence and on their shape, in the spirit of the above paragraphs. Examples related with the representation theory of infinite quivers and with derived equivalences in algebraic geometry are also discussed. Finally, in \textbf{Section 6} we study recollements of unbounded derived categories: methods to generate them and equivalences between them. We provide criteria in a rather general framework for a recollement of derived categories to be the derived version of a recollement of abelian categories. At the end, as an application, we show that this is always the case for derived categories of hereditary finite dimensional algebras.

This paper also includes an appendix by Ester Cabezuelo Fern\'andez and Olaf M.~Schn{\"u}rer. It discusses a detailed proof for the fact that the realisation functor built as in \cite[Appendix]{B} is indeed a triangle functor. For this purpose, however, it seems necessary to consider an extra axiom, as first proposed in \cite{Schn}, for filtered enhancements of triangulated categories. We refer to Remark \ref{fcat7} and to Appendix \ref{sec:beil-real-funct} for a detailed discussion.

\begin{ackn}
The authors would like to thank Lidia Angeleri H\"ugel and Steffen Koenig for their support in this project and their detailed comments on a preliminary version of this paper; Olaf Schn\"urer for his help in clarifying the proof of Theorem \ref{lift implies diagram} and Greg Stevenson for suggesting the use of a d\'evissage argument for the proof of Theorem \ref{prophomolemb}; Silvana Bazzoni, Henning Krause, Pedro Nicol\'as and Manuel Saorin for many discussions, questions and comments regarding this project; Changchang Xi and his research group at the Capital Normal University for many questions and comments on this work during a visit of the first named author;  the University of Bielefeld, the University of Stuttgart, the University of Verona and the Norwegian University of Science and Technology for hosting the authors in several occasions of their continued collaboration. 
\end{ackn}
\section{Preliminaries: t-structures, recollements and glueing}
\subsection{Conventions and notation} 
In this paper we consider only \textit{abelian categories with the property that the derived category has $\Hom$-sets}. In most contexts, however, the derived categories occurring here come from abelian categories with either enough injectives or enough projectives - and these will have $\Hom$-sets. Given an abelian category $\A$, we denote by $\mathsf{D}(\A)$ its derived category. If $A$ is a unitary ring, we denote by $\Mod{A}$ the category of right $A$-modules and by $\mathsf{D}(A)$ its derived category. Right bounded, left bounded or bounded derived categories are denoted as usual by $\mathsf{D}^-$, $\mathsf{D}^+$ and $\mathsf{D}^b$, respectively. 
For any triangulated category $\T$, we denote by $[1]$ its suspension functor. 

For a category $\C$, we denote by $\mathsf{Ob}\ \C$ its class of objects. The word \textit{subcategory}, unless otherwise stated, stands for a full and strict subcategory. For an additive functor $F\colon \A \lxr \B$ between additive categories  the essential image of $F$ is the subcategory of $\B$ given by $\Image F=\{B \in \B \ | \ B \cong F(A) \ \text{for some} \ A \in \A\}$ and the kernel of $F$ is the subcategory of $\A$ given by $\Ker F = \{A \in \A \ | \ F(A) = 0\}$. If $F$ is a right exact (respectively, left exact) functor between abelian categories, we denote its left derived functor by $\mathbb{L}F$ (respectively, $\mathbb{R}F$). If $F$ is exact, its derived functor will often be also denoted by $F$.

\subsection{t-structures}
We begin with recalling the definition of the key notion of this paper.
\begin{defn}\cite{BBD}\label{defn t-str}
A {\bf t-structure} in a triangulated category $\T$ is a pair $\mathbb{T}=(\mathbb{T}^{\leq 0},\mathbb{T}^{\geq 0})$ of full subcategories such that, for $\mathbb{T}^{\leq n}:=\mathbb{T}^{\leq 0}[-n]$ and $\mathbb{T}^{\geq n}:=\mathbb{T}^{\geq 0}[-n]$ ($n\in \mathbb Z$), we have:
\begin{enumerate}
\item $\Hom_{\T}(\mathbb{T}^{\leq 0},\mathbb{T}^{\geq 1})=0$, i.e. $\Hom_{\T}(X,Y)=0$ for all $X$ in $\mathbb{T}^{\leq 0}$ and $Y$ in $\mathbb{T}^{\geq 1}$;
\item $\mathbb{T}^{\leq 0}\subseteq \mathbb{T}^{\leq 1}$ and $\mathbb{T}^{\geq 1}\subseteq \mathbb{T}^{\geq 0}$;
\item For every object $X$ in $\T$ there are $Y$ in $\mathbb{T}^{\leq 0}$, $Z$ in $\mathbb{T}^{\geq 1}$ and a triangle $Y \lxr X \lxr Z \lxr  Y[1]$.
\end{enumerate}
The subcategories $\mathbb{T}^{\leq 0}$, $\mathbb{T}^{\geq 0}$ and $\Hcal(\mathbb{T}):=\mathbb{T}^{\leq 0}\cap\mathbb{T}^{\geq 0}$ are called, respectively, the \textbf{aisle}, the \textbf{coaisle} and the \textbf{heart} of $\mathbb{T}$. 
\end{defn}

It follows from \cite{BBD} that the heart of a t-structure $\mathbb{T}$ in $\T$ is an abelian category with the exact structure induced by the triangles of $\T$ lying in $\Hcal(\mathbb{T})$. Furthermore, there is a \textbf{cohomological functor} (i.e. a functor sending triangles in $\T$ to long exact sequences in $\Hcal(\mathbb{T})$) defined by:
\[
\textsf{H}^0_\mathbb{T}\colon \T\lxr \mathcal{H}(\mathbb{T}), \ X \ \mapsto \textsf{H}^0_{\mathbb{T}}(X):=\tau^{\geq 0}\tau^{\leq 0}(X)\iso \tau^{\leq 0}\tau^{\geq 0}(X),
\]
where $\tau^{\leq 0}\colon \T\lxr \mathbb{T}^{\leq 0}$ and $\tau^{\geq 0}\colon \T\lxr \mathbb{T}^{\geq 0}$ are the \textbf{truncation functors} (i.e. the right and left adjoins, respectively, of the inclusions of $\mathbb{T}^{\leq 0}$ and $\mathbb{T}^{\geq 0}$ in $\T$). Similarly, one can define functors $\tau^{\leq n}$, $\tau^{\geq n}$ and $H^n_\mathbb{T}:=(\tau^{\leq n}\tau^{\geq n})[n]$, for any integer $n$. The triangle in Definition \ref{defn t-str}(iii) can be expressed functorially as 
$$\xymatrix{\tau^{\leq 0}X\ar[r]^{\ f} & X\ar[r]^{g \ \ \ }& \tau^{\geq 1}X\ar[r]& (\tau^{\leq 0}X)[1]}$$
where the maps $f$ and $g$ come, respectively, from the counit and unit of the relevant adjunctions. In particular, it follows that if $f=0$ (respectively, $g=0$), then $\tau^{\leq 0}X=0$ (respectively, $\tau^{\geq 1}X=0$). Note also that the aisle $\mathbb{T}^{\leq 0}$ determines the t-structure ($\mathbb{T}^{\geq 1}=\Ker \Hom_\T(\mathbb{T}^{\leq 0},-)$, see also \cite{KV}).
\begin{exam}
For an abelian category $\A$, there is a \textbf{standard t-structure} $\mathbb{D}_\A:=(\mathbb{D}_\A^{\leq 0},\mathbb{D}_\A^{\leq 0})$ in its derived category $\mathsf{D}(\A)$ defined by the complex cohomology functors ($\textsf{H}_0^i$, for all $i$ in $\mathbb{Z}$) as follows:
$$\mathbb{D}_\A^{\leq 0}:=\{X\in\mathsf{D}(\A) \ | \ \textsf{H}_0^i(X)=0, \forall i>0\},  \ \ \ \ \mathbb{D}_\A^{\geq 0}:=\{X\in\mathsf{D}(\A) \ | \ \textsf{H}_0^i(X)=0, \forall i<0\}.$$
The heart of this t-structure consists of the complexes with cohomologies concentrated in degree zero and, thus, it is equivalent to $\A$. Moreover, the associated cohomology functors coincides with the complex cohomologies. Note also that the standard t-structure restricts to the (right, left) bounded derived category $\mathsf{D}^b(\A)$. Throughout we fix the notation in this example for the standard t-structure, although the subscript $\A$ will be omitted whenever $\A$ is fixed.
\end{exam}

A t-structure $\mathbb{T}=(\mathbb{T}^{\leq 0},\mathbb{T}^{\geq 0})$ in $\T$ is said to be \textbf{nondegenerate} if $\cap_{n\in\mathbb{Z}}\mathsf{Ob}\ \mathbb{T}^{\leq n}=0=\cap_{n\in\mathbb{Z}}\mathsf{Ob}\ \mathbb{T}^{\geq n}$ and  \textbf{bounded} if $\cup_{n\in\mathbb{Z}}\mathsf{Ob}\ \mathbb{T}^{\leq n}=\mathsf{Ob}\ \T=\cup_{n\in\mathbb{Z}}\mathsf{Ob}\ \mathbb{T}^{\geq n}$. Given an abelian category $\A$, it is easy to see that the standard t-structure is nondegenerate in both $\mathsf{D}^b(\A)$ and $\mathsf{D}(\A)$, but bounded only in $\mathsf{D}^b(\A)$.

For functors between triangulated categories endowed with t-structures there is a natural notion of (left, right) exactness. Given two triangulated categories $\T$ and $\V$ endowed with t-structures $\mathbb{T}=(\mathbb{T}^{\leq 0},\mathbb{T}^{\geq 0})$ and $\mathbb{V}=(\mathbb{V}^{\leq 0},\mathbb{V}^{\geq 0})$ respectively, a functor $\alpha\colon \T\lxr \V$ is said to be \textbf{left t-exact} if $\alpha(\mathbb{T}^{\geq 0})\subseteq \mathbb{V}^{\geq 0}$, \textbf{right t-exact} if $\alpha(\mathbb{T}^{\leq 0})\subseteq \mathbb{V}^{\leq 0}$ and \textbf{t-exact} if it is both left and right t-exact. As an example, consider two abelian categories $\A$ and $\B$ and a functor $F\colon \A\lxr \B$. If $F$ is exact, then its derived functor is t-exact with respect to the standard t-structures in $\mathsf{D}(\A)$ and $\mathsf{D}(\B)$. If $\A$ has enough projectives (respectively, injectives) and $F$ is right (respectively, left) exact, then the left derived functor $\mathbb{L}F$ is right t-exact (respectively, the right derived functor $\mathbb{R}F$ is left t-exact).

\subsection{Recollements}
Recall first that a diagram of the form
$$\xymatrix{0\ar[r]&\U\ar[r]^{i}&\T\ar[r]^q&\V\ar[r]&0}$$
is said to be a short exact sequence of triangulated (respectively, abelian) categories if the functor $i$ is fully faithful, $i(\U)$ is a \textbf{thick subcategory} of $\T$, i.e. a triangulated subcategory closed under summands (respectively, a \textbf{Serre subcategory} in the abelian case, i.e. a subcategory closed under extensions, subobjects and quotients) and if $q$ induces an equivalence  $\T/i(\U)\cong \V$.  Note that, in this case, $\Image{i}=\Ker{q}$ and $q$ is essentially surjective. 

Recollements are particularly well-behaved short exact sequences. Recollements of both abelian and triangulated categories appeared in \cite{BBD}, but the properties of recollements of abelian categories were only explored later, for example in \cite{FP}.
\begin{defn}\label{rec}
Let $\U$, $\T$ and $\V$ be triangulated (respectively, abelian) categories. A {\bf recollement of $\T$ by $\U$ and $\V$} is a diagram of triangle (respectively, additive) functors 
\[
\xymatrix@C=0.5cm{
\U \ar[rrr]^{i_*} &&& \T \ar[rrr]^{j^*}  \ar @/_1.5pc/[lll]_{i^*}  \ar
 @/^1.5pc/[lll]_{i^!} &&& \V
\ar @/_1.5pc/[lll]_{j_!} \ar
 @/^1.5pc/[lll]_{j_*}
 } 
 \]
satisfying the following conditions:
\begin{enumerate}
\item[(i)] $(i^*,i_*,i^!)$ and $(j_!,j^*,j_*)$ are adjoint triples;

\item[(ii)] The functors $i_*$, $j_!$, and $j_*$ are fully faithful;

\item[(iii)] $\Image{i_*}=\Ker{j^*}$. 
\end{enumerate}
\end{defn}
Generally, we will use the symbols $\U$, $\T$ and $\V$ to denote triangulated categories (and we write $\textsf{R}_{\mathsf{tr}}(\U,\T,\V)$ for the recollement in Definition \ref{rec}) and the symbols $\B$, $\A$ and $\C$ to denote abelian categories (and we write $\textsf{R}_{\mathsf{ab}}(\B,\A,\C)$ for the analogous recollement to the one in Definition \ref{rec}). Observe that it follows easily from the definition of recollement that the compositions $i^*j_!$ and $i^!j_*$ are identically zero and that the units $\iden\lxr i^!i_*$ and $\iden\lxr j^*j_!$ and the counits $i^*i_*\lxr\iden$ and $j^*j_*\lxr \iden$ of the adjunctions are natural isomorphisms (where the identity functors are defined in the obvious categories). Moreover, it can also be shown (see, for example, \cite{BBD} and \cite{Psaroud:homolrecol}) that for a recollement of triangulated categories $\mathsf{R}_{\mathsf{tr}}(\U,\T,\V)$, for every object $X$ in $\T$, the remaining units and the counits of the adjunctions induce triangles
\[
 \ \ \ \ \ \ \ \ \xymatrix@C=0.5cm{j_!j^*X\ar[r]&X\ar[r]&i_*i^*X\ar[r]&j_!j^*X[1] } \ \ \text{and} \ \ \xymatrix@C=0.5cm{i_*i^!X\ar[r]&X\ar[r]&j_*j^*X\ar[r]&i_*i^!X[1].}
\]  
Similarly, for a recollement of abelian categories $\mathsf{R}_{\mathsf{ab}}(\B,\A,\C)$, for every object $X$ in $\A$, the remaining units and counits of the adjunctions induce exact sequences$\colon$
\[
 \ \ \ \ \ \ \ \ \xymatrix@C=0.5cm{j_!j^*X\ar[r]^{\ \ \ f}&X\ar[r]&i_*i^*X\ar[r]&0} \ \ \text{and} \ \ \xymatrix@C=0.5cm{0\ar[r]& i_*i^!X\ar[r]&X\ar[r]^{g\ \ \ }&j_*j^*X}
\]  
with $\Ker{f}$ and $\Coker{g}$ lying in $i_*\B$ (see, for example, \cite{FP} and  \cite{Psaroud:homolrecol} for details).

A useful tool to produce recollements of module categories or of derived module categories is the concept of \textbf{ring epimorphsim}, i.e. an epimorphism in the category of (unital) rings. It is well-known (see \cite{GdP,GeLe}) that ring epimorphisms with domain $R$ (up to the natural notion of equivalence) are in bijection with \textbf{bireflective} subcategories of $\Mod{R}$, i.e. full subcategories of $\Mod{R}$ whose inclusion functor admits both left and right adjoints. In order for the (exact) restriction of scalars functor to induce a fully faithful functor on the derived level, however, one needs to require more from a ring epimorphism.

\begin{thm}\textnormal{\cite[Theorem 4.4]{GeLe}\cite[Section 4]{NS}}
\label{hom ring epi}
Let $A$ be a ring and $f\colon A\lxr B$ a ring homomorphism. The following are equivalent.
\begin{enumerate}
\item The derived functor $f_*\colon \mathsf{D}(B)\lxr \mathsf{D}(A)$ is fully faithful.
\item The map $f$ is a ring epimorphism and, for any $i>0$, we have $\Tor_i^A(B,B)=0$.
\item For any $i\geq 0$, and any $B$-modules $M$ and $N$, we have $\Ext^i_B(M,N)=\Ext^i_A(M,N)$.
\end{enumerate}
\end{thm}
We say that a ring epimorphism is \textbf{homological} if the above equivalent conditions are satisfied. Note that while ring epimorphisms do not always give rise to recollements of module categories (see \cite{PsaroudVitoria}), a homological ring epimorphism always gives rise to a recollement of triangulated categories (see \cite{NS}). 

\begin{exam}\label{examrecol}
Let $A$ be a ring, $e$ an idempotent element of $A$ and $f\colon A\lxr A/AeA$ the associated ring epimorphism. There is
a recollement of $\Mod{A}$, as in the diagram below, which is said to be \textbf{induced by the idempotent element $e$}. 
\[
\xymatrix@C=0.5cm{
\Mod{A/AeA} \ar[rrr]^{f_*} &&& \Mod{A} \ar[rrr]^{\Hom_A(eA,-) \ } \ar
@/_1.5pc/[lll]_{-\otimes_AA/AeA}  \ar
 @/^1.5pc/[lll]^{\Hom_A(A/AeA,-)} &&& \Mod{eAe}.
\ar @/_1.5pc/[lll]_{-\otimes_{eAe}eA} \ar
 @/^1.5pc/[lll]^{\Hom_{eAe}(Ae,-)}
 } 
 \]
Moreover, if $\Tor_i^A(A/AeA,A/AeA)=0$ for all $i>0$ (i.e. $f$ is a homological ring epimorphism), it follows from \cite{CPSmemoirs} that there is a recollement of triangulated categories$\colon$
 \begin{equation}
 \label{stratifyingrecol}
\xymatrix@C=0.5cm{
\mathsf{D}(A/AeA) \ar[rrr]^{f_*} &&& \mathsf{D}(A) \ar[rrr]^{\mathbb{R}\Hom_A(eA,-) \ } \ar
@/_1.5pc/[lll]_{-\otimes^{\mathbb{L}}_AA/AeA}  \ar
 @/^1.5pc/[lll]^{\mathbb{R}\Hom_A(A/AeA,-)} &&& \mathsf{D}(eAe).
\ar @/_1.5pc/[lll]_{-\otimes^{\mathbb{L}}_{eAe}eA} \ar
 @/^1.5pc/[lll]^{\mathbb{R}\Hom_{eAe}(Ae,-)}
 } 
 \end{equation}  
\end{exam}

\begin{defn}
\label{defnstratifyingrecollement}
We say that a recollement of a derived module category by derived module categories is {\bf stratifying} if it is of the form (\ref{stratifyingrecol}).
\end{defn}

We study recollements up to the following notion of equivalence (see also \cite{PsaroudVitoria}). Throughout the paper, a diagram of functors is said to be commutative if it commutes up to natural equivalence of functors.

\begin{defn}
\label{def eq rec}
Two recollements of triangulated categories $\textsf{R}_{\mathsf{tr}}(\U,\T,\V)$ and $\textsf{R}_{\mathsf{tr}}(\U^\prime,\T^\prime,\V^\prime)$ are \textbf{equivalent}, if there are triangle equivalence functors  $F\colon \T\lxr \T^\prime$ and $G\colon \V\lxr \V^\prime$ such that  the diagram below commutes, i.e. there is a natural equivalence of functors $G j^*\cong j^{*\prime}F$.
\[
\xymatrix{
  \T \ar[d]_{F}^{\cong} \ar[r]^{j^*} & \V \ar[d]^{G}_{\cong}     \\
  \T^\prime    \ar[r]^{j^{*\prime}} & \V^\prime                  } 
\]
An equivalence of recollements of abelian categories is defined analogously.
\end{defn}

Note that the commutativity (up to natural equivalences) of the above diagram is equivalent to the commutativity of the six diagrams associated to the six functors of the recollements (see \cite[Lemma $4.2$]{PsaroudVitoria} for more details on the abelian case; the triangulated case is analogous, see \cite{PS}).

\subsection{Glueing}
\label{subBBDconstruction}
Example \ref{examrecol} shows how sometimes it is possible to build a recollement of triangulated categories from a recollement of abelian categories. In this subsection we recall from \cite{BBD} a procedure in the opposite direction, using t-structures. 

\begin{thm}\textnormal{\cite{BBD}}
\label{glueing rec heart}
Let $\textsf{R}_{\mathsf{tr}}(\U,\T,\V)$ be a recollement of triangulated categories of the form (\ref{rec}). 
Suppose that $\mathbb{U}=(\mathbb{U}^{\leq 0}, \mathbb{U}^{\geq 0})$ and $\mathbb{V}=(\mathbb{V}^{\leq 0},\mathbb{V}^{\geq 0})$ are t-structures in $\U$ and $\V$, respectively. Then there is a $t$-structure $\mathbb{T}=(\mathbb{T}^{\leq 0},\mathbb{T}^{\geq 0})$ in $\T$ defined by
\[
\mathbb{T}^{\leq 0}=\big\{X\in \T \ | \ j^*(X)\in \mathbb{V}^{\leq 0} \ \ \text{and} \ \ i^*(X)\in \mathbb{U}^{\leq 0} \big\}, \ \ \ 
\mathbb{T}^{\geq 0}=\big\{X\in \T \ | \ j^*(X)\in \mathbb{V}^{\geq 0} \ \ \text{and} \ \ i^!(X)\in \mathbb{U}^{\geq 0} \big\}.
\]
Convesely, given a t-structure $\mathbb{T}=(\mathbb{T}^{\leq 0},\mathbb{T}^{\geq 0})$ in $\T$, $\mathbb{T}$ is obtained as above from t-structures $\mathbb{U}$ and $\mathbb{V}$ in $\U$ and $\V$, respectively, if and only if $j_!j^*\mathbb{T}^{\leq 0}\subseteq \mathbb{T}^{\leq 0}$. In that case, $\mathbb{U}$ and $\mathbb{V}$ are uniquely determined by
$\mathbb{U}=(i^*\mathbb{T}^{\leq 0},i^!\mathbb{T}^{\geq 0})$ and $\mathbb{V}=(j^*\mathbb{T}^{\leq 0},j^*\mathbb{T}^{\geq 0}),$ the functors $i_*$ and $j^*$ are t-exact and their left (respectively, right) adjoints are right (respectively, left) t-exact. Moreover, the recollement of triangulated categories $\textsf{R}_{\mathsf{tr}}(\U,\T,\V)$ induces a recollement of abelian categories of the corresponding hearts $\textsf{R}_{\mathsf{ab}}(\Hcal(\mathbb{U}),\Hcal(\mathbb{T}),\Hcal(\mathbb{V}))$.
\end{thm}

We explain how to build the recollement $\textsf{R}_{\mathsf{ab}}(\Hcal(\mathbb{U}),\Hcal(\mathbb{T}),\Hcal(\mathbb{V}))$ from $\textsf{R}_{\mathsf{tr}}(\U,\T,\V)$, as stated in the theorem. Consider the cohomological functors $\textsf{H}^0_{\mathbb{U}}\colon \U\lxr \mathcal{H}(\mathbb{U})$, $\textsf{H}^0_{\mathbb{T}}\colon \T\lxr \mathcal{H}(\mathbb{T})$, $\textsf{H}^0_{\mathbb{V}}\colon \V\lxr \mathcal{H}(\mathbb{V})$, and the full embeddings $\varepsilon_{\mathbb{U}}\colon \mathcal{H}(\mathbb{U})\lxr \U$, $\varepsilon_{\mathbb{T}}\colon \mathcal{H}(\mathbb{T})\lxr \T$, $\varepsilon_{\mathbb{V}}\colon \mathcal{H}(\mathbb{V})\lxr \V$ associated with the t-structures $\mathbb{U}$, $\mathbb{T}$ and $\mathbb{V}$, respectively. Then the recollement $\textsf{R}_{\mathsf{ab}}(\Hcal(\mathbb{U}),\Hcal(\mathbb{T}),\Hcal(\mathbb{V}))$ is given by
\[
\xymatrix@C=0.5cm{
\mathcal{H}(\mathbb{U}) \ar[rrr]^{I_*} &&& \mathcal{H}(\mathbb{T}) \ar[rrr]^{J^*}  \ar @/_1.5pc/[lll]_{I^*}  \ar
 @/^1.5pc/[lll]_{I^!} &&& \mathcal{H}(\mathbb{V})
\ar @/_1.5pc/[lll]_{J_!} \ar
 @/^1.5pc/[lll]_{J_*}
 }
\]
where the functors are defined as follows:
\[
{\begin{array}{rlrl}
I^*\hspace{-0.25cm}&=\textsf{H}^{0}_{\mathbb{U}} \circ i^*\circ \epsilon_{\mathbb{T}}       \hspace{1.0cm}  & J_*\hspace{-0.25cm}&=\textsf{H}^{0}_{\mathbb{T}}\circ j_!\circ\epsilon_{\mathbb{V}}\\\noalign{\vspace{0.2cm}}
I_*\hspace{-0.25cm}&=\textsf{H}^{0}_{\mathbb{T}}\circ i_*\circ \epsilon_{\mathbb{U}}  \hspace{1.0cm}   & J^*\hspace{-0.25cm}&=\textsf{H}^{0}_{\mathbb{V}}\circ j^*\circ\epsilon_{\mathbb{T}} \\\noalign{\vspace{0.2cm}}
I^!\hspace{-0.25cm}&=\textsf{H}^{0}_{\mathbb{U}} \circ i^!\circ \epsilon_{\mathbb{T}}  \hspace{1.0cm}   & J_*\hspace{-0.25cm}&=\textsf{H}^{0}_{\mathbb{T}}\circ j_*\circ\epsilon_{\mathbb{V}.}
\end{array}}
\]
In other words, we have the following diagram:
\begin{equation}\nonumber
\xymatrix@C=0.5cm{ \U \ar[ddd]_{\textsf{H}^{0}_{\U}} \ar[rrr]^{i_*} &&& \T \ar[ddd]_{\textsf{H}^{0}_{\T}} \ar[rrr]^{j^*}  \ar @/_1.5pc/[lll]_{i^*}  \ar @/^1.5pc/[lll]^{i^!} &&& \V \ar[ddd]_{\textsf{H}^{0}_{\V}}
\ar @/_1.5pc/[lll]_{j_!} \ar
 @/^1.5pc/[lll]^{j_*} &&&&  \textsf{R}_{\mathsf{tr}}(\U,\T,\V) \ar[ddd]_{\textsf{H}^0} \\ 
 &&& &&& &&&& \\
 &&& &&& &&&& \\
\mathcal{H}(\mathbb{U}) \ar @/_1.5pc/[uuu]_{\varepsilon_{\U}} \ar[rrr]^{I_*} &&& \mathcal{H}(\mathbb{T}) \ar @/_1.5pc/[uuu]_{\varepsilon_{\T}} \ar[rrr]^{J^*}  \ar @/_1.5pc/[lll]_{I^*}  \ar
 @/^1.5pc/[lll]^{I^!} &&& \mathcal{H}(\mathbb{V}) \ar @/_1.5pc/[uuu]_{\varepsilon_{\V}}
\ar @/_1.5pc/[lll]_{J_!} \ar
 @/^1.5pc/[lll]^{J_*} &&&& \textsf{R}_{\mathsf{ab}}(\mathcal{H}(\mathbb{U}),\mathcal{H}(\mathbb{T}),\mathcal{H}(\mathbb{V})) \ar @/_1.5pc/[uuu]_{\varepsilon}
 } 
\end{equation}
where the functors in the lower recollement are defined using the vertical functors as described above. A recollement obtained in this way will be called \textbf{a recollement of hearts}. 
\section{Realisation functors}\label{sec 3}
Given a t-structure in a triangulated category $\T$ with heart $\Hcal$, it is natural to ask how does the (bounded) derived category of $\Hcal$ compare with $\T$. In \cite{BBD}, a functor between these two categories is built under some assumptions on the category $\T$ and the t-structure on it: the realisation functor. We consider the more general approach from \cite{B} that allows the construction of the realisation functor for t-structures in any triangulated category that admits an f-category over it. We survey this construction and the relevant associated notions in some detail, as we will need a deeper understanding of this functor later in this text. We refer to \cite{KV2} for a different approach to the realisation functor. Here is an informal overview of this section.\\

\textbf{Subsection 3.1: f-categories}
\begin{itemize}
\item We review the definition of a filtered enhancement of a triangulated category, motivated by the example of filtered derived categories (Example \ref{exam filtered derived cat}).

\item We recall how to lift a t-structure from a triangulated category to a filtered enhancement (Proposition \ref{f-cat properties}). This is an important step towards the construction of realisation functors.

\item We show that given a triangulated category with a filtered enhancement, there are compatible filtered enhancements on any thick subcategory and on any Verdier quotient (Proposition \ref{f-categories over exact seq}). This results is useful for the use of realisation functors in the context of recollements (Section 6).
\end{itemize}

\textbf{Subsection 3.2: Realisation functors and their properties}
\begin{itemize}
\item We recall with some detail the construction and basic properties of realisation functors (Theorem \ref{real}) which will be used throughout the paper.

\item We discuss a result of Beilinson concerning commutative diagrams of functors involving realisation functors (Theorem \ref{lift implies diagram}). This is particularly relevant for the study of the \textit{shape} of derived equivalences in Subsection 5.2 and for the construction of equivalences of recollements in Section~$6$.
\end{itemize}

\textbf{Subsection 3.3: Examples of realisation functors}
\begin{itemize}
\item We show that the realisation functor of the standard t-structure in a derived category (built with respect to the filtered derived category) is essentially an identity functor (Proposition~\ref{standard identity}).

\item Given a triangle equivalence $\phi$ between the bounded derived categories of two abelian categories, we show that $\phi$ is naturally equivalent to a realisation functor composed with the derived functor of an exact equivalence of abelian categories (Proposition~\ref{everything is real}).
\end{itemize}

\subsection{f-categories}
The key idea for constructing a realisation functor in \cite{BBD} was that of using the so-called filtered derived category. However, as it was observed in \cite{B}, we only need the abstract properties of such categories for this construction - giving rise to the notion of an f-category. For a detailed survey on f-categories, we refer to \cite{Schn}.

\begin{defn}
\label{defnfcat}
An \textbf{f-category} is a triangulated category $\Xcal$ endowed with an autoequivalence $s\colon \Xcal\lxr \Xcal$ (called \textit{f-shift}), a natural transformation $\alpha\colon \iden_{\X}\lxr s$ and two full triangulated subcategories $\Xcal({\geq 0})$ and $\Xcal(\leq 0)$  such that, for $\Xcal(\geq n):=s^n\Xcal(\geq 0)$ and $\Xcal(\leq n):=s^n\Xcal({\leq 0})$, we have
\begin{enumerate}
\item $\Hom_\Xcal(\Xcal(\geq 1),\Xcal(\leq 0))=0$;
\item For every object $X$ in $\Xcal$, there are $Y$ in $\Xcal(\geq 1)$, $Z$ in $\Xcal(\leq 0)$ and a triangle in $\X$:
$$\xymatrix{Y\ar[r]& X\ar[r]& Z\ar[r]& Y[1];}$$
\item $\mathsf{Ob}\ \Xcal=\cup_{n\in\mathbb{Z}}\mathsf{Ob}\ \Xcal(\geq n)=\cup_{n\in\mathbb{Z}}\mathsf{Ob}\ \Xcal(\leq n)$;
\item $\Xcal(\geq 1)\subseteq \Xcal(\geq 0)$ and $\Xcal(\leq -1)\subseteq \Xcal(\leq 0)$;
\item $\alpha_X=s(\alpha_{s^{-1}X})$ for all $X$ in $\Xcal$;
\item For any $X\in\Xcal(\geq 1)$ and $Y\in\Xcal(\leq 0)$, $\Hom_{\Xcal}(\alpha_Y,X)$ and $\Hom_{\Xcal}(Y,\alpha_X)$ are isomorphisms.
\end{enumerate}
Given a triangulated category $\T$, an \textbf{f-category over} $\T$ (or an \textbf{f-enhancement of} $\T$) is a pair $(\X,\theta)$ where $\X$ is an f-category and $\theta\colon \T\lxr \X(\geq 0)\cap\X(\leq 0)$ is an equivalence of triangulated categories.
\end{defn}

For an f-category $\X$, we write the whole data as $(\X, \X({\geq 0}), \X({\leq 0}), s\colon \X\stackrel{\cong}{\lxr} \X, \alpha\colon \iden_{\X}\lxr s)$, although we write just $\X$ when the remaining data is fixed. Let $(\Xcal,\theta)$ denote an f-category over $\T$. Note that for any $n$ in $\mathbb{Z}$, the pair $(\Xcal(\geq n+1),\Xcal(\leq n))$ is a \textbf{stable} t-structure, i.e. a t-structure whose aisle is a triangulated subcategory. In particular, there are truncation functors $\sigma_{\geq n}\colon \X\lxr \X(\geq n)$ and $\sigma_{\leq n}\colon \X\lxr \X(\leq n)$, which are triangle functors.  We define the following further triangle functors
$$\mathsf{gr}^n_\X:=\theta^{-1}s^{-n}\sigma_{\leq n}\sigma_{\geq n}\colon \X\lxr \T.$$

There are standard f-categories over a large class of triangulated categories: (bounded) derived categories of abelian categories. These are the so-called \textbf{filtered derived categories}. In the following example we build the filtered derived category of the unbounded derived category of an abelian category. The bounded setting is entirely analogous.

\begin{exam}\cite{Ill}\label{exam filtered derived cat}
Given an abelian category $\A$, consider the (additive) category $\mathsf{C}\mathsf{F}(\A)$ of complexes of objects in $\A$ endowed with a finite decreasing filtration. The objects in $\mathsf{C}\mathsf{F}(\A)$ are, thus, pairs $(X,F)$, where $X$ lies in the category of complexes $\mathsf{C}(\A)$ and $F$ is a filtration of $X$ as follows:
\[
X=F_aX\supseteq F_{a+1}X\supseteq F_{a+2}X\supseteq \cdots \supseteq F_{b-1}X\supset F_bX=0,
\]
with $a\leq b$ integers. The morphisms in $\mathsf{C}\mathsf{F}(\A)$ are morphisms of complexes respecting the filtration, i.e. given two filtered complexes $(X,F)$ and $(Y,G)$ a morphism $f\colon (X,F)\lxr (Y,G)$ in $\mathsf{C}\mathsf{F}(\A)$ is a sequence of chain maps $(\dots, F_{a}f,F_{a+1}f,\dots, F_bf,\dots)$, with $F_if\colon F_iX\lxr G_iY$ compatible with the inclusion maps of the filtrations $F$ and $G$. There are natural functors $\mathsf{gr}^i\colon \mathsf{C}\mathsf{F}(\A)\lxr \mathsf{C}(\A)$ associating to a filtered complex $(X,F)$ its $i$-th graded component $\textsf{gr}_F^i(X):=F_iX/F_{i+1}X$. A morphism $\phi$ in $\mathsf{C}\mathsf{F}(\A)$ is said to be a \textbf{filtered quasi-isomorphism} if $F_if$ is a quasi-isomorphism, for all $i$ in $\mathbb{Z}$ (see \cite{Ill} for further equivalent definitions of filtered quasi-isomorphisms). The \textbf{filtered derived category} $\mathsf{D}\mathsf{F}(\A)$ of an abelian category $\A$ is the localisation of $\mathsf{C}\mathsf{F}(\A)$ on filtered quasi-isomorphisms. Moreover, one can also define the filtered homotopy category $\mathsf{KF}(\A)$ of $\A$, where the objects are the same as in $\mathsf{CF}(\A)$ but the morphisms are equivalences classes of morphisms in $\mathsf{C}\mathsf{F}(\A)$ modulo filtered homotopy (two morphisms $f, f'\colon (X,F)\lxr (Y,G)$ in $\mathsf{C}\mathsf{F}(\A)$ are homotopic, if there is a homotopy from $f$ to $f'$ compatible with the filtrations). The filtered derived category can also obtained as the localisation of $\mathsf{KF}(\A)$ on filtered quasi-isomorphisms.

Note that there is a natural fully faithful functor $\xi\colon \mathsf{D}(\A)\lxr \mathsf{D}\mathsf{F}(\A)$ sending an object $X$ in $\mathsf{D}(\A)$ to the pair $(X,0)$, where $0$ indicates the trivial filtration $X=F_0(X)\supseteq F_1X=0$. The filtered derived category comes naturally equipped with an autoequivalence $s\colon \mathsf{D}\mathsf{F}(\A)\lxr \mathsf{D}\mathsf{F}(\A)$ corresponding to the shift on filtration, i.e. $s(X,F)=(X,G)$, where $G_iX=F_{i-1}X$. Also, there is a natural transformation $\alpha\colon \iden_{\mathsf{D}\mathsf{F}(\A)}\lxr s$ such that for any object $X$, $\alpha_X$ is induced by the inclusion maps of $F_{i+1}X$ into $F_iX$, for all $i$ in $\mathbb{Z}$. Finally, consider $\mathsf{D}\mathsf{F}(\A)(\geq 0)$ (respectively, $\mathsf{D}\mathsf{F}(\A)(\leq 0)$) to be the full subcategory of $\mathsf{D}\mathsf{F}(\A)$ spanned by filtered complexes whose non-trivial graded components are in non-negative (respectively, non-positive) degrees. It follows that $\mathsf{D}\mathsf{F}(\A)$ is an f-category over $\mathsf{D}(\A)$.
There is also a natural functor $\omega\colon \mathsf{D}\mathsf{F}(\A)\lxr\mathsf{D}(\A)$, the forgetful functor. Note also that the functors $\mathsf{gr}^i$ defined at the level of complexes induce triangle functors from $\mathsf{D}\mathsf{F}(\A)\lxr \mathsf{D}(\A)$ and, as the notation suggests, these are the analogues in this setting of the $\mathsf{gr}$-functors defined in the general context of f-categories.
\end{exam}

We summarise some useful facts about f-categories over a triangulated category. 

\begin{prop}\textnormal{\cite{B,Schn}}
\label{f-cat properties}
If $(\X,\theta)$ is an f-category over a triangulated category $\T$, then:
\begin{enumerate}
\item there is an exact functor $\omega\colon \X\lxr \T$, unique up to natural equivalence, such that:
\begin{itemize}
\item its restriction to $\Xcal(\geq 0)$ is right adjoint to the functor $\T\lxr \Xcal(\geq 0)$ induced by $\theta$;
\item its restriction to $\Xcal(\leq 0)$ is left adjoint to the functor $\T\lxr \Xcal(\leq 0)$ induced by $\theta$;
\item for any $X$ in $\Xcal$, the map $\omega(\alpha_X)\colon \omega X\lxr \omega sX$ is an isomorphism;
\item for any $X$ in $\Xcal(\leq 0)$ and $Y$ in $\Xcal(\geq 0)$, $\omega$ induces an isomorphism between $\Hom_\Xcal(X,Y)$ and $\Hom_\T(\omega X,\omega Y)$.
\end{itemize}
\item given a t-structure $\mathbb{T}=(\mathbb{T}^{\leq 0},\mathbb{T}^{\geq 0})$ in $\T$, there is a unique t-structure $\mathbb{X}=(\mathbb{X}^{\leq 0},\mathbb{X}^{\geq 0})$ in $\X$ such that $\theta$ is a t-exact functor and $s\mathbb{X}^{\leq 0}\subseteq \mathbb{X}^{\leq -1}$. Moreover, the t-structure $\mathbb{X}$ can be described by
$$\mathbb{X}^{\leq 0}=\{X\in\X \ | \ \textsf{gr}^n_\X(X)\in\mathbb{T}^{\leq n}\}\ \ \text{and}\ \ \mathbb{X}^{\geq 0}=\{X\in\X \ | \ \textsf{gr}^n_\X(X)\in\mathbb{T}^{\geq n}\}$$
and the heart $\Hcal(\mathbb{X})$ is equivalent to the category $\Ccal^b(\Hcal(\mathbb{T}))$ of chain complexes over $\Hcal(\mathbb{T})$.
\end{enumerate}
\end{prop}

\begin{rem}\label{equivalence f-heart}
We point out how to build the functor yielding an equivalence between $\Hcal(\mathbb{X})$ and $\mathsf{C}^b(\Hcal(\mathbb{T}))$ (see \cite{B}). There is a cohomological functor which we will, abusively, denote by $\mathsf{H}^0_\mathbb{X}\colon \X\lxr \mathsf{C}^b(\Hcal(\mathbb{T}))$. Given an object $X$ in $\X$, define $\mathsf{H}^0_\mathbb{X}(X)$ to be a complex whose $i$-th component is $\mathsf{H}^i_\mathbb{T}(\mathsf{gr}^i_\X X)$. In order to define the differential $d^i\colon \mathsf{H}^i_\mathbb{T}(\mathsf{gr}^i_\X X)\lxr \mathsf{H}^{i+1}_\mathbb{T}(\mathsf{gr}^{i+1}_\X X)$ consider the triangle in $\T$ given by
$$\xymatrix{\omega\sigma_{\leq i+1}\sigma_{\geq i+1}X\ar[r]& \omega\sigma_{\leq i+1}\sigma_{\geq i}X\ar[r]&\omega\sigma_{\leq i}\sigma_{\geq i}X\ar[r]^{\bar{d}^i\ \ \ \ \  } &(\omega\sigma_{\leq i+1}\sigma_{\geq i+1}X)[1]}.$$
Now, the properties of $\omega$ listed above insure that, for any $n$ in $\mathbb{Z}$, we have $\omega\sigma_{\leq n}\sigma_{\geq n}X\cong \mathsf{gr}_\X^{n}X$. It then follows that we can define $d^i$ as $\mathsf{H}_\mathbb{T}^i(\bar{d}^i)$, for any $i$ in $\mathbb{Z}$. This defines the functor $\mathsf{H}_\mathbb{X}^0$ and it can be seen that this functor yields an exact equivalence between $\Hcal(\mathbb{X})$ and $\mathsf{C}^b(\Hcal(\mathbb{T}))$ as wanted (see \cite{BBD} for a proof in the case of filtered derived categories; the statement for f-categories is available without proof in \cite{B} since the arguments are analogous). The abuse of notation here is justified by the fact that indeed $\mathsf{H}_\mathbb{X}^0$ can be regarded as a cohomological functor associated with the t-structure $\mathbb{X}$ in $\X$ (see \cite{B}).
\end{rem}

The functor $\omega$ in the proposition will be called the \textbf{f-forgetful functor}, as motivated by the actual forgetful functor in the case of filtered derived categories. Note that the existence of f-enhancements of triangulated categories is not a priori guaranteed - although conjectured (see \cite{Bondarko}).

We turn now to functors between f-categories.

\begin{defn}
\label{defnffunctor}
Let $(\X, \X({\geq 0}), \X({\leq 0}), s, \alpha)$ and $(\Y, \Y({\geq 0}), \Y({\leq 0}), t, \beta)$ be two $f$-categories. An {\bf f-functor} between the f-categories $\X$ and $\Y$ is a triangle functor $F\colon \X\lxr \Y$ such that:
\begin{enumerate}
\item $F(\X({\geq 0}))\subseteq \Y({\geq 0})$ and $F(\X({\leq 0}))\subseteq \Y({\leq 0})$;

\item $F s\cong t F$ and $F(\alpha_X)=\beta_{F(X)}$, for all $X$ in $\X$.
\end{enumerate}
The f-categories $\X$ and $\Y$ are {\bf equivalent}, if there is an f-functor $F\colon \X\lxr \Z$ which is a triangle equivalence. If $(\X,\theta)$ and $(\Y,\eta)$ are f-categories over triangulated categories $\T$ and $\U$, respectively and $\phi\colon \T\lxr \U$ is a triangle functor, we say that $\phi$ \textbf{lifts} to the f-categories $(\X,\theta)$ and $(\Y,\eta)$ if there is an f-functor $\Phi\colon \X\lxr \Y$ such that $\Phi\theta\cong \eta\phi$. When the f-categories are fixed, we will just say that $\phi$ \textbf{admits an f-lifting}.
\end{defn}

\begin{exam}\label{eq f-lift}
Given triangulated categories $\T$ and $\U$ and a triangle equivalence $\phi\colon \T\lxr\U$, if  $(\X,\theta)$ is an f-category over $\T$, then $(\X,\theta\phi^{-1})$ is an f-category over $\U$ and, hence, $\iden_\X$ is an f-lifting of $\phi$.
 \end{exam}

We will show that given an exact sequence of triangulated categories 
$$0\lxr \U\lxr \T \lxr \T/\U\lxr 0$$
and an $f$-category over $\T$, there are induced f-categories over the thick subcategory $\U$ and over the Verdier quotient $\T/\U$, improving on \cite[Proposition 2.7]{W}. 
We say that a thick subcategory $\Y$ of an f-category $\X$ is an \textbf{f-subcategory} if $\Y$ is an f-category with the induced f-structure, i.e. $\Y(\leq 0)=\X(\leq 0)\cap \Y$, $\Y(\geq 0)=\X(\geq 0)\cap \Y$ and both $s\colon\Y\lxr \Y$ and $\alpha\colon\iden_\Y\rightarrow s$ are the restrictions of the corresponding functor/natural transformation in $\X$.

\begin{lem}\label{quotient f-cat}
Let $\Y$ be a thick f-subcategory of an f-category $\X$. Then the Verdier quotient $\X/\Y$ has a natural f-category structure induced by the one in $\X$.
\end{lem}
\begin{proof}
Since $\Y$ is a thick subcategory, there is a short exact sequence of triangulated categories 
\begin{equation}
\label{exseqfcat}
\xymatrix{
  0 \ar[r] & \Y \ar[r]^{j \ } & \X \ar[r]^{p \ \ } &  \X/\Y \ar[r] &  0}.
\end{equation}
Set $\Z=\X/\Y$ and consider the full triangulated subcategories $\Z(\geq 0)=p(\X(\geq 0))$ and $\Z(\leq 0)=p(\X(\leq 0))$. From the sequence (\ref{exseqfcat}), one can easily observe that the functor $s\colon \X\lxr \X$ induces a functor $s_{\Z}\colon \Z\lxr \Z$. It is obvious that $s_\Z$ is essentially surjective. It is also fully faithful (and hence an autoequivalence of $\Z$) since its action on morphisms can be described by applying $s$ to a roof in $\Z$. Non-trivial roofs will remain non-trivial due to the fact that $s$ restricts as an autoequivalence to $\Y$ (since $\Y$ is an f-subcategory). We also obtain an induced natural transformation $\gamma\colon \iden_{\Z}\lxr s_{\Z}$ of triangulated functors. Indeed, using the calculus of fractions available for morphisms in $\Z$, it is easy to check that defining $\gamma_{p(X)}:=p(\alpha_X)$, for any $X$ in $\X$, yields the wanted natural transformation.

Since $(\Y(\geq{1}), \Y(\leq 0))$ is a stable t-structure in $\Y$, from \cite[Proposition 1.5]{IKM} we get that $(\Z(\geq 1), \Z(\leq 0))$ is a stable t-structure in $\Z$, thus proving the properties (i) and (ii) of Definition~\ref{defnfcat}. Clearly we have $\mathsf{Ob}\ \Z=\cup_{n\in\mathbb{Z}}\mathsf{Ob}\ \Z({\geq n})=\cup_{n\in\mathbb{Z}}\mathsf{Ob}\ \Z({\leq n})$ since the same relation holds for objects in $\X$. Also, we have $$\Z({\geq 1})=s_{\Z}(p(\X({\geq 0})))=p(s(\X({\geq 0})))=p(\X({\geq 1}))\subseteq p(\X({\geq 0}))=\Z({\geq 0})$$ and similarly we get that $\Z({\leq -1})\subseteq \Z({\leq 0})$. For condition (v), observe that, for any $X$ in $\X$, we have
$$\gamma_{p(X)}=p(\alpha_X)=ps(\alpha_{s^{-1}X})=s_\Z(\gamma_{ps^{-1}(X)})=s_\Z(\gamma_{s_\Z^{-1}(p(X))}).$$ 

It remains to prove condition (vi) of Definition~\ref{defnfcat}. Let $X$ be an object in $\X(\geq 1)$ and $Y$ an object in $\X(\leq 0)$. We will show that $\Hom_\Z(\gamma_{p(Y)},p(X))$ is an isomorphism. The proof that $\Hom_\Z(p(Y),\gamma_{p(X)})$ is an isomorphism is analogous. Let $f\colon p(Y)\rightarrow p(X)$ be a morphism in $\Z$, represented by a roof of the form 
\begin{equation}\nonumber
\xymatrix{ & K \ar[ld]_{c}\ar[rd]^d\\ Y &&X}
\end{equation}
with $\cone(c)$ in $\Y$. We want to show that $f$ admits a unique preimage under the map $\Hom_\Z(\gamma_{p(Y)},p(X))$. In order to do that, we first compose $f$ with an isomorphism and write the composition as a roof in a convenient way that will allow us to use axiom (vi) of the f-category $\X$.
\begin{itemize}
\item[(1)] Applying the triangle functor $\sigma_{\geq 1}$ to the triangle induced by the map $c$ and using the fact that $Y$ lies in $\X(\leq 0)$, it follows that  $\sigma_{\geq 1}K$ lies in $\Y$. Consider the composition of the natural map $\sigma_{\geq 1}K\lxr K$ with $d\colon K\lxr X$, and denote by $g\colon X\lxr \overline {X}$ its mapping cone. It is clear that $p(g)$ is an isomorphism in $\Z$. Moreover, using the fact that $\Hom_\Z(s_\Z p(Y),p(g))$ and $\Hom_\Z(p(Y),p(g))$ are isomorphisms, the map $f$ admits a unique preimage under $\Hom_\Z(\gamma_{p(Y)},p(X))$ if and only if $\bar{f}:=p(g)\circ f$ admits a unique preimage under $\Hom_\Z(\gamma_{p(Y)},p(\overline{X}))$.
\item[(2)] Since $Y=\sigma_{\leq 0}Y$, $c$ factors through the natural map $K\lxr \sigma_{\leq 0}K$ and the cone of $\sigma_{\leq 0}c$ is precisely $\sigma_{\leq 0}\mathsf{cone}(c)$, which lies in $\Y$ since $\Y$ is an f-subcategory of $\X$. By construction of $g$, also $g\circ d$ factors through the natural map $K\lxr \sigma_{\leq 0}K$ (via a map $h\colon \sigma_{\leq 0}K\lxr \overline{X}$). It then can be checked that $\bar{f}\colon p(Y)\lxr p(\overline{X})$ is equivalent to the following roof.
$$\xymatrix{ & \sigma_{\leq 0}K \ar[ld]_{\sigma_{\leq 0}c}\ar[rd]^h\\ Y &&\overline{X}}$$
\end{itemize}
Now, from axiom (vi) of the f-category $\X$ there is a unique morphism $m\colon s(\sigma_{\leq 0}K)\lxr \overline{X}$ such that $h=m\circ\alpha_{\sigma_{\leq 0}K}$ and then, the morphism $s_{\Z}p(Y)\lxr p(\overline{X})$ of $\Z$ represented by the fraction $m\circ s(\sigma_{\leq 0}c)^{-1}$ is a preimage of $\overline{f}$ by the map $\Hom_\Z(\gamma_{p(Y)},p(\overline{X}))$. Using the uniqueness of $m$ and the description of morphisms in $\Z$ as roofs, it easily follows that this preimage is unique, finishing the proof.
\end{proof}

We are now ready to show how an f-category over a triangulated category $\T$ induces f-categories over thick subcategories or over Verdier quotients. Recall that given two triangulated subcategories $\U$ and $\V$ of a triangulated category $\T$, one denotes by $\U\ast \V$ the subcategory of $\T$ formed by the objects $T$ such that there are objects $U$ in $\U$, $V$ in $\V$ and a triangle
$$U\lxr T\lxr V\lxr U[1].$$
It is not always true that $\U\ast\V$ is a triangulated subcategory of $\T$. In fact, $\U\ast\V$ is triangulated if and only if $\Hom_{\T/\U\cap\V}(\pi(\U),\pi(\V))=0$, where $\pi\colon \T\lxr \T/(\U\cap\V)$ is the quotient functor (\cite[Theorem A]{JK}).

\begin{prop}\label{f-categories over exact seq}
Let $0\lxr \U\stackrel{i}{\lxr} \T \stackrel{q}{\lxr} \T/\U\lxr 0$ be an exact sequence of triangulated categories and $(\X,\theta)$ be an f-category over $\T$. Then $(\X,\theta)$ induces  f-category structures over $\U$ and $\T/\U$.
\end{prop}
\begin{proof}
We assume without loss of generality (by Example \ref{eq f-lift})  that $\U$ is a thick subcategory of $\T$, $i$ is the embedding functor (and we identify $\U$ with $i(\U)$) and $q$ is the natural projection to the Verdier quotient. 

Recall from \cite[Proposition 2.2]{W} that the f-category $(\X,\theta)$ induces an f-category $\Y$ over $\U$ defined by   $\Y= \{X\in \X \ | \ \mathsf{gr}^n_\X(X)\in \U \ \text{for all} \ n\in \mathbb{Z} \}$. It is easy to check that $\Y$ is a thick subcategory of $\X$ (since $\U$ is a thick subcategory of $\T$) and that it is an f-subcategory of $\X$. Moreover, $(\Y,\theta_{|\U})$ is an f-category over $\U$ since we have a commutative diagram as follows, where the vertical arrows are the natural inclusions.
\begin{equation}\nonumber
\xymatrix{
\U \ar @/^2.0pc/[rr]_{\theta_{|\U}} \ar[r]_{\cong \ \ \ \ \  \ \ \ \ \ } \ar[d] & \Y(\geq 0)\cap\Y(\leq 0) \ar[d] \ar@{^{(}->}[r] \ar[d] & \Y \ar@{_{(}->}[d] \\ 
\T \ar @/_2.0pc/[rr]^{\theta} \ar[r]^{\cong \ \ \ \ \  \ \ \ \ \ } & \X(\geq 0)\cap\X(\leq 0) \ar@{^{(}->}[r] & \X }
\end{equation}

By Lemma \ref{quotient f-cat}, $\Z:=\X/\Y$ is an f-category. It remains to show that $\X/\Y$ is indeed an f-category over $\T/\U$. It is clear that $\theta$ induces a functor $\bar{\theta}\colon \T/\U\lxr \X/\Y$ and a commutative diagram between exact sequences of triangulated categories as follows.
\begin{equation}
\label{exactcomdiagramfcat}
\xymatrix{
  0 \ar[r] & \U \ar[r]^{i \ } \ar[d]^{\theta|_{\U}} & \T \ar[r]^{q \ \ } \ar[d]^{\theta} &  \T/\U \ar[r] \ar[d]^{\bar{\theta}} &  0 \\
  0 \ar[r] & \Y \ar[r]^{j \ } & \X \ar[r]^{p \ \ } &  \X/\Y \ar[r] &  0. } 
\end{equation}
The functors $\theta$ and $\theta|_{\U}$ are obviously fully faithful and we claim that so is $\bar{\theta}$. Using \cite[Lemma $4.7.1$]{Krauselocalization}, it is enough to show that any map $f\colon Y\lxr \theta(T)$ in $\X$, with $Y$ in $\Y$ and $T$ in $\T$, factors through an object of 
\[
\theta(\U)=\Y(\geq 0)\cap \Y(\leq 0)=\X(\geq 0)\cap \X(\leq 0)\cap \Y=\theta(\T)\cap \Y.
\]
Since $\theta(T)$ lies in $\X(\leq 0)\cap\X(\geq 0)$, we may assume without loss of generality that $Y$ lies in $\Y(\leq 0)$ (this follows from the triangle in Definition \ref{defnfcat}(ii) and the orthogonality relation between $\Ycal(\geq 1)$ and $\Xcal(\leq 0)$). By Proposition \ref{f-cat properties}(i), the restriction of the f-forgetful functor $\omega$ to $\X(\leq 0)$ is left adjoint to the inclusion of $\T$ (by $\theta$) in $\Xcal(\leq 0)$. Thus, considering the unit of the adjunction, $\eta_Y\colon Y\lxr \theta\omega(Y)$, we get that $f=\theta(\tilde{f})\circ \eta_Y$, where $\tilde{f}\colon \omega(Y)\lxr T$ is the map corresponding to $f$ under the isomorphism $\Hom_\X(Y,\theta(T))\cong \Hom_\T(\omega(Y),T)$. We now show that $\omega(Y)$ lies in $\U$. We do this by induction on the \textit{graded length} of $Y$, i.e. on $n\geq 0$ such that $Y$ lies in $\X(\geq -n)\cap \X(\leq 0)$ (such $n$ always exists by Definition \ref{defnfcat}(iii) and by our assumption that $Y$ lies in $\X(\leq 0)$). If $Y$ lies in $\Xcal(\geq 0)$, then $Y\cong \theta\mathsf{gr}^0Y$ and $\omega(Y)=\theta^{-1}(Y)=\mathsf{gr}^0(Y)$ lies in $\U$, by definition of $\Y$. Suppose now that the result is valid for objects with graded length $n-1$ and let $Y$ lie in $\Ycal(\geq -n)$. Then there is a triangle
$$Y=\sigma_{\geq -n}Y\lxr \sigma_{\geq -n+1}Y\lxr s^n\theta\mathsf{gr^{-n+1}}(Y)\lxr (\sigma_{\geq -n}Y)[1].$$
Applying the triangle functor $\omega$ to it, since $\omega s^n\theta\mathsf{gr^{-n+1}}(Y)\cong \mathsf{gr^{-n+1}}(Y)$ lies in $\U$ and, by induction hypothesis, so does $\omega\sigma_{\geq -n+1}Y$, it follows that $\omega(Y)$ lies in $\U$, as wanted.

It remains to show that the essential image of $\bar{\theta}$ is $\Z(\geq 0)\cap\Z(\leq 0)$. By the commutativity of (\ref{exactcomdiagramfcat}) and since $\bar{\theta}$ is fully faithful, it suffices to prove that $\Z(\geq 0)\cap\Z(\leq 0)=p(\X(\geq 0)\cap\X(\leq 0))$. We first show that $\Z(\geq 0)\cong (\X(\geq 0)\ast \Y)/\Y=p(\X(\geq 0)\ast \Y)$ (dual arguments also show that $\Z(\leq 0)\cong (\Y\ast \X(\leq 0))/\Y=p(\Y\ast \X(\leq 0))$). Note that $\X(\geq 0)\cap \Y=\Y(\geq 0)$. Using \cite[Theorem A]{JK}, it is then enough to prove that $\Hom_{\X/\Y(\geq 0)}(\pi(\X(\geq 0)),\pi(\Y))=0$, where $\pi\colon \X\lxr \X/\Y(\geq 0)$ is the Verdier quotient functor. Consider an element in $\Hom_{\X/\Y(\geq 0)}(\pi(X),\pi(Y))$, with $X$ in $\X(\geq 0)$ and $Y$ in $\Y$, represented by a roof
\begin{equation}\nonumber
\xymatrix{ & K \ar[ld]_{g}\ar[rd]^f\\ X &&Y}
\end{equation}
with $f\in \Hom_\X(K,Y)$, $g\in \Hom_\X(K,X)$ and $\mathsf{cone(g)}$ in $\Y(\geq 0)$. Since both $X$ and $\mathsf{cone(g)}$ lie in $\X(\geq 0)$, it follows that also $K$ lies in $\X(\geq 0)$. Hence, it follows that $f$ must factor through the natural map $\sigma_{\geq 0}Y\lxr Y$. In particular, the roof is equivalent to the zero morphism in $\X/\Y(\geq 0)$, as wanted. Thus we may rewrite the intersection $\Z(\leq 0)\cap \Z(\geq 0)$ as follows
$$\Z(\leq 0)\cap \Z(\geq 0)=p(\Y\ast \X(\leq 0))\cap p(\X(\geq 0)\ast \Y)=p((\Y\ast \X(\leq 0))\cap(\X(\geq 0)\ast\Y)),$$
 where the last equality follows from \cite[Lemma 2.4(i)(a)]{JK}. Finally, we finish the proof by showing that the last term above equals $p(\X(\leq 0)\cap \X(\geq 0))$. Observe first that given an object $X$ in $\X(\geq 0)\ast \Y$, it follows that $\sigma_{\leq -1}X$ lies in $\Y$. In fact, by the assumption on $X$ there is a triangle
 $$\xymatrix{X^\prime\ar[r]^f &X\ar[r]& Y\ar[r]&X^\prime[1]}$$
 with $X^\prime$ in $\Xcal(\geq 0)$ and $Y$ in $\Y$. Applying to it the triangle functor $\sigma_{\leq -1}$, since $\sigma_{\leq -1}X^\prime=0$ we get that $\sigma_{\leq -1}X\cong \sigma_{\leq -1}Y$. Since $\Y$ is an f-subcategory of $\X$, it then follows that $\sigma_{\leq -1}X$ lies in $\Y$. Analogously, one can show that given $X$ in $\Y\ast\X(\leq 0)$, $\sigma_{\geq 1}X$ lies in $\Y$. Hence for any object $X$ in the intersection $(\Y\ast \X(\leq 0))\cap(\X(\geq 0)\ast\Y)$, both $\sigma_{\leq -1}X$ and $\sigma_{\geq 1}X$ lie in $\Y$ and, thus, $p(X)\cong p(\theta\mathsf{gr}^0(X))$, showing that $p((\Y\ast \X(\leq 0))\cap(\X(\geq 0)\ast\Y))$ is contained in $p(\X(\leq 0)\cap \X(\geq 0))$. Since the other inclusion is trivial (because $\X(\leq 0)\cap \X(\geq 0)\subset (\Y\ast \X(\leq 0))\cap(\X(\geq 0)\ast\Y)$), we have finished the proof. 
\end{proof}

\begin{cor}
\label{ffadmitsflifting}
Let $0\lxr \U\stackrel{i}{\lxr} \T \stackrel{q}{\lxr} \T/\U\lxr 0$ be an exact sequence of triangulated categories and $(\X,\theta)$ be an f-category over $\T$. Then the functors $i$ and $q$ admit f-liftings for suitable choices of f-categories over $\U$ and $\T/\U$.
\end{cor}
\begin{proof}
By construction of the f-categories $(\Y,\theta_{|\U})$ and $(\Z,\bar{\theta})$ in the proof of Proposition~\ref{f-categories over exact seq}, we get that $i$ and $j$ admit f-lifitings, namely the f-functors $j$ and $p$ in the diagram $(\ref{exactcomdiagramfcat})$ above. 
\end{proof}

\begin{rem}\label{fcat7}
Recall from \cite[Definition 2.4]{Neeman0} that a morphism of triangles $(a,b,c)$ is said to be \textbf{middling good} if it can be completed to a commutative $3\times 3$ diagram (in the sense of \cite[Lemma 1.1.11]{BBD})  in which all rows and all columns are triangles. In \cite{Schn}, the following extra axiom for f-categories is proposed. 
\begin{enumerate}
\item[(fcat7)] For any morphism $f\colon X\lxr Y$ in $\Xcal$, the triple
$\Delta_f:=(\alpha_{\sigma_{\geq 1(Y)}}\circ \sigma_{\geq 1}(f),\alpha_Y\circ f, \alpha_{\sigma_{\leq 0(Y)}}\circ \sigma_{\leq 0}(f))$ 
is a middling good morphism of triangles.
\end{enumerate}
Although we have not made use of this axiom so far, we will implicitly make use of it in the next subsection (see Remark \ref{fcat7 again}(ii) and the Appendix to this paper). At this point it is worth noting the following facts.
\begin{itemize}
\item Filtered derived categories, as discussed in Example \ref{exam filtered derived cat} satisfy axiom (vii). This is proved in \cite[Lemma 7.4]{Schn}.
\item In the context of Lemma \ref{quotient f-cat}, if $\Xcal$ satisfies axiom (fcat7), then so does $\Zcal:=\Xcal/\Ycal$. Note that if a morphism of triangles is middling good, then so is its composition with an isomorphism of triangles. Hence, given $f\colon X\lxr Y$ in $\Zcal$, we may assume without loss of generality that $f=p(f^\prime)$, where $p$ is the Verdier quotient functor. Since $\Xcal$ satisfies axiom (vii), the morphism $\Delta_{f^\prime}$ is middling good. As a consequence $p(\Delta_{f^{\prime}})$ is middling good. Finally, using the t-exactness of $p$ and the compatibility of $p$ with the the functor $s$ and the natural transformations $\alpha$ and $\gamma$, we conclude that $p(\Delta_{f^{\prime}})=\Delta_{p(f^{\prime})}$ (see the proof of Lemma \ref{quotient f-cat} for notation and details).
\end{itemize}
From now on, we will assume f-categories to satisfy this new axiom.
\end{rem}

\subsection{Realisation functors and their properties}
We are now ready to build realisation functors. Let $\T$ be a triangulated category and $\mathbb{T}=(\mathbb{T}^{\leq 0},\mathbb{T}^{\geq 0})$ a t-structure in $\T$. Suppose that $(\X,\theta)$ is an f-category over $\T$. By Proposition \ref{f-cat properties}, there is a t-structure $\mathbb{X}$ in $\X$ defined by
\[
\mathbb{X}^{\leq 0}=\{X\in\X \ | \ \textsf{gr}^n_\X(X)\in\mathbb{T}^{\leq n} \ \text{for all} \ n\in \mathbb{Z} \}\ \ \text{and}\ \ \mathbb{X}^{\geq 0}=\{X\in\X \ | \ \textsf{gr}^n_\X(X)\in\mathbb{T}^{\geq n} \ \text{for all} \ n\in \mathbb{Z} \};
\]
whose heart $\Hcal(\mathbb{X})$ is equivalent to $\mathsf{C}^b(\Hcal(\mathbb{T}))$. Let $G\colon \mathsf{C}^b(\Hcal(\mathbb{T}))\lxr \Hcal(\mathbb{X})$ denote the inverse of that equivalence (described in Remark \ref{equivalence f-heart}). The realisation functor of $\mathbb{T}$ with respect to the f-category $(\X,\theta)$ is then obtained as follows.  Moreover, we collect the first properties of the functor $\mathsf{real}_\mathbb{T}^\X$. These were first proved in \cite{BBD} and restated in a more general setting in \cite{W}. Our statement differs to that in \cite{W} only on the class of triangulated categories we consider - see Remark \ref{Wildeshaus remark}(i). We also include a sketch for the proof of the theorem for the convenience of the reader.

\begin{thm}\textnormal{\cite[Section 3.1]{BBD}\cite[Appendix]{B}\cite[Theorem 1.1]{W}}\label{real}
Let $\T$ be a triangulated category and $(\X,\theta)$ an f-category over $\T$ with f-forgetful functor $\omega\colon \X\lxr \T$. Let $\mathbb{T}$ be a t-structure in $\T$, $\mathbb{X}$ the corresponding induced t-structure on $\X$ and $G\colon \mathsf{C}^b(\Hcal(\mathbb{T}))\lxr \Hcal(\mathbb{X})$ the exact equivalence of abelian categories described in the paragraph above. Let $Q\colon \mathsf{C}^b(\Hcal(\mathbb{T}))\lxr \mathsf{D}^b(\Hcal(\mathbb{T}))$ be the natural localisation functor. Then there is a unique functor, called \textbf{the realisation functor of $\mathbb{T}$ with respect to $(\X,\theta)$}, $\mathsf{real}_\mathbb{T}^\X\colon \mathsf{D}^b(\Hcal(\mathbb{T}))\lxr \T$ such that the following diagram naturally commutes
$$\xymatrix{\mathsf{C}^b(\Hcal(\mathbb{T}))\ar[r]^Q\ar[d]_G&\mathsf{D}^b(\Hcal(\mathbb{T}))\ar[ddl]^{\mathsf{real}_\mathbb{T}^\X}\\ \Hcal(\mathbb{X})\ar[d]_{\omega_{|\Hcal(\mathbb{X})}}\\ \T}$$
Furthermore, $\mathsf{real}_\mathbb{T}^\X\colon \mathsf{D}^b(\Hcal(\mathbb{T}))\longrightarrow \T$ is a triangle functor and satisfies the following properties.
\begin{enumerate}
\item $\mathsf{H}^i_0\cong \mathsf{H}^i_\mathbb{T}\circ \mathsf{real}_\mathbb{T}^\X$, for all $i\in\mathbb{Z}$. In particular, $\mathsf{real}_\mathbb{T}^\X$ acts as the identity functor on $\Hcal(\mathbb{T})$ and it is t-exact with respect to the standard t-structure in $\mathsf{D}^b(\Hcal(\mathbb{T}))$ and $\mathbb{T}$ in $\T$.

\item The functor $\mathsf{real}_\mathbb{T}^\X$ induces isomorphisms $\Hom_{\mathsf{D}^b(\Hcal(\mathbb{T}))}(X,Y[n])\cong \Hom_{\T}(X,Y[n])$ for any $X$ and $Y$ in $\Hcal(\mathbb{T})$ and for $n\leq 1$.

\item The following statements are equivalent. 
\begin{itemize}
\item[(a)] The functor $\mathsf{real}_\mathbb{T}^\X$ is fully faithful;

\item[(b)]The functor $\mathsf{real}_\mathbb{T}^\X$ induces isomorphisms $\Hom_{\mathsf{D}^b(\Hcal(\mathbb{T}))}(X,Y[n])\cong \Hom_{\T}(X,Y[n])$, for all $n\geq 2$ and for all $X$ and $Y$ in $\Hcal(\mathbb{T})$.

\item[(c)] (\textsf{Ef}) Given objects $X$ and $Y$ in $\Hcal(\mathbb{T})$, $n\geq 2$ and a morphism $f\colon X\lxr Y[n]$ in $\T$, there is an object $Z$ in $\Hcal(\mathbb{T})$ and an epimorphism in $\Hcal(\mathbb{T})$, $g\colon Z\lxr X$, such that $fg=0$.

\item[(d)] (\textsf{CoEf}): Given objects $X$ and $Y$ in $\Hcal(\mathbb{T})$, $n\geq 2$ and a morphism $f\colon X\longrightarrow Y[n]$ in $\T$, there is an object $Z$ in $\Hcal(\mathbb{T})$ and a monomorphism in $\Hcal(\mathbb{T})$, $g\colon Y\longrightarrow Z$, such that $g[n]f=0$.
\end{itemize}
\item The essential image of $\mathsf{real}_\mathbb{T}^\X$ is contained in $$\T^{b(\mathbb{T})}:=\bigcup\limits_{n,m\in\mathbb{Z}}\mathbb{T}^{\leq n}\cap\mathbb{T}^{\geq m}$$
and it coincides with it whenever $\mathsf{real}_\mathbb{T}^\X$ is fully faithful.
\end{enumerate}
\end{thm}
\begin{proof}
For the existence of the functor $\mathsf{real}_\mathbb{T}^\X$ and for property (i) we refer to \cite[Appendix A.5, A.6]{B}. Property (ii) is clear for $n\leq 0$. For $n=1$, the statement boils down to show that the Yoneda extension group $\Ext^1_{\Hcal(\mathbb{T})}(X,Y)$ in $\Hcal(\mathbb{T})$ coincides with $\Hom_\T(X,Y[1])$, which is a well-known fact about hearts of t-structures (see \cite[Remark 3.1.17(ii)]{BBD}). 

For part (iii), we prove in detail the most delicate implication: (c)$\Longrightarrow$ (b) (the implication (d)$\Longrightarrow (b)$ is analogous). Assuming the condition (\textsf{Ef}) we show the isomorphism in (b) by induction on $n\geq 1$ (for $n=1$, (b) holds by statement (ii)). Let $X$ and $Y$ be objects in $\Hcal(\mathbb{T})$,  $n\geq 2$ and consider the induced map $\mathsf{real}_{\mathbb{T}}^{\X}(X,Y[n])\colon\Hom_{\mathsf{D}^b(\Hcal(\mathbb{T}))}(X,Y[n])\lxr \Hom_\T(X,Y[n])$. First we prove surjectivity. Let $f\colon X\lxr Y[n]$ be a map in $\T$. By assumption there is an epimorphism $g\colon Z\lxr X$ in $\Hcal(\mathbb{T})$ such that $f\circ g=0$ in $\T$. Let $K$ be the kernel of $g$ in $\Hcal(\mathbb{T})$ and consider the triangle induced by $g$ in $\T$:
\[
\xymatrix{K\ar[r]& Z\ar[r]^g& X\ar[r]^{h \ \ } & K[1].}
\]
Since $g\circ f=0$, there is $u\colon K[1]\lxr Y[n]$ such that  $u\circ h=f$. By induction hypothesis, there is $u^\prime$ in $\Hom_{\mathsf{D}^b(\Hcal(\mathbb{T}))}(K[1],Y[i])$ such that $\mathsf{real}^\X_\mathbb{T}(u^\prime)=u$. By (ii), there is also $h^\prime$ in $\Hom_{\mathsf{D}^b(\Hcal(\mathbb{T}))}(X,K[1])$ such that $\mathsf{real}^\X_\mathbb{T}(h')=h$ and, thus, we have that $\mathsf{real}^\X_\mathbb{T}(u^\prime\circ h')=u\circ h=f$. To prove the injectivity of $\mathsf{real}_{\mathbb{T}}^{\X}(X,Y[n])$, let $\alpha$ be an element in $\Hom_{\mathsf{D}^b(\Hcal(\mathbb{T}))}(X,Y[n])$ such that $\mathsf{real}_\mathbb{T}^\X(\alpha)=0$. Since $\alpha$ can be thought of as an Yoneda extension (of degree $n$) between $X$ and $Y$, it represents an exact sequence in $\Hcal(\mathbb{T})$ of the form
\[
\xymatrix{0\ar[r]&Y\ar[r]&A_1\ar[r]&A_2\ar[r]&\cdots\ar[r]&A_n\ar[r]^\beta& X\ar[r]&0.}
\]
It is easy to check that $\alpha\circ\beta=0$. Let $L$ denote the kernel of $\beta$ and consider the triangle in $\mathsf{D}^b(\Hcal(\mathbb{T}))$:
\[
\xymatrix{A_n\ar[r]^\beta&X\ar[r]^{\gamma\ }& L[1]\ar[r]^{\rho[1] \ }&A_n[1].}
\]
Then there is $\delta\colon L[1]\lxr Y[n]$ such that $\alpha=\delta\circ\gamma$. Now, we have that $0=\real^\X_\mathbb{T}(\alpha)=\real^\X_\mathbb{T}(\delta)\circ\mathsf{real}_{\mathbb{T}}^{\X}(\gamma)$ and, hence, there is a map $\epsilon\colon A_n[1]\lxr Y[n]$ in $\T$ such that $\real^\X_\mathbb{T}(\delta)=\epsilon\circ\real^\X_\mathbb{T}(\rho[1])$. Since $\real_\mathbb{T}^\X(A_i[1],Y[i])$ is surjective, by induction hypothesis, it follows that there is a map  $\epsilon^\prime\colon A_i[1]\lxr Y[i]$ in $\mathsf{D}^b(\Hcal(\mathbb{T}))$ such that $\real^\X_\mathbb{T}(\epsilon^\prime)=\epsilon$. Now, $\mathsf{real}_\mathbb{T}^\X(\epsilon^\prime\circ \rho[1])=\mathsf{real}_\mathbb{T}^\X(\delta)$ which, since $\mathsf{real}_\mathbb{T}^\X(L[1],Y[i])$ is injective by induction hypothesis, implies that $\epsilon^\prime\circ \rho[1]=\delta$. Hence, we have $\alpha=\delta\circ \gamma=\epsilon^\prime\circ \rho[1]\circ \gamma=0$ since $\rho[1]\circ \gamma=0$. 

With regards to the remaining implications of (iii): it is clear that (a) implies (b) and the converse follows from a \textit{d\'evissage} argument (see our proof of Theorem~\ref{prophomolemb} for such an argument). The fact that (a) implies (c) or (d) can easily be observed from properties of Yoneda extensions (in fact the proof that (a) implies (c) is essentially contained in the above paragraph). Finally, regarding property (iv), since $\mathsf{real}_\mathbb{T}^\X$ is t-exact and the standard t-structure in $\mathsf{D}^b(\Hcal(\mathbb{T}))$ is bounded, it follows easily that the $\Image(\mathsf{real}_\mathbb{T}^\X)$ is contained in $\T^{b(\mathbb{T})}$. If $\mathsf{real}_\mathbb{T}^\X$ is fully faithful, it can be proved by induction on $l=b-a$ (with $b\geq a$) that $\mathbb{T}^{\leq b}\cap\mathbb{T}^{\geq a}$ lies in its essential image (see the proof of \cite[Proposition 3.1.16]{BBD}).
\end{proof}

\begin{rem}\label{Wildeshaus remark}\label{fcat7 again}
\begin{enumerate}
\item In \cite{W} the above result is stated for triangulated subcategories of the derived category of an abelian category. This restriction is only to ensure that the triangulated category considered admits an \textit{f-enhancement}, but the arguments carry through in the more general setting here considered.
\item Although the claim that $\mathsf{real}^\Xcal_\mathbb{T}$ is a triangle functor is implicit in \cite[Appendix]{B}, the only proof of this fact that the authors are aware of is due to E. Cabezuelo Fern\'andez and O. Schnürer and it makes use of the axiom (vii) in Remark \ref{fcat7} (see the Appendix to this paper for details)
\end{enumerate}
\end{rem}

One further property of realisation functors that is particularly useful in our applications is that they behave \textit{naturally} in certain contexts. The following theorem was presented in \cite{B} without proof. 

\begin{thm}\textnormal{\cite[Lemma A7.1]{B}}
\label{lift implies diagram}
Let $(\X,\theta)$ and $(\Y,\eta)$ be f-categories over triangulated categories $\T$ and $\U$, respectively, and let $\phi:\T\lxr \U$ be a triangle functor. Suppose that $\phi$ is t-exact with respect to t-structures $\mathbb{T}$ and $\mathbb{U}$ in $\T$ and $\U$, respectively. If $\phi$ admits an f-lifting, then there is a commutative diagram
$$\xymatrix{\mathsf{D}^b(\Hcal(\mathbb{T}))\ar[r]^{\mathsf{D}^b(\phi^0)}\ar[d]_{\mathsf{real}_\mathbb{T}^\X}&\mathsf{D}^b(\Hcal(\mathbb{U}))\ar[d]^{\mathsf{real}_\mathbb{U}^\Y}\\ \T\ar[r]^\phi&\U}$$
where $\mathsf{D}^b(\phi^0)$ is the derived functor of the exact functor $\phi^0\colon \Hcal(\mathbb{T})\lxr \Hcal(\mathbb{U})$, induced by $\phi_{|\Hcal(\mathbb{T})}$. 
\end{thm}
\begin{proof}
Let $\Phi\colon \X\lxr\Y$  be an f-lifting of $\phi$ and let $\mathbb{X}$ and $\mathbb{Y}$ be the t-structures in $\X$ and $\Y$ compatible with $\mathbb{T}$ and $\mathbb{U}$ as in Proposition \ref{f-cat properties}(ii). Since $\Phi$ is an f-lifting of $\phi$ (and, in particular, an f-functor),  we have
\begin{equation}\nonumber
\begin{array}{lll}
\phi(\mathsf{gr}_{\X}^n(X)) &=& \phi(\theta^{-1}s^{-n}\sigma_{\leq n}\sigma_{\geq n}(X)) \nonumber \\
 &\cong& \eta^{-1}\Phi|_{\X(\geq 0)\cap \X(\leq 0)}s^{-n}\sigma_{\leq n}\sigma_{\geq n}(X) \nonumber \\
 &\cong& \eta^{-1} t^{-n}\Phi(\sigma_{\leq n}\sigma_{\geq n}(X)) \nonumber \\ 
 &\cong& \eta^{-1} t^{-n}\sigma_{\leq n}\sigma_{\geq n}\Phi(X) \nonumber \\
 &\cong& \mathsf{gr}_{\Y}^n(\Phi(X)). \nonumber
\end{array}
\end{equation}
Using this fact, since $\phi$ is t-exact (with respect to $\mathbb{T}$ and $\mathbb{U}$), we get that $\Phi$ is also t-exact (with respect to $\mathbb{X}$ and $\mathbb{Y}$), inducing an exact functor $\Phi^0\colon \Hcal(\mathbb{X})\lxr \Hcal(\mathbb{Y})$. This yields the following diagram of functors.
\begin{equation}\nonumber
\xymatrix{
& \ar @{} [dr] |{(3)}\\
\mathsf{D}^b(\Hcal(\mathbb{T}))  \ar @/^3.0pc/[rrr] \ar[rddd]_{\mathsf{real}_\mathbb{T}^\X} & \mathsf{C}^b(\Hcal(\mathbb{T}))\ar @{} [ld] |{(1)} \ar @{} [dr] |{(5)} \ar[l] \ar[r]_{\mathsf{C}^b(\phi^0)}  & \mathsf{C}^b(\Hcal(\mathbb{U}))\ar @{} [dr] |{(2)} \ar[r] & \mathsf{D}^b(\Hcal(\mathbb{U})) \ar[lddd]^{\mathsf{real}_\mathbb{U}^{\Y}} \\
& \mathcal{H}(\mathbb{X}) \ar[r]_{\Phi^0} \ar[u]^{\cong} \ar[d] \ar @{} [dr] |{(4)} & \mathcal{H}(\mathbb{Y}) \ar[d] \ar[u]_{\cong} & \\
 & \X \ar[r]_{\Phi} \ar[d]^{\omega_{\X}} \ar @{} [dr] |{(6)} & \Y \ar[d]_{\omega_{\Y}} & & \\
 & \T \ar[r]_\phi & \U & }
\end{equation}
In order to prove the theorem, it is enough to check the commutativity of all the internal diagrams. Diagrams (1), (2), (3) and (4) commute by construction of the functors involved. Diagram (5) commutes using again above property that $\phi\mathsf{gr}_\X^n\cong \mathsf{gr}^n_\Y\Phi$ and Remark \ref{equivalence f-heart}.

Finally, let us prove in detail the commutativity of diagram (6). We first show that (6) naturally commutes for objects in $\X(\leq 0)$, i.e. that there is a natural equivalence $\mu^0\colon \omega_{\Y|\Y(\leq 0)}\Phi_{|\X(\leq 0)}\lxr \phi\omega_{\X|\X(\leq 0)}$. To simplify the notation we write the upperscript $\leq 0$ to denote the restriction of the functors to $\X(\leq 0)$ or to $\Y(\leq 0)$ (depending on the domain of the functor). Consider the unit of the adjunction $(\omega_{\X}^{\leq 0},\theta)$ and denote it by $\delta\colon \iden_{\X(\leq 0)}\lxr \theta\omega_{\X}^{\leq 0}$. We define $\mu^0$ as the following natural composition
\[
\xymatrix{\omega_\Y^{\leq 0}\Phi^{\leq 0}\ar[rr]^{\omega_\Y^{\leq 0}\Phi^{\leq 0}(\delta)\ \ \ }&&\omega_\Y^{\leq 0}\Phi^{\leq 0}\theta\omega_\X^{\leq 0}\ar[r]^{\ \cong}& \omega_\Y^{\leq 0}\eta\phi\omega_\X^{\leq 0}\ar[r]^{\ \ \cong}&\phi\omega_\X^{\leq 0}.}
\]
Note that we use the fact that $\Phi$ is an f-lifting of $\phi$ in order to get a natural equivalence $\Phi^{\leq 0}\theta\cong \eta\phi$. We also use that $\omega_\Y^{\leq 0}$ is a left inverse to $\eta$. Consider now the subcategory of $\X(\leq 0)$ formed by all the objects $X$ such that $\mu^0_X$ is an isomorphism. It is easy to see that this subcategory is triangulated and it contains $s^n(\X(\leq 0)\cap\X(\geq 0))=s^n\theta(\T)$ for any $n\leq 0$. Since every object in $\X(\leq 0)$ can be obtained as a finite extension of such objects, it follows that $\mu^0$ is a natural equivalence. Now, given $n\geq 0$, we define natural transformations $\mu^n\colon \omega_{\Y|\Y(\leq n)}\Phi_{|\X(\leq n)}\lxr \phi\omega_{\X|\X(\leq n)}$. If $X$ lies in $\X(\leq n)$ then $s^{-n}X$ lies in $\X(\leq 0)$ and, hence, we may define $\mu^n:=\mu^0s^{-n}$. It is clear that $\mu^n$ is also a natural equivalence. Thus, we have a family of natural equivalences $(\mu^n)_{n\geq 0}$. It follows from Definition \ref{defnfcat}(iii) and (iv) that in order to define a natural equivalence $\mu\colon \omega_{\Y}\Phi\lxr \phi\omega_{\X}$, it is enough to show that, for any $m>n$ and for any $X$ in $\X(\leq n)$, $\mu^n_X$ is naturally isomorphic to $\mu^m_X$ (note that $\mu^m=\mu^ns^{n-m}$). Let $X$ lie in $\X(\leq n)$ and consider the map 
$$\mu^n_{s^{n-m}X}\colon \omega_\Y\Phi(s^{n-m}X)\lxr \phi\omega_\X(s^{n-m}X).$$
Using property (iii) in Definition \ref{defnfcat}, we define the following composition of natural morphisms
$$\alpha^{[n,m]}_X:\xymatrix{s^{n-m}X\ar[rr]^{\alpha_{s^{n-m}X}\ }&& s^{n-m+1}X\ar[rr]^{\ \alpha_{s^{n-m+1}X}}&&\cdots\ar[rr]^{\alpha_{s^{-1}X}}&& X.}$$
By the naturality of $\alpha$, we get a commutative diagram as follows
$$\xymatrix{\omega_\Y\Phi(s^{n-m}X)\ar[rr]^{\mu^ns^{n-m}}\ar[d]_{\omega_\Y\Phi(\alpha^{[n,m]}_X)}&&\phi\omega_\X(s^{n-m}X)\ar[d]^{\phi\omega_\X(\alpha^{[n,m]}_X)}\\ \omega_\Y\Phi(X)\ar[rr]^{\mu_n}&&\phi\omega_\X(X)}$$
Since $\Phi$ is an f-functor, $\Phi(s^{n-m}X)$ is naturally isomorphic to $t^{n-1}\Phi(X)$ and $\Phi(\alpha^{[n,m]}_X)=\beta^{[n,m]}_{\Phi(X)}$. By Proposition \ref{f-cat properties}(i) it follows that the vertical maps are isomorphisms, as wanted.
\end{proof}

\subsection{Examples of realisation functors}
We begin with the simplest realisation functor: the one associated to the standard t-structure in a derived category and with respect to the filtered derived category.

\begin{prop}\label{standard identity}
Let $\Acal$ be an abelian category. Then the realisation functor associated to the standard t-structure in $\mathsf{D}(\A)$ with respect to the filtered derived category of $\A$  is naturally equivalent to the inclusion functor of $\mathsf{D}^b(\A)$ in $\mathsf{D}(\A)$.
\end{prop}
\begin{proof}
Going through the construction of the realisation functor for the standard t-structure, we show that it acts as the identity both on objects and on morphisms. From Proposition \ref{f-cat properties} (ii), there is a t-structure in the filtered derived category $\mathsf{D}\mathsf{F}(\A)$ compatible with the standard t-structure in $\mathsf{D}(\A)$, whose heart is
\[
\A\mathsf{F}:=\{(X,F)\in \mathsf{DF}(\A) \ | \ \textsf{gr}^i_F(X)\in \A[-i] \ \forall i \in \mathbb{Z}\}
\]
and the equivalence between $\A\mathsf{F}$ and $\mathsf{C}^b(\A)$ (as in Remark \ref{equivalence f-heart})  is given by assigning to $(X,F)$ in $\A\mathsf{F}$ the complex $(\textsf{gr}_F^i(X),d^i)$, where $d^i\colon \textsf{gr}_F^i(X)\lxr \textsf{gr}_F^{i+1}(X)$ is defined by the canonical triangle in $\mathsf{D}^b(\A)$
\[
\xymatrix{\textsf{gr}^i_F(X)[-1]\ar[r]^{ \ d^i }& \textsf{gr}^{i+1}_F(X)\ar[r]& F_{i}X/F_{i+2}X\ar[r]& \textsf{gr}^i_F(X).}
\]
Note that this makes sense since the map $d^i$ in the above triangle is indeed a map in $\A[-i-1]$, by definition of $\A\mathsf{F}$. In order to compute the realisation functor, one needs to describe an inverse of this equivalence of abelian categories. Given a complex $Y=(Y^i,d^i)$ in $\mathsf{C}^b(\A)$ consider a filtration on $Y$ defined by the \textit{stupid truncations}, i.e. for any integer $n$ define 
\[
F_nY= \xymatrix{(\cdots \ar[r]& 0\ar[r]& Y^n\ar[r]^{d^n}& Y^{n+1}\ar[r]^{d^{n+1}}& Y^{n+2}\ar[r]^{d^{n+2}} & \cdots),}
\]
It is easy to see that, in fact, the object $(Y,F)$ belongs to $\A\mathsf{F}$. This assignment clearly gives rise to a functor $G\colon\mathsf{C}^b(\A)\lxr \A\mathsf{F}$ which can easily be checked to be the wanted inverse functor. Consider now the composition $\mathsf{C}^b(\A)\lxr \A\mathsf{F}\lxr \mathsf{DF}(\A)\lxr \mathsf{D}(\A)$ of $G$ and the forgetful functor $\omega\colon \mathsf{D}\mathsf{F}(\A)\lxr \mathsf{D}(\A)$. It is clear that this composition sends a complex $Y$ to itself as an object of the derived category - and similarly for morphisms. Hence, the realisation functor, being the universal functor induced by the localisation of $\mathsf{C}^b(\A)$ at the quasi-isomorphisms, is naturally equivalent to the inclusion functor of $\mathsf{D}^b(\A)$ in $\mathsf{D}(\A)$. 
\end{proof}

In the above proposition, if we restrict the codomain to $\mathsf{D}^b(\A)$, the realisation functor is then naturally equivalent to $\iden_{\mathsf{D}^b(\A)}$. Throughout the paper we will restrict the codomain of the realisation functor from unbounded derived categories to bounded ones whenever possible and convenient without further mention. 

A recurrent problem when dealing with realisation functors is that, as they are defined, their domain is a bounded (rather than unbounded) derived category. In order to also discuss functors defined in unbounded derived categories, we need the following notion.

\begin{defn}\label{defn restrict extend}
Let $\A$ and $\B$ be abelian categories. An equivalence $\Phi\colon \mathsf{D}(\A)\lxr \mathsf{D}(\B)$ is said to be \textbf{restrictable} if, by restriction, it induces an equivalence $\phi:\mathsf{D}^b(\A)\lxr \mathsf{D}^b(\B)$. In this case we also say that $\phi$ is \textbf{extendable}. In other words, the equivalences $\Phi$ and $\phi$ are, respectively, restrictable or extendable if there is a commutative diagram as follows, where the vertical arrows are the natural inclusions.
$$\xymatrix{\mathsf{D}^b(\A)\ar[r]^\phi\ar[d]&\mathsf{D}^b(\B)\ar[d]\\ \mathsf{D}(\A)\ar[r]^\Phi&\mathsf{D}(\B)}$$
\end{defn}

\begin{exam}\label{example restrict extend}
\begin{enumerate}
\item Let $\A$ and $\B$ be abelian categories and $\Phi\colon \mathsf{D}(\A)\lxr\mathsf{D}(\B)$ a triangle equivalence. Then $\Phi$ is restrictable if and only if $\mathsf{D}(\B)^{b(\mathbb{T})}=\mathsf{D}^b(\B)$, where $\mathbb{T}:=(\Phi(\mathbb{D}^{\leq 0}),\Phi(\mathbb{D}^{\geq 0}))$ (for further equivalent conditions on the t-structure, see Lemma \ref{eq bdd}). In fact, since $\Phi_{|\mathsf{D}^b(\A)}$ is fully faithful, the objects in $\Image{\Phi_{|\mathsf{D}^b(\A)}}$ are precisely those which can be obtained as finite extensions in $\mathsf{D}(\B)$ of shifts of objects in $\Hcal(\mathbb{T})=\Phi(\A)$ - and this is precisely the subcategory $\mathsf{D}(\B)^{b(\mathbb{T})}$. 
\item Equivalences of standard type between bounded derived categories of rings are extendable (\cite{Keller}).
\end{enumerate}
\end{exam}
 
The next proposition shows that equivalences of unbounded or bounded derived categories \textit{do not differ much} from suitably chosen realisation functors. By this we mean that the difference between a derived equivalence and our choice of realisation functor is a \textit{trivial} derived equivalence, i.e. the derived functor of an exact equivalence of abelian categories.

\begin{prop}\label{everything is real}
Let $\A$ and $\B$ be abelian categories. The following statements hold.
\begin{enumerate}
\item Let $\phi\colon \mathsf{D}^b(\A)\lxr\mathsf{D}^b(\B)$ be a triangle equivalence and let $\mathbb{T}$ be the t-structure $(\phi(\mathbb{D}^{\leq 0}_\A),\phi(\mathbb{D}^{\geq 0}_\A))$ in $\mathsf{D}^b(\B)$. Then there is an f-category $(\X,\theta)$ over $\mathsf{D}^b(\B)$ such that $\phi\cong \mathsf{real}_\mathbb{T}^\X\circ \mathsf{D}^b(\phi^0)$, for the exact equivalence of abelian categories $\phi^0\colon\A\lxr\phi(\A)$, induced by $\phi$.
\item Let $\Phi\colon \mathsf{D}(\A)\lxr\mathsf{D}(\B)$ be a triangle equivalence and let $\mathbb{T}$ be the t-structure $(\phi(\mathbb{D}^{\leq 0}_\A),\phi(\mathbb{D}^{\geq 0}_\A))$ in $\mathsf{D}(\B)$. Then there is an f-category $(\X,\theta)$ over $\mathsf{D}(\B)$ such that $\Phi^b:=\Phi_{|\mathsf{D}^b(\A)}\cong \mathsf{real}_\mathbb{T}^\X\circ \mathsf{D}^b(\Phi^0)$, for the exact equivalence of abelian categories $\Phi^0\colon\A\lxr\Phi(\A)$, induced by $\Phi$. In particular, if $\Image(\mathsf{real}_\mathbb{T}^\X)=\mathsf{D}(\B)^{b(\mathbb{T})}=\mathsf{D}^b(\B)$, then $\Phi$ is a restrictable equivalence.
\end{enumerate} 
In both cases, $\mathsf{real}_\mathbb{T}^\X$ is fully faithful, thus inducing an equivalence between $\mathsf{D}^b(\A)$ and its essential image.
\end{prop}
\begin{proof}
(i) Consider the filtered bounded derived category $\mathsf{DF}^b(\A)$ as an f-category over $\mathsf{D}^b(\A)$ and let $(\X,\theta)$ be an f-category over $\mathsf{D}^b(\B)$ defined by $\X=\mathsf{DF}^b(\A)$ and $\theta=\theta^\prime\phi^{-1}$, where $\theta^\prime$ is the natural inclusion of $\mathsf{D}^b(\A)$ in $\mathsf{DF}^b(\A)$ (see Example~\ref{eq f-lift}). This choice of f-categories over $\mathsf{D}^b(\A)$ and $\mathsf{D}^b(\B)$ guarantees, trivially, that the identity functor on $\mathsf{DF}^b(\A)$ is an f-lifting of $\phi$. Now, the functor $\phi$ is t-exact with respect to the standard t-structure in $\mathsf{D}^b(\A)$ and the t-structure $\mathbb{T}$ in $\mathsf{D}^b(\B)$ ($\mathbb{T}$ was chosen for this purpose) and, thus, $\Hcal(\mathbb{T})=\phi(\A)$. Then Theorem~\ref{lift implies diagram} shows that $\mathsf{real}^\X_\mathbb{T}\circ \mathsf{D}^b(\phi^0)\cong\phi\circ \mathsf{real}_{\mathbb{D}}^{\mathsf{DF}(\A)}$, where $\mathsf{real}_{\mathbb{D}}^{\mathsf{DF}(\A)}$ denotes the realisation of the standard t-structure in $\mathsf{D}^b(\A)$ with respect to the filtered derived category of $\A$. Since by Proposition \ref{standard identity}, $\mathsf{real}_{\mathbb{D}}^{\mathsf{DF}(\A)}$ is naturally equivalent to $\iden_{\mathsf{D}^b(\A)}$, we get that $\phi\cong \mathsf{real}^\X_\mathbb{T}\circ \mathsf{D}^b(\phi^0)$. 

(ii) Consider the f-category $(\X,\theta):=(\mathsf{DF}(\A),\theta^\prime\Phi^{-1})$ over $\mathsf{D}(\B)$, where now $\theta^\prime$ is the unbounded version of the functor stated in part (i). Then the first statement can be proved analogously to (i). The second statement, follows from Example \ref{example restrict extend}(i). 

Finally, note that in the above cases, since both $\phi$ and $\mathsf{D}^b(\phi^0)$ (respectively, $\Phi$ and $\mathsf{D}^b(\Phi^0)$) are fully faithful, then so is $\mathsf{real}_\mathbb{T}^\X$, finishing the proof.
\end{proof}

The above statement is not particularly surprising. Given a derived category, any equivalence with another derived category yields an obvious \textit{new} f-enhancement. The proposition translates this in terms of functors. The motto could be \textit{studying derived equivalence functors corresponds to studying f-enhancements}. If, however, we want to study realisation functors with respect to fixed f-categories (for example, filtered derived categories), the problem resides then on the f-lifting property, as we will see in Section 5.

\section{Silting and cosilting $t$-structures}
We will now discuss a class of t-structures arising from certain objects (called silting or cosilting) in a triangulated category. Within this class, it will be possible to characterise exactly which associated realisation functors yield derived equivalences (see Section~\ref{sectiondereq}). These t-structures have appeared in the literature in various incarnations (\cite{KV, AI, NSZ, Wei, AMV1}). In this section we provide a general definition which covers, up to our knowledge, all the examples of silting complexes appearing in the literature, including non-compact ones. Furthermore, we introduce the dual notion of cosilting. Later, we specify to tilting and cotilting objects, observing how they can provide derived equivalences even when they are not compact. 

In this section, $\T$ will denote a triangulated category. Given an object $X$ in $\T$, we will denote by $\Add(X)$ (respectively, $\Prod(X)$) the full subcategory of $\T$ consiting of all objects which are summands of a direct sum (respectively, of a direct product) of $X$. Note that without further assumptions, the category $\T$ might not admit arbitrary (set-indexed) coproducts or products of an object $X$. We will say that a triangulated category is \textbf{TR5} if it has set-indexed coproducts and \textbf{TR5*} if it has set-indexed products. Recall from \cite[Proposition 1.2.1]{Neeman} that, in a TR5 (respectively, TR5*)  triangulated category, coproducts (respectively, products) of triangles are again triangles. Given an object $X$ in $\T$ and an interval $I$ of integers, we consider the following orthogonal subcategories of $\T$
\[
X^{\perp_I}=\{Y\in\T \ | \ \Hom_\T(X,Y[i])=0, \forall i\in I\}, \ \ \ \ {}^{\perp_I}X=\{Y\in\T \ | \ \Hom_\T(Y,X[i])=0, \forall i\in I\}.
\]
If the interval $I$ is unbounded, we often replace it by symbols such as $>n$, $<n$, $\geq n$, $\leq n$, $\neq n$ (with $n\in\mathbb{Z}$) with the obvious associated meaning. 
We say that an object $X$ \textbf{generates} $\T$ if $X^{\perp_\mathbb{Z}}=0$ and it \textbf{cogenerates} $\T$ if ${}^{\perp_\mathbb{Z}}X=0$. Recall also that an object $X$ in a TR5 triangulated category is said to be \textbf{compact} if $\Hom_\T(X,-)$ commutes with coproducts. Here is an informal overview of this section.\\

\textbf{Subsection 4.1: (Co)Silting objects in triangulated categories} 
\begin{itemize}
\item We introduce the notion of silting (respectively, cosilting) objects in a triangulated category and list some examples and properties. In particular, we see in Proposition \ref{gen cogen heart} that the hearts of the associated t-structures have a projective generator (respectively, an injective cogenerator).
\end{itemize}

\textbf{Subsection 4.2: Bounded (co)silting objects}
\begin{itemize}
\item We show that silting objects in derived categories of Grothendieck categories admit a more familiar description (Proposition~\ref{silting as expected}).
\item We define bounded (co)silting objects through the requirement that their associated t-structures restrict to bounded derived categories. This is necessary for the applications in Sections 5 and 6.
\item We prove that bounded silting objects in $\mathsf{D}(R)$, for a ring $R$, lie in $\mathsf{K}^b(\Proj{R})$ (Proposition \ref{bdd silting over rings}).
\end{itemize}

\subsection{(Co)Silting objects in triangulated categories}
We begin with the key notions for this section. 
\begin{defn}
\label{defncosiltcotilt}
An object $M$ in a triangulated category $\T$ is called:
\begin{itemize}
\item \textbf{silting} if $(M^{\perp_{>0}},M^{\perp_{<0}})$ is a t-structure in $\T$ and $M\in M^{\perp_{>0}}$;
\item \textbf{cosilting} if $({}^{\perp_{<0}}M,{}^{\perp_{>0}}M)$ is a t-structure in $\T$ and $M\in{}^{\perp_{>0}}M$;
\item \textbf{tilting} if it is silting and $\Add(M)\subset M^{\perp_{\neq 0}}$;
\item \textbf{cotilting} if it is cosilting and $\Prod(M)\subset{}^{\perp_{\neq 0}}M$.
\end{itemize}
We say that a t-structure is silting (respectively, cosilting, tilting or cotilting) if it arises as above from a silting (respectively, cosilting, tilting or cotilting) object.
\end{defn}

Note that, in parallel work \cite{NSZ}, silting objects are defined in an equivalent way. 

It is clear from the definition that an object $M$ is silting in $\T$ if and only if $M$ is cosilting in the opposite category $\T^{op}$. Hence, as we will see, many facts about silting can easily be dually stated for cosilting. A first easy observation, for example, is that if $M$ is silting, then $\Add(M)$ lies in $M^{\perp_{>0}}$ and, dually, if $M$ is cosilting, then $\Prod(M)$ lies in ${}^{\perp_{>0}}M$. 

Recall that an object $X$ in an abelian category $\A$ is a  \textbf{generator} (respectively, a \textbf{cogenerator}) if $\Hom_\A(X,-)$ (respectively, $\Hom_\A(-,X)$) is a faithful functor. It is well-known (see, for example, \cite[Propositions IV.6.3 and IV.6.5]{St}) that a projective (respectively, injective) object $X$ is a generator (respectively, cogenerator) if and only if $\Hom_\A(X,Y)\neq 0$ (respectively, $\Hom_\A(Y,X)\neq 0$) for all $Y$ in $\A$. If $\A$ is cocomplete, then $G$ is a generator if and only if every object in $\A$ is isomorphic to a quotient of a coproduct of copies of $G$ (\cite[Proposition IV.6.2]{St}).

\begin{exam}
\label{easyexample}
The following span some expected classes of examples. 
\begin{enumerate}
\item Let $\T$ be a TR5 triangulated category. Then it follows from \cite[Corollary 4.7]{AI} that any silting object in the 
sense of Aihara and Iyama in \cite{AI} is silting according to our definition. In fact, our definition of silting is motivated by that result. In particular, any tilting object in a TR5 triangulated category, as defined in \cite{BR} is tilting according to our definition. 
\item Let $\Acal$ be an abelian category. If $\A$ has a projective generator $P$, then $P$ is a silting (in fact, tilting) object in $\mathsf{D}(\A)$ and the associated silting t-structure is the standard one. If $\A$ has an injective cogenerator $E$, then $E$ is a cosilting (in fact, cotilting) object in $\mathsf{D}(\A)$ and the associated cosilting t-structure is also the standard one. For a proof, see Lemma \ref{generator properties} and Remark \ref{dual generator properties}.
\item Let $\A$ be an abelian category with a projective generator $P$ (respectively, an injective cogenerator $E$) and $\T$ a triangulated category. If $\Phi\colon \mathsf{D}(\A)\lxr \T$ is a triangle equivalence, then $\Phi(P)$ is a tilting object in $\T$ (respectively, $\Phi(E)$ is a cotilting object in $\T$).
\item Let $\A$ be a Grothendieck category. Then any 1-tilting object $T$ in $\A$ in the sense of \cite{Colpi} is a tilting object in $\mathsf{D}(\A)$ according to our definition. It is easy to check that the t-structure associated to $T$ is the HRS-tilt (\cite[Proposition 2.1]{HRS}) corresponding to the torsion pair $(\Gen(T), T^{\perp_0})$ in $\A$ (see \cite[Theorem 4.9]{AMV1} for an analogous argument when $\A=\Mod{R}$, for a ring $R$).
\item Let $A$ be a ring. It follows from  \cite[Proposition 4.2]{AMV1} that any silting complex $M$ in $\mathsf{D}(A)$ following the definition in \cite{Wei, AMV1} is silting according to our definition. This includes, in particular, any compact tilting complex, as originally defined by Rickard in \cite{Rick}. 
\item Let $A$ be a ring. Then it is shown in \cite[Theorem 4.5]{Stovicek} that any (large) $n$-cotilting module is cotilting according to our definition. Moreover, it is also shown in \cite{Ba} that any (large) $n$-tilting module is tilting according to our definition. For commutative rings, t-structures for tilting and cotilting modules were considered in \cite{AS}.
\item Let $\Lambda$ be a finite dimensional algebra over a field $\mathbb{K}$. Given a compact object $M$ in $\mathsf{D}(\Lambda)$, there is a natural equivalence 
$\Hom_\mathbb{K}(\Hom_{\mathsf{D}(\Lambda)}(M,-),\mathbb{K})\cong \Hom_{\mathsf{D}(\Lambda)}(-,\nu M)$
of contravariant functors between $\mathsf{D}(A)$ and the category of abelian groups, where $\nu:=-\otimes^\mathbb{L}_\Lambda\Hom_\mathbb{K}(\Lambda,\mathbb{K})$ is the Nakayama functor (see \cite{KL} for more details). Using this equivalence, it is easy to show that if $M$ is a compact silting object in $\mathsf{D}(\Lambda)$, then $\nu M$ is a cosilting object in $\mathsf{D}(\Lambda)$. Compact silting objects over finite dimensional algebras were the first ones to be studied, already in \cite{KV} and later in \cite{KN} and \cite{KY}, among others. 
\end{enumerate}
\end{exam}

Given a silting or a cosilting object $M$ in $\T$, we denote by $\mathbb{T}_M$ the associated silting or cosilting t-structure and by $\Hcal_M$ its heart. Note that if $M$ is silting, then $\Hcal_M=M^{\perp_{\neq 0}}$ and if $M$ is cosilting, then $\Hcal_M={}^{\perp_{\neq 0}}M$. Given such a t-structure, we denote by $\textsf{H}^n_M\colon \T\lxr \Hcal_M$ the associated cohomology functors and by $\tau_M^{\leq n}$ and $\tau_{M}^{\geq n}$ the corresponding truncation functors, for all $n\in\mathbb{Z}$.

\begin{prop}\label{gen cogen heart}
Let $\T$ be a triangulated category. If $M$ is a silting (respectively, cosilting) object in $\T$, then $M$ is a generator (respectively, cogenerator) of $\T$, $\mathbb{T}_M$ is a nondegenerate t-structure in $\T$ and $\textsf{H}^{ \, 0}_M(M)$ is a projective generator (respectively, injective cogenerator) in $\Hcal_M$. 
\end{prop}
\begin{proof}
Assume first that $M$ is silting. Given $Y$ in $\T$, consider the canonical triangle
\[
\xymatrix{
\tau_M^{\leq -1}Y \ar[r] & Y \ar[r] & \tau^{\geq 0}_MY \ar[r] & (\tau_M^{\leq -1}Y)[1]
}
.\]
If $Y$ lies in $M^{\perp_\mathbb{Z}}$, it follows by applying $\Hom_\T(M,-)$ to the triangle and its rotations that $Y=0$ and, hence, $M$ is a generator in $\T$. From this fact, it easily follows that the t-structure $\mathbb{T}_M$ is nondegenerate. In fact, if an object $X$ lies in $M^{\perp_{>0}}[k]$ for all $k\in\mathbb{Z}$, then $\Hom_\T(M,X[k+1])=0$ for all $k\in\mathbb{Z}$ and, thus, $X=0$; similarly, if an object $X$ lies in $M^{\perp_{<0}}[k]$ for all $k\in\mathbb{Z}$, we get that $X=0$, as wanted.

Now we show that $\textsf{H}^0_M(M)$ is projective in $\Hcal_M$, i.e. that $\Ext^1_{\Hcal_M}(\textsf{H}^0_M(M),X)=0$, for any $X$ in $\Hcal_M$. From Theorem \ref{real}(ii), we have that $\Ext^1_{\Hcal_M}(\textsf{H}^0_M(M),X)\cong\Hom_{\T}(\textsf{H}^0_M(M),X[1])$. Since $X$ lies in $M^{\perp_{>0}}$, we have that $\Hom_\T(M,X[1])=0$ and, since $X[1]$ lies in $M^{\perp_{<0}}[1]$ and $(\tau_M^{\leq -1}M)[1]$ in $M^{\perp_{>0}}[2]$, we conclude that $\Hom_\T((\tau_M^{\leq -1}M)[1],X[1])=0$. Using the above triangle for $Y=M$, where $\tau^{\geq 0}_M(M)=\textsf{H}_M^0(M)$, it follows that $\Hom_\T(\textsf{H}^0_M(M),X[1])=0$, as wanted.

To show that a projective object in $\Hcal_M$ is a generator, it is enough to show that it has non-zero morphisms to any object in $\Hcal_M$. Since $\Hcal_M=M^{\perp_{\neq 0}}$ and $M$ generates $\T$, we have $\Hom_\T(M,X)\neq 0$ for any non-zero $X$ in $\Hcal_M$. Therefore, using again the above triangle for $Y=M$, we can conclude that $\Hom_\T(\mathsf{H}^0_M(M),X)\cong \Hom_\T(M,X)\neq 0$, showing that $\textsf{H}^0_M(M)$ is a projective generator.

If $M$ is cosilting, for any $Y$ in $\T$ we may again consider the above triangle and apply the dual arguments to those in the previous paragraphs to see that $M$ is a cogenerator in $\T$, $\mathbb{T}_M$ is nondegenerate and $\textsf{H}^0_M(M)$ is an injective cogenerator of $\Hcal_M$. Note that since $M$ is cosilting, for $Y=M$ we have $\tau^{\leq 0}_M(M)=\mathsf{H}^0_M(M)$.
\end{proof}

\begin{rem}\label{nondeg}
From \cite[Proposition 1.3.7]{BBD} (see also \cite[Theorem IV.4.11]{GM}), since silting t-structures are nondegenerate, they can be \textit{cohomologically described}, i.e. for a silting object $M$ in $\T$, we have
\[
M^{\perp_{>0}}=\{X\in\T \ | \ \textsf{H}_M^k(X)=0, \forall k>0\}\ \ \ \ \ \ M^{\perp_{<0}}=\{X\in\T \ | \ \textsf{H}_M^k(X)=0, \forall k<0\}.
\]
Dually, cosilting t-structures are nondegenerate and can also be, therefore, cohomologically described.
\end{rem}

With further assumptions on $\T$ we can say more about hearts of silting and cosilting t-structures. We say that a t-structure in a TR5 (respectively, TR5*) triangulated category is \textbf{smashing} (respectively, \textbf{cosmashing}) if its coaisle (respectively, its aisle) is closed under coproducts (respectively, products).

\begin{lem}\label{facts silting t-structures}
Let $\T$ be a triangulated category and $M$ an object in $\T$. 
\begin{enumerate}
\item \textnormal{\cite[Propositions 3.2 and 3.3]{PaSa}} If $\T$ is TR5 (respectively, TR5*), then the heart of any t-structure in $\T$ has set-indexed coproducts (respectively, products). Furthermore, if the t-structure is smashing (respectively, cosmashing), then coproducts (respectively, products) are exact in the heart.
\item If $\T$ is TR5 and $M$ is a  silting object of $\T$, then the following statements hold.
\begin{enumerate}
\item $\Add(M)=M^{\perp_{>0}}\cap {}^{\perp_0}(M^{\perp_{>0}}[1])$;
\item If $\T$ is also TR5*, $(M^{\perp_{>0}},M^{\perp_{<0}})$ is a cosmashing t-structure and products are exact in $\Hcal_M$.
\end{enumerate}
\item If $\T$ is TR5* and $M$ is a cosilting object of $\T$, then the following statements hold.
\begin{enumerate}
\item $\Prod(M)={}^{\perp_{>0}}M\cap ({}^{\perp_{>0}}M[-1])^{\perp_0}$;
\item If $\T$ is also TR5, $({}^{\perp_{<0}}M,{}^{\perp_{>0}}M)$ is a smashing t-structure and coproducts are exact in $\Hcal_M$.
\end{enumerate} 
\end{enumerate}
\end{lem}
\begin{proof}
Statement (i) was proved in \cite[Propositions 3.2 and 3.3]{PaSa}. Recall that in a TR5 (respectively,~TR5*) triangulated category, given a family of objects $(X_i)_{i\in I}$, for some set $I$, in the heart $\Hcal$ of a t-structure $(\mathbb{T}^{\leq 0},\mathbb{T}^{\geq 0})$, their coproduct (respectively, their product) in $\Hcal$ is given by $\tau^{\geq 0}\coprod_{i\in I}X_i$ (respectively, $\tau^{\leq 0}\prod_{i\in I}X_i$), where the coproduct $\coprod_{i\in I}X_i$ (respectively, the product $\prod_{i\in I}X_i$) is taken in $\T$. 

We now prove (ii) (the arguments for (iii) are dual). The proof of (ii)(a) essentially mimics that of \cite[Lemma 4.5]{AMV1}, as follows. Let $M$ be a silting object in $\T$. It is clear that $\Add(M)\subseteq M^{\perp_{>0}}\cap {}^{\perp_0}(M^{\perp_{>0}}[1])$. Conversely, let $X$ lie in $M^{\perp_{>0}}\cap {}^{\perp_0}(M^{\perp_{>0}}[1])$ and let $I$ be the set $\Hom_{\T}(M,X)$. Consider the triangle
\[
\xymatrix{
K \ar[r] & M^{(I)} \ar[r]^{u} & X\ar[r]^{v  \ \ } & K[1],}
\]
where $u$ is the universal map from $M$ to $X$. We show that $v=0$ and, thus, $u$ splits. Since $\Hom_\T(M,u)$ is surjective and $\Hom_\T(M,M^{(I)}[1])=0$, it follows that $\Hom_{\T}(M,K[1])=0$. Similarly, we have that, for all $i\geq 1$, $\Hom_\T(M,K[i])=0$. Thus, $K$ lies in $M^{\perp_{>0}}$. Since $X$ lies in ${}^{\perp_0}(M^{\perp_{>0}}[1])$, it follows that $v=0$. 
Finally, to prove (ii)(b), it suffices to note that $M^{\perp_{>0}}$ is closed under products whenever they exist.
\end{proof}

\begin{defn}
Let $\T$ be a triangulated category and $M$ and $N$ silting (respectively, cosilting) objects of $\T$. We say that $M$ and $N$ are \textbf{equivalent} if they yield the same silting (respectively, cosilting) t-structure.
\end{defn}

By Lemma \ref{facts silting t-structures}, two silting (respectively, cosilting) objects $M$ and $N$ in a TR5 (respectively, TR5*) triangulated category are equivalent if and only if $\Add(M)=\Add(N)$ (respectively, $\Prod(M)=\Prod(N)$). 

The following corollary brings us to the more familiar setting of compact tilting objects, where the endomorphism ring of a fixed tilting object plays an important role (see also \cite[Corollary 4.2]{BR} and \cite[Theorem 1.3]{HKM}).

\begin{cor}
\label{corheartmodulecat}
Let $\T$ be a TR5 triangulated category and $M$ a silting object which is equivalent to a compact one. Then the heart $\Hcal_M$ is equivalent to $\Mod{\End_{\T}(\textsf{H}^{ \, 0}_M(M))}$.
\end{cor}
\begin{proof}
From Proposition \ref{gen cogen heart}, $\textsf{H}^0_M(M)$ is a projective generator of $\Hcal_M$. Since equivalent silting objects yield the same t-structure, assume without loss of generality that $M$ is compact. Using the triangle
\[
\xymatrix{
\tau_M^{\leq -1}M \ar[r] & M \ar[r] & \textsf{H}^{0}_M(M) \ar[r] & (\tau_M^{\leq -1}M)[1]
}
\]
it is easy to see that there is an isomorphism of functors $\Hom_{\Hcal_M}(\textsf{H}^0_M(M),-)\cong \Hom_{\T}(M,-)_{|\Hcal_M}$. Since $M$ is compact, the coaisle $M^{\perp_{<0}}$ is closed under coproducts. By the proof of Lemma \ref{facts silting t-structures}(i), coproducts in $\Hcal_M$ coincide with coproducts in $\T$ and, thus, both functors above commute with coproducts. This shows that $\textsf{H}^0_M(M)$ is small in $\Hcal_M$. The result then follows by classical Morita theory.
\end{proof}

One may ask for another way of describing the (co)aisle of a (co)silting t-structure in terms of the given (co)silting object $M$. We will see that there is a \textit{smallest (co)aisle containing $M$} and that it coincides with the (co)aisle of the (co)silting t-structure. For that, recall that a subcategory of a triangulated category is said to be \textbf{suspended} (respectively, \textbf{cosuspended}) if it is closed under extensions and positive (respectively, negative) iterations of the suspension functor. Regarding these subcategories, one has the following useful lemma coming from \cite{AJS}.

\begin{lem}\textnormal{\cite[Lemma 3.1]{AJS}}
\label{AJS lemma}
Let $\T$ be a triangulated category. Let $X$ be an object in $\T$ and $\mathcal{S}$ be the smallest suspended (respectively, cosuspended) subcategory containing $X$ and closed under summands and all existing coproducts (respectively, products). 
Then $\mathcal{S}^{\perp_0}=X^{\perp_{\leq 0}}$ (respectively, ${}^{\perp_0}\mathcal{S}={}^{\perp_{\leq 0}}X$).
\end{lem}
\begin{proof}
The proof in \cite[Lemma 3.1]{AJS} does not depend on the existence of arbitrary coproducts. Also, the arguments can easily be dualised to obtain the cosuspended case.
\end{proof}

The smallest aisle of a triangulated category $\T$ containing an object $X$, denoted by $\aisle(X)$, is known to exist whenever $\T$ is the derived category of a Grothendieck abelian category or whenever $X$ is a compact object (\cite[Theorems 3.4 and A.1]{AJS}). We show that it also exists for a silting object in any triangulated category  (the dual statement about the smallest coaisle containing a cosilting object also holds). 

\begin{prop}\label{smallest aisle}
Let $M$ be a silting (respectively, cosilting) object in a triangulated category $\T$. Then the smallest aisle (respectively, coaisle) containing $M$ exists and it coincides with $M^{\perp_{>0}}$ (respectively, ${}^{\perp_{>0}}M$).
\end{prop}
\begin{proof}
Let us prove the silting case (the cosilting case can be proved dually). Let $M$ be a silting object and let $\mathcal{S}$ denote the smallest suspended subcategory of $\T$ containing $M$ and closed under summands and existing coproducts. Let $\mathcal{V}$ be an aisle containing $M$ (which exists since $M^{\perp_{>0}}$ is an aisle where $M$ lies). Then certainly we have that $\mathcal{V}\supseteq \mathcal{S}$ since aisles are suspended, closed under summands and existing sums. Hence, we also have ${}^{\perp_0}(\mathcal{V}^{\perp_0})\supseteq {}^{\perp_0}(\mathcal{S}^{\perp_0})$. But in this inclusion, the left-hand side coincides with $\mathcal{V}$ since $\mathcal{V}$ is an aisle and the right-hand side clearly contains ${}^{\perp_0}(M^{\perp_{\leq 0}})$ (it in fact coincides with it by Lemma~\ref{AJS lemma}). Since $M$ is silting, $M^{\perp_{>0}}={}^{\perp_0}(M^{\perp_{\leq 0}})$ and, thus, $\mathcal{V}$ contains $M^{\perp_{>0}}$, showing that $M^{\perp_{>0}}$ is indeed the smallest aisle containing $M$.
\end{proof}

\subsection{Bounded (co)silting objects}
We will now discuss examples of silting and cosilting objects in derived categories of abelian categories. As suggested by Example~\ref{easyexample}, a good assumption on the underlying abelian category $\A$ is the presence of a generator or a cogenerator. Particularly well-behaved examples of abelian categories with a generator (and a cogenerator) are Grothendieck abelian categories. An abelian category $\A$ is said to be \textbf{Grothendieck} if it is a cocomplete abelian category with exact direct limits and a generator.  It is well-known that if $\A$ is Grothendieck, $\mathsf{D}(\A)$ is both TR5 and TR5* (see also Remark \ref{prods}).  Moreover, $\A$ has an injective cogenerator and it also follows from \cite[Theorem 5.4]{AJS2} that every object in $\mathsf{D}(\A)$ admits a homotopically injective coresolution.  

We begin by collecting some useful observations that clarify the relation between the existence of a generator (or cogenerator) in an abelian category $\A$ and the standard t-structure in $\mathsf{D}(\A)$. The assumptions in the lemma include the particular case when $\A$ is Grothendieck.

 \begin{lem}\label{Lemma generator1}\label{generator properties}
Let $\A$ be an abelian category with a generator $G$. Suppose that every object in $\mathsf{D}(\A)$ admits a homotopically injective coresolution or that $G$ is a projective object in $\A$. Then we have:
\begin{enumerate}
\item For any object $X$ in $\mathsf{D}(\A)$ and any $i\in\mathbb{Z}$, if $\Hom_{\mathsf{D}(\A)}(G,X[i])=0$ then $\mathsf{H}^i_0(X)=0$;
\item If $\mathsf{D}(\A)$ is TR5, the smallest aisle of $\mathsf{D}(\A)$ containing $G$ exists and coincides with the standard aisle $\mathbb{D}^{\leq 0}$;
\item Consider the following statements:
\begin{enumerate}
\item $G$ is projective in $\A$;
\item For any object $X$ in $\mathsf{D}(\A)$ and any $i\in\mathbb{Z}$, $\mathsf{H}^i_0(X)=0$ if and only if $\Hom_{\mathsf{D}(\A)}(G,X[i])=0$;
\item $G$ is a silting object in $\mathsf{D}(\A)$;
\item $G$ is a tilting object in $\mathsf{D}(\A)$.
\end{enumerate}
Then \textnormal{(a)}$\Longrightarrow$\textnormal{(b)}$\Longrightarrow$\textnormal{(c)}$\Longleftrightarrow$\textnormal{(d)}. If, furthermore, $\mathsf{D}(\A)$ is TR5, then the statements are equivalent.
\end{enumerate}
\end{lem}
\begin{proof}
(i) Assume that $\textsf{H}^i_0(X)\neq 0$. We will construct a non-trivial map from $G$ to $X[i]$. By our hypothesis on $\mathsf{D}(\A)$, without loss of generality we may assume that $X=(X^j,d^j)_{j\in\mathbb{Z}}$ is a complex of injective objects in $\A$ or that $G$ is projective in $\A$. Thus, we have $\Hom_{\mathsf{D}(\A)}(G,X[i])=\Hom_{\mathsf{K}(\A)}(G,X[i])$. Denote by $q\colon\Ker{d^i}\lxr \mathsf{H}^i_0(X)$ the natural projection map. Since $G$ is a generator and $\mathsf{H}^{i}_0(X)\neq 0$, $\Hom_\A(G,q)\neq 0$ and, thus, there is a map $f\colon G\lxr \Ker{d^i}$ such that $q\circ f\neq 0$, i.e. $f$ does not factor through the natural inclusion $k\colon\Image{d^{i-1}}\lxr \Ker{d^i}$. Consider the following diagram
\[
\xymatrix{&&& \Image{d^{i-1}}\ar[r]^k& \Ker{d^i}\ar[rd]^m& G\ar[l]_{\ \ \ f}\ar@{-->}[d]^{m\circ f}\\ 
\cdots\ar[r] &X^{i-2}\ar[r]^{d^{i-2}}& X^{i-1}\ar[ur]^p\ar[rrr]^{d^{i-1}}&&& X^{i}\ar[r]^{d^{i}}&X^{i+1}\ar[r]& \cdots}
\]
where $m$ and $p$ are the obviously induced maps. Then it is clear that the map $m\circ f$ lies in $\Hom_{\mathsf{K}(\A)}(G,X[i])$ and it is non-zero by the choice of $f$.

(ii) Let $\mathcal{S}$ denote the smallest suspended subcategory of $\mathsf{D}(\A)$ containing $G$ and closed under coproducts. If $\aisle(G)$ exists, it contains $\mathcal{S}$ and, since $G$ lies in $\mathbb{D}^{\leq 0}$, it is also contained in $\mathbb{D}^{\leq 0}$. Hence, if one shows that $\mathcal{S}=\mathbb{D}^{\leq 0}$ the statement follows. By Lemma \ref{AJS lemma}, it follows that $\mathcal{S}^{\perp_0}=G^{\perp_{\leq 0}}$. It is clear that $\mathbb{D}^{\geq 1}\subseteq G^{\perp_{\leq 0}}$. On the other hand, given $X$ in $G^{\perp_{\leq 0}}$, it follows from (i) that $\mathsf{H}_0^i(X)=0$ for all $i\leq 0$. Hence, we have that $\mathbb{D}^{\geq 1}\supseteq G^{\perp_{\leq 0}}$ and ${}^{\perp_0}(\mathcal{S}^{\perp_0})=\mathbb{D}^{\leq 0}$ as wanted.

(iii) (a)$\Longrightarrow$(b): Assume that $G$ is projective and suppose that $\mathsf{H}^i_0(X)=0$ (which, following the notation of the diagram above, is equivalent to $k\colon\Image{d^{i-1}}\lxr \Ker{d^i}$ being an isomorphism). Let $g\colon G\lxr X[i]$ be a map in $\mathsf{D}(\A)$. Since $G$ is projective, this is also a map in $K(\A)$ and, therefore, $d^ig=0$, i.e. $g$ induces a map $\tilde{g}\colon G\lxr \Ker{d^i}$. Since $k\colon \Image{d^{i-1}}\lxr \Ker{d^i}$ is an isomorphism by assumption, $\tilde{g}$ factors through $k$. This factorisation yields a homotopy, thus showing that $g=0$, as wanted.

(b)$\Longrightarrow$(c): From (b) it is clear that we have the equality $(G^{\perp_{>0}},G^{\perp_{<0}})=(\mathbb{D}^{\leq 0},\mathbb{D}^{\geq 0})$, thus showing that $(G^{\perp_{>0}},G^{\perp_{<0}})$ is a t-structure. Since $G$ lies in $\mathbb{D}^{\leq 0}$, then it also lies in $G^{\perp_{>0}}$, showing that it is silting.

(c)$\Longleftrightarrow$(d): Objects of $\A$ do not admit negative self-extensions. In particular, $G$ lies in $G^{\perp_{<0}}$. Therefore, $G$ is silting if and only if it is tilting.

When $\mathsf{D}(\A)$ is TR5, we can also prove (d)$\Longrightarrow$(a): If $G$ is a tilting object in $\mathsf{D}(\A)$, then $(G^{\perp_{>0}},G^{\perp_{<0}})$ is a t-structure and $G$ lies in $\Hcal_G$, where it is projective (Proposition \ref{gen cogen heart}). From Proposition \ref{smallest aisle} and item (ii) above, we have that $G^{\perp_{>0}}=\aisle(G)=\mathbb{D}^{\leq 0}$. Hence, the associated tilting t-structure is the standard one and $\Hcal_G=\A$, showing that $G$ is projective in $\A$.
\end{proof}

\begin{rem}
\label{dual generator properties}
Note that if $\A$ has an injective cogenerator $E$, then one can argue as in (iii) above to show that $E$ is a cotilting object in $\mathsf{D}(\A)$ and that for any object $X$ in $\mathsf{D}(\A)$ and any $i\in\mathbb{Z}$, $\mathsf{H}^i_0(X)=0$ if and only if $\Hom_{\mathsf{D}(\A)}(X,E[-i])=0$, i.e. the associated cotilting t-structure is the standard one. 
\end{rem}

The above lemma and the intrinsic properties of tilting t-structures yield the following corollary.

\begin{cor}
Let $\A$ be an abelian category such that $\mathsf{D}(\A)$ is TR5 and TR5*. If $\A$ has a projective generator (respectively, an injective cogenerator), then $\A$ has exact products (respectively, coproducts). In particular, if $\A$ is a Grothendieck abelian category with a projective generator, then $\A$ has exact products.
\end{cor}
\begin{proof}
From Lemma \ref{generator properties}, the projective generator $G$ is a tilting object in $\mathsf{D}(\A)$ whose associated tilting t-structure is the standard one. From Lemma \ref{facts silting t-structures}(ii)(b) it then follows that $\Hcal_G\cong \A$ has exact products. The dual statement follows from Remark \ref{dual generator properties} and Lemma \ref{facts silting t-structures}(iii)(b).
\end{proof}

As mentioned earlier, if $\A$ is a Grothendieck category, the smallest aisle of $\mathsf{D}(\A)$ containing an object always exists (\cite[Theorem 3.4]{AJS}). Using this fact, silting objects become easier to describe.  This description covers the previous definitions of silting objects occurring in the literature (\cite{AI,AMV1,KV,KN,KY,Wei}).

\begin{prop}\label{silting as expected}
Let $\A$ be a Grothendieck category and $M$ an object in $\mathsf{D}(\A)$. Then $M$ is silting if and only if the following three conditions hold:
\begin{enumerate}
\item $\Hom_{\mathsf{D}(\A)}(M,M[k])=0$, for all $k>0$;
\item $M$ generates $\mathsf{D}(\A)$;
\item $M^{\perp_{>0}}$ is closed under coproducts.
\end{enumerate}
\end{prop}
\begin{proof}
If $M$ is silting, it follows from Proposition \ref{gen cogen heart} that $M$ is a generator and, thus, it satisfies the listed conditions. Conversely, suppose $M$ satisfies (i), (ii) and (iii). Since $\A$ is a Grothendieck category, it follows from \cite[Theorem 3.4]{AJS} that $\textsf{aisle}(M)$ coincides with the smallest subcategory containing $M$ and closed under positive shifts, extensions and coproducts. As a consequence, and by assumption on $M^{\perp_{>0}}$, we have that $\aisle(M)^{\perp_0}=M^{\perp_{\leq 0}}$ (by Lemma \ref{AJS lemma}) and $\aisle(M)\subseteq M^{\perp_{>0}}$. We will show the reverse inclusion, thus proving that $M$ is silting. Let $X$ be an object in $M^{\perp_{>0}}$ and consider a triangle
\[
\xymatrix{
Y \ar[r] & X \ar[r] & Z \ar[r] & Y[1]
}
\]
such that $Y$ lies in $\aisle(M)$ and $Z$ lies in $\aisle(M)^{\perp_0}=M^{\perp_{\leq 0}}$. We show that $Z=0$. Note that $\Hom_{\mathsf{D}(\A)}(M,Z[k])=0$, for all $k\leq 0$. Let $k<0$ and consider a map $f\colon M[k]\lxr Z$. Since $\aisle(M)\subseteq M^{\perp_{>0}}$ (and, thus, $\Hom_{\mathsf{D}(\A)}(M[k],Y[1])=0$), it follows that there is $\bar{f}\colon M[k]\lxr X$ which, by assumption on $X$, must vanish. Therefore, $f=0$ and since $M$ generates $\mathsf{D}(\A)$, we get that $Z=0$. From the above triangle we infer that $X$ lies in $\aisle(M)$. 
\end{proof}

This proposition shows, in particular, that the compact silting (and tilting) complexes in derived module categories that appear abundantly in the literature fit in our definition. We will now explore in more detail the connection between the (not necessarily compact) silting complexes defined in \cite{Wei, AMV1} and silting objects as defined here. For this purpose we need the notion of a bounded (co)silting object, which will play an important role in the coming sections. This concept is defined via a property of the associated t-structure.

\begin{lem}\label{eq bdd}
Let $\A$ be an abelian category and let $\mathbb{T}=(\mathbb{T}^{\leq 0},\mathbb{T}^{\geq 0})$ be a t-structure in $\mathsf{D}(\A)$ with heart $\Hcal$. Then the following are equivalent.
\begin{enumerate}
\item The full subcategory $\mathsf{D}(\A)^{b_\mathbb{T}}$ of $\mathsf{D}(\A)$ with objects $\mathsf{Ob}\ \mathsf{D}(\A)^{b(\mathbb{T})}=\bigcup_{n,m\in\mathbb{Z}}\mathsf{Ob}\  \mathbb{T}^{\leq n}\cap\mathbb{T}^{\geq m}$ coincides with $\mathsf{D}^b(\A)$.
\item $\mathbb{T}\cap\mathsf{D}^b(\A):=(\mathbb{T}^{\leq 0}\cap\mathsf{D}^b(\A),\mathbb{T}^{\geq 0}\cap\mathsf{D}^b(\A))$ is a bounded t-structure in $\mathsf{D}^b(\A)$ and $\Hcal\subseteq \mathsf{D}^b(\A)$.
\end{enumerate}
If $\A$ is Grothendieck, then the above conditions are furthermore equivalent to the following ones.
\begin{enumerate}
\item[(iii)]  There are integers $n\leq m$ such that $\mathbb{D}^{\leq n}\subseteq \mathbb{T}^{\leq 0}\subseteq \mathbb{D}^{\leq m}$.
\item[(iv)] $\mathsf{Ob}\ \mathsf{D}^-(\A)=\bigcup_{n\in\mathbb{Z}}\mathsf{Ob}\ \mathbb{T}^{\leq 0}[n]$ and $\mathsf{Ob}\ \mathsf{D}^+(\A)=\bigcup_{n\in\mathbb{Z}}\mathsf{Ob}\ \mathbb{T}^{\geq 0}[n]$.
\end{enumerate}
\end{lem}
\begin{proof}
(i)$\Longrightarrow$(ii): We first observe that $(\mathbb{T}^{\leq 0}\cap\mathsf{D}(\A)^{b(\mathbb{T})}, \mathbb{T}^{\geq 0}\cap\mathsf{D}(\A)^{b(\mathbb{T})})$ is always a t-structure in $\mathsf{D}(\A)^{b(\mathbb{T})}$. In fact, given an object $X$ in $\mathsf{D}(\A)^{b(\mathbb{T})}$ it is clear that its truncations with respect to $\mathbb{T}$, $\tau_\mathbb{T}^{\leq 0}X$ and $\tau_\mathbb{T}^{\geq 1}X$, lie, respectively, in $\cup_{n\in\mathbb{Z}}\mathbb{T}^{\leq n}$ and in $\cup_{n\in\mathbb{Z}}\mathbb{T}^{\geq n}$. Each of these unions forms a triangulated subcategory of $\mathsf{D}(\A)$. Since $X$ lies in both, it follows that so do its truncations. Hence, $(\mathbb{T}^{\leq 0}\cap\mathsf{D}(\A)^{b(\mathbb{T})}, \mathbb{T}^{\geq 0}\cap\mathsf{D}(\A)^{b(\mathbb{T})})$ is indeed a t-structure and (ii) then follows immediately from (i).

(ii)$\Longrightarrow$(i): Note that every object in $\mathbb{T}^{\leq n}\cap \mathbb{T}^{\geq m}$ can be obtained by a finite sequence of extensions in $\mathsf{D}(\A)$ of objects in $\Hcal$ (such sequences are often represented by Postnikov towers - see also \cite[Lemma 2.3]{Br}). Since $\Hcal$ lies in $\mathsf{D}^b(\A)$ then so does every object in $\mathbb{T}^{\leq n}\cap \mathbb{T}^{\geq m}$. Hence $\mathsf{D}(\A)^{b(\mathbb{T})}$ is contained in $\mathsf{D}^b(\A)$. The converse inclusion also holds since $\mathbb{T}\cap\mathsf{D}^b(\A)$ is a bounded t-structure in $\mathsf{D}^b(\A)$.

Suppose now that $\A$ is a Grothendieck category with a generator $G$ and an injective cogenerator $E$.

(ii)$\Longrightarrow$(iii): By assumption, $G$ lies in $\mathbb{T}^{\leq k}$ for some integer $k$. By Lemma \ref{generator properties}, we have that $\aisle(G)=\mathbb{D}^{\leq 0}$ and, thus, $\mathbb{D}^{\leq 0}\subseteq \mathbb{T}^{\leq k}$. Since, by assumption, $E$ lies in $\mathbb{T}^{\geq t}$ for some integer $t$ and $\mathbb{T}^{\geq t}$ is closed under negative shifts, using Remark \ref{dual generator properties} we see that  $\mathbb{D}^{\leq 0}={}^{\perp_{<0}}E\supseteq  {}^{\perp_{<0}}(\mathbb{T}^{\geq t})={}^{\perp_{0}}(\mathbb{T}^{\geq t+1})$, as wanted.

(iii)$\Longrightarrow$(iv): This is obvious.

(iv)$\Longrightarrow$(ii): In order to see that the t-structure $\mathbb{T}$ restricts to $\mathsf{D}^b(\A)$ we only need to show (as in (i)$\Longrightarrow$(ii)) that given an object $X$ in $\mathsf{D}^b(\A)$, its truncations with respect to $\mathbb{T}$ also lie in $\mathsf{D}^b(\A)$.  Consider the truncation triangle with respect to $\mathbb{T}$ given as follows
\[
\xymatrix{
\tau^{\leq 0}X \ar[r] & X \ar[r] & \tau^{\geq 1}X \ar[r]^{} &  (\tau^{\leq 0}X)[1].
}
\]
By assumption, $\mathbb{T}^{\leq n}\subseteq \mathsf{D}^-(\A)$ and, thus, $\tau^{\leq 0}X$ lies in $\mathsf{D}^{-}(\A)$. Since $\mathsf{D}^-(\A)$ is a triangulated subcategory and $X$ lies in $\mathsf{D}^b(\A)$, we have that also $\tau^{\geq 1}X$ lies in $\mathsf{D}^-(\A)$. Again by assumption, we also have that $\tau^{\geq 1}X$ lies in $\mathbb{T}^{\geq 1}\subseteq \mathsf{D}^+(\A)$ and, thus, in $\mathsf{D}^-(\A)\cap\mathsf{D}^+(\A)=\mathsf{D}^b(\A)$. Since $\mathsf{D}^b(\A)$ is triangulated, also $\tau^{\leq 0}X$ lies in $\mathsf{D}^b(\A)$. Hence $\mathbb{T}\cap\mathsf{D}^b(\A)$ is a t-structure in $\mathsf{D}^b(\A)$. Finally, it is clear from (iv) that $\mathbb{T}\cap\mathsf{D}^b(\A)$ is bounded and that $\Hcal(\mathbb{T})=\mathbb{T}^{\leq 0}\cap\mathbb{T}^{\geq 0}\subset \mathsf{D}^-(\A)\cap\mathsf{D}^+(\A)=\mathsf{D}^b(\A)$, proving (ii).
\end{proof}

\begin{defn}\label{bounded silting}
Let $\A$ be an abelian category.  A silting (respectively, cosilting) object in $\mathsf{D}(\A)$ is said to be \textbf{bounded} if the associated t-structure satisfies the equivalent conditions (i) and (ii) of Lemma \ref{eq bdd}.
\end{defn}

\begin{rem}\label{bdd silting is bdd}
\begin{enumerate}
\item If $\A$ is an abelian category, any bounded tilting and bounded cotilting objects in $\mathsf{D}(\A)$ lie in $\mathsf{D}^b(\A)$ since they lie in their own associated hearts. 
\item If $\A$ is a Grothendieck category and $M$ is a bounded silting object in $\mathsf{D}(\A)$, then $M$ lies in $\mathsf{D}^b(\A)$. From (iii) of Lemma \ref{eq bdd}, it is clear that there are integers $n\leq m$ such that $M$ lies in $\mathbb{D}^{\leq m}$ and $\mathbb{D}^{\leq n}\subseteq M^{\perp_{>0}}$. For $E$ and injective cogenerator of $\A$, since $E[-n]$ lies in $\mathbb{D}^{\leq n}$, we have that $\Hom_{\mathsf{D}(A)}(M,E[-n+k])=0$ for all $k>0$. In particular, it follows from Remark \ref{dual generator properties} that $M$ lies in $\mathbb{D}^{\geq n+1}$, proving that $M$ is a cohomologically bounded object. 
\item If $\A$ is a Grothendieck category with a projective generator $G$, then the dual arguments to the above example show that bounded cosilting objects also lie in $\mathsf{D}^b(\A)$.
\end{enumerate}
\end{rem}

We will see that, in the derived category of modules over a ring $A$, bounded silting objects are not only cohomologically bounded but they must also lie in $\mathsf{K}^b(\Proj{A})$.  Hence, they coincide with the silting complexes of \cite{Wei, AMV1}. Analogously, bounded cosilting objects in $\mathsf{D}(A)$ must lie in $\mathsf{K}^b(\Inj(A))$. This can be shown using dual arguments to those in \cite[Lemma 4.5]{AMV1} - we leave that to the reader. 

\begin{prop}\label{bdd silting over rings}
Let $A$ be a ring. Then a silting (respectively, cosilting) object is bounded in $\mathsf{D}(A)$ if and only if it lies in $\mathsf{K}^b(\Proj{A})$ (respectively, $K^b(\Inj{A})$).
\end{prop}
\begin{proof}
By Lemma \ref{eq bdd}, the aisle of the silting t-structure associated to a bounded silting object in $\mathsf{D}(A)$ lies between shifts of the standard aisle. It then follows from Lemma \ref{facts silting t-structures}(ii)(a) and \cite[Lemma 4.5]{AMV1} that bounded silting objects lie in $\mathsf{K}^b(\Proj{A})$. For the converse see \cite[Lemma 4.1 and Proposition 4.2]{Wei}.
\end{proof}

We finish this section with an example of an unbounded silting object in the context of quiver representations of infinite quivers. Let $\mathbb{K}$ be a field. Given a (possibly infinite) locally finite quiver $Q$  (i.e. every vertex has only finitely many adjacent arrows), consider the \textbf{path category} of $Q$ to be the category whose objects are the vertices of $Q$ and whose morphisms between two vertices $i$ and $j$ are elements of the $\mathbb{K}$-vector space spanned by the paths between $i$ and $j$. We denote it by $\mathbb{K}Q$. Consider the associated category of functors $\mathcal{M}(\mathbb{K}Q):=((\mathbb{K}Q)^{op},\Mod{\mathbb{K}})$, which is called the category of right modules over $\mathbb{K}Q$. If $Q$ is finite, $\mathcal{M}(\mathbb{K}Q)$ is equivalent to usual category of right modules over the path algebra, $\Mod{\mathbb{K}Q}$. Still, even when $Q$ is infinite, $\mathcal{M}(\mathbb{K}Q)$ is well-known to be a Grothendieck category (see, for example, \cite{Freyd}). Given a vertex $x$ of $Q$, we consider the projective object $P_x=\Hom_{\mathbb{K}Q}(x,-)$ in $\mathcal{M}(\mathbb{K}Q)$.

\begin{exam}\label{unb silting}
Let $Q$ be the linearly oriented quiver of type $A_\infty$, i.e. the quiver
\[ 
\xymatrix{
1 \ar[r] & 2 \ar[r] & 3 \ar[r] & 4 \ar[r] & 5 \ar[r] & \cdots}
\]
Consider the derived category $\mathsf{D}(\mathcal{M}(\mathbb{K}Q))$. We show, using Proposition \ref{silting as expected}, that 
$$M:=\bigoplus\limits_{i\in\mathbb{N}}P_i[i]$$
is a silting object but not a bounded silting object. Since $\bigoplus_{i\in\mathbb{N}}P_i$ is a projective generator of $\mathcal{M}(\mathbb{K}Q)$ (see \cite[Theorem 5.35]{Freyd}), it is clear that $M$ generates $\mathsf{D}(\mathcal{M}(\mathbb{K}Q))$. Furthermore, it is easy to see that $\Hom_{\mathsf{D}(\M(\mathbb{K}Q))}(M,M[i])=0$ for all $i>0$. It remains to check that $M^{\perp_{>0}}$ is closed under coproducts in $\mathcal{M}(\mathbb{K}Q)$. First note that each $P_i$ is compact in $\mathsf{D}(\A)$ (see also \cite[Section 4.2]{Keller:dg}). Now, given a family $(X_\lambda)_{\lambda\in\Lambda}$ of objects in $M^{\perp_{>0}}$, for any $k>0$, we have
$$\Hom_{\mathsf{D}(\M(\mathbb{K}Q))}(M,\bigoplus\limits_{\lambda\in\Lambda}X_\lambda[k])
=\prod_{i\in\mathbb{N}}\bigoplus\limits_{\lambda\in\Lambda}\Hom_{\mathsf{D}(\M(\mathbb{K}Q))}(P_i[i],X_\lambda[k])=0.$$
This shows that $M$ is a silting object in $\mathcal{M}(\mathbb{K}Q)$. Since $\mathcal{M}(\mathbb{K}Q)$ is a Grothendieck category, it follows from Remark \ref{bdd silting is bdd}(ii) that any bounded silting object in $\mathsf{D}(\mathcal{M}(\mathbb{K}Q))$ must lie in $\mathsf{D}^b(\mathcal{M}(\mathbb{K}Q))$. Hence $M$ is not a bounded silting object.
\end{exam}

\section{Derived equivalences}
\label{sectiondereq}
In this section we will combine contents of Sections 3 and 4 in order to discuss derived equivalences arising from realisation functors associated to tilting or cotilting t-structures. We will also reinterpret in terms of realisation functors a problem by Rickard on the \textit{shape} of derived equivalences (\cite{Rick2}). 

We will often consider the unbounded derived category of a heart of the form $\Hcal_M$, for some silting or cosilting object $M$ in a triangulated category $\T$. No set-theoretical problems arise here, since from \cite[Theorem 1]{San}, the category $\mathsf{D}(\Hcal_M)$ exists (i.e. it has $\Hom$-sets) when $\Hcal_M$ has coproducts and enough projectives or when $\Hcal_M$  has products and enough injectives. From Proposition \ref{gen cogen heart} and Lemma \ref{facts silting t-structures}, this includes the silting and cosilting cases, respectively, which are the focus of our approach. 

In our discussion of derived equivalences, we will frequently interchange between considerations on bounded and unbounded derived categories. The reasons for this are already apparent in previous sections. While the unbounded derived category is a better setting for categorical constructions as it often admits products and coproducts (see also Section~\ref{sectionrecollements} for more advantages of working in the unbounded setting), it is in the bounded setting that we come across the current tools to build realisation functors. We believe that this obstacle can be overcome with a different approach to the construction of realisation functors, but this falls outside of the scope of this paper. 

When realisation functors are considered with respect to filtered derived categories, we will omit in the notation of the functor the superscript referring to the f-category. Also for simplicity, the subscript of the functor indicative of a (silting) t-structure will be replaced by the silting or cosilting object that uniquely determines it. Here is an informal overview of this section.\\

\textbf{Subsection 5.1: Tilting and cotilting equivalences}
\begin{itemize}
\item We show that a realisation functor associated to a (co)silting t-structure is fully faithful if and only if it actually comes from a (co)tilting t-structure (Proposition \ref{der equival}).
\item We state a derived Morita theory for abelian categories with a projective generator or an injective cogenerator (Theorem \ref{eq tilt cotilti shape}).
\item We show the invariance of finite global dimension under tilting or cotilting derived equivalences of abelian categories.
\end{itemize}

\textbf{Subsection 5.2: Standard forms}
\begin{itemize}
\item We prove that the existence of a derived equivalence of standard type between $\mathbb{K}$-algebras (projective over a commutative ring $\mathbb{K}$) forces the associated realisation functor to be also an equivalence of standard type. Moreover, we provide an equivalent condition for a derived equivalence between such $\mathbb{K}$-algebras to be of standard type (Theorem \ref{standard f-lifts}).
\item We show that Fourier-Mukai transforms in algebraic geometry are equivalent to some realisation functors (Proposition \ref{FM is real}). In particular, we observe that Fourier-Mukai transforms can be thought of as cotilting equivalences. 
\end{itemize}

\subsection{Tilting and cotilting equivalences}
We begin by discussing realisation functors associated to silting or cosilting objects. 

\begin{prop}\label{der equival}
Let $(\X,\theta)$ be an f-category over a TR5 (respectively, TR5*) triangulated category $\T$. For a silting (respectively, cosilting) object $M$ in $\T$, the realisation functor $\mathsf{real}^\X_M\colon \mathsf{D}^b(\Hcal_M)\lxr \T$ is fully faithful if and only if $M$ is a tilting (respectively, cotilting) object. 
\end{prop}
\begin{proof}
We show the statement for silting/tilting objects, using the condition (\textsf{Ef}) of Theorem \ref{real}. The cosilting/cotilting case is entirely dual (using the condition (\textsf{CoEf}) of Theorem \ref{real}).

Let $M$ be a tilting object in $\T$. We only need to show that condition (\textsf{Ef}) holds for $\Hcal_M$. Take $X$ and $Y$ in $\Hcal_M$ and a morphism $g\colon X\lxr Y[n]$ in $\T$, for some $n\geq 2$. By Proposition \ref{gen cogen heart}, $M$ is a generator in $\Hcal_M$ and, thus, there is an epimorphism $h\colon M^{(I)}\lxr X$ in $\Hcal_M$, for some set $I$. Note that, $\Hcal_M$ admits coproducts and that, since $\Add(M)$ is contained in the heart, coproducts of $M$ in $\Hcal_M$ coincide with those in $\T$ (see Lemma \ref{facts silting t-structures}). Since $Y$ lies in $M^{\perp_{>0}}$, it is clear that $g\circ h=0$ and, thus, we have (\textsf{Ef}).

Conversely, suppose that $\mathsf{real}_M$ is fully faithful (i.e. we assume condition (\textsf{Ef})). We show that, for any set $I$, $\tau_M^{\leq -1}(M^{(I)})=0$, thus proving that  $\Add(M)$ is contained in $\Hcal_M$. Consider the canonical triangle
\[
\xymatrix{
\mathsf{H}^0_M(M^{(I)})[-1] \ar[r]^{} & \tau^{\leq -1}_MM^{(I)} \ar[r] & M^{(I)} \ar[r]^{ } & \mathsf{H}^0_M(M^{(I)})
}
\]
and the canonical morphism $\tau^{\leq -1}_M(M^{(I)})\lxr \tau^{\geq -1}_M\tau^{\leq -1}_MM^{(I)}=\mathsf{H}^{-1}_M(M^{(I)})[1]$. Let $g$ be the morphism between $\mathsf{H}^0_M(M^{(I)})[-1]$ and $\mathsf{H}^{-1}_M(M^{(I)})[1]$ obtained as the composition of the two morphisms above. Now, by condition (\textsf{Ef}), there is an object $C$ in $\Hcal_M$ and an epimorphism $h\colon C\lxr \mathsf{H}^0_M(M^{(I)})$ such that the composition $g[1]\circ h\colon C\lxr \mathsf{H}^0_M(M^{(I)})\lxr \mathsf{H}^{-1}_M(M^{(I)})[2]$ is zero. Now, the arguments in the proof of Proposition \ref{gen cogen heart} can also show that $\mathsf{H}_M^0(M^{(I)})$ is a projective object in $\mathcal{H}_M$ and, thus, the epimorphism $h$ splits in $\Hcal_M$. This proves that $g[1]$ (and hence $g$) is the zero map. Since $\Hom_{\T}(M, \mathsf{H}^{-1}_M(M^{(I)})[1])=0$, we conclude that the canonical map $\tau^{\leq -1}_M(M^{(I)})\lxr \mathsf{H}^{-1}_M(M^{(I)})[1]$ must also be the zero map and, therefore, $\mathsf{H}^{-1}_M(M^{(I)})=0$ (see Subsection 2.1). This shows that $\tau_M^{\leq -2}(M^{(I)})\cong \tau_M^{\leq -1}(M^{(I)})$ and we can repeat the argument by considering the canonical map $\tau_M^{\leq -2}(M^{(I)})\lxr \mathsf{H}^{-2}_M(M^{(I)})[2]$. It follows by induction that $\mathsf{H}^i_M(\tau^{\leq -1}_M(M^{(I)}))=0$ for all $i\leq -1$. Since the t-structure induced by $M$ is nondegenerate (see Remark \ref{nondeg}), it follows that $\tau^{\leq -1}_M(M^{(I)})=0$ and, thus, $M^{(I)}\cong \mathsf{H}^0_M(M^{(I)})$. 
\end{proof}

The following corollary is a non-compact analogue of \cite[Theorem III.4.3]{BR}.

\begin{cor}\label{der equival2}
Let $\A$ be an abelian category such that $\mathsf{D}(\A)$ is TR5 (respectively, TR5*) and let $M$ be a bounded silting (respectively, cosilting) object in $\mathsf{D}(\A)$. Then the functor $\mathsf{real}_M\colon \mathsf{D}^b(\Hcal_M)\lxr \mathsf{D}(\A)$ induces an equivalence between $\mathsf{D}^b(\Hcal_M)$ and $\mathsf{D}^b(\A)$ if and only if $M$ is tilting (respectively, cotilting).
\end{cor}
\begin{proof}
If $M$ is a bounded (co)silting object, then the essential image of $\mathsf{real}_M$ lies in $\mathsf{D}^b(\A)$ (see Lemma \ref{eq bdd}) and the image coincides with $\mathsf{D}^b(\A)$ whenever $\mathsf{real}_M$ is fully faithful (see Theorem \ref{real}). The result then follows from Proposition \ref{der equival}.
\end{proof}

Given a bounded silting or cosilting object $M$ in the derived category of an abelian category, we will keep the notation $\mathsf{real}_M$ for the induced functor with codomain the bounded derived category. We are now able to discuss a derived Morita theory for abelian categories with a projective generator or an injective cogenerator (Theorem A in the introduction). The proof of the following theorem is a simple application of the above proposition and corollary. 

\begin{thm}\label{eq tilt cotilti shape}
Let $\A$ and $\B$ be abelian categories such that $\mathsf{D}(\A)$ is TR5 (respectively, TR5*) and $\B$ has a projective generator (respectively, an injective cogenerator). Consider the following statements.
\begin{enumerate}
\item There is a restrictable triangle equivalence $\Phi\colon \mathsf{D}(\B)\lxr \mathsf{D}(\A)$.
\item There is a bounded tilting (respectively, cotilting) object $M$ in $\mathsf{D}(\A)$ such that $\Hcal_M\cong \B$.
\item There is a triangle equivalence $\phi\colon \mathsf{D}^b(\B)\lxr \mathsf{D}^b(\A)$.
\end{enumerate}
Then we have \textnormal{(i)}$\Longrightarrow$\textnormal{(ii)}$\Longrightarrow$\textnormal{(iii)}. Moreover, if $\B$ has a projective generator and $\A=\Mod{R}$, for a ring $R$, then we also have \textnormal{(iii)}$\Longrightarrow$\textnormal{(ii)}.
\end{thm}
\begin{proof}
Let $\A$ be such that $\mathsf{D}(\A)$ is TR5 and assume that $\B$ has a projective generator $P$ (the proof for $\B$ with an injective cogenerator is entirely dual). By Lemma~\ref{generator properties}(iii), $P$ is a tilting object in $\mathsf{D}(\B)$ and the associated tilting t-structure is the standard one. It is then clear that $P$ is a bounded tilting object. 

(i)$\Longrightarrow$(ii): Denote by $M$ the object $\Phi(P)$. Clearly, $M$ is a tilting object in $\mathsf{D}(\A)$ (see Example \ref{easyexample}(iii)) and the associated tilting t-structure is the image by $\Phi$ of the standard t-structure in $\mathsf{D}(\B)$. Hence, we have $\Hcal_M\cong \B$. Moreover, $M$ is a bounded tilting object since $\Phi$ is a restrictable equivalence (recall Definition \ref{defn restrict extend}).

(ii)$\Longrightarrow$(iii): This follows from Proposition~\ref{der equival} and Corollary~\ref{der equival2}.

Suppose now that  $\A=\Mod{R}$, for a ring $R$, and $\B$ has a projective generator. 

(iii)$\Longrightarrow$(ii): Denote by $M$ the object $\phi(P)$. It follows directly by the arguments in \cite[Proposition 6.2]{Rick} that $\phi$ induces an equivalence between $\mathsf{K}^b(\Proj\B)$ and $\mathsf{K}^b(\Proj{R})$. It is, however, clear that $\phi$ also induces an equivalence between $\mathsf{K}^b(\Proj\B)=\mathsf{K}^b(\Add(P))$ and the smallest thick subcategory of $\mathsf{D}^b(R)$ containing $\Add(M)$ (denoted by $\mathsf{thick}(\Add(M))$). Therefore, we conclude that $M$ is an object of $\mathsf{D}^b(R)$ such that $\Hom_{\mathsf{D}^b(R)}(M,M^{(I)}[k])=0$ for all $k\neq 0$ and $\mathsf{thick}(\Add(M))=\mathsf{K}^b(\Proj{R})$. It then follows from \cite[Proposition 4.2]{AMV1} that $M$ is indeed a tilting object in $\mathsf{D}(R)$. Finally, $M$ is a bounded tilting object as a consequence of Proposition~\ref{bdd silting over rings}.
\end{proof}

\begin{exam}
\begin{enumerate}
\item Note that the equivalence between (ii) and (iii) in Theorem~\ref{eq tilt cotilti shape} recovers Rickard's derived Morita theory for rings (\cite[Theorem 6.4]{Rick}). For this purpose it is enough to recall from Corollary~\ref{corheartmodulecat} that compact tilting objects yield hearts which are module categories. 
\item Let $R$ be a ring and $T$ be a large $n$-tilting (respectively, $n$-cotilting) $R$-module (see Example~\ref{easyexample}(vi)). Note that by Proposition \ref{bdd silting over rings}, $T$ is a bounded tilting (respectively, cotilting) object in $\mathsf{D}(R)$. Then, by Theorem~\ref{eq tilt cotilti shape} it follows that there is a triangle equivalence between $\mathsf{D}^b(\Hcal_T)$ and $\mathsf{D}^b(R)$. In the n-cotilting case, this is a bounded version of \cite[Theorem 5.21]{Stovicek}.

\item Since Grothendieck categories have injective cogenerators and their derived categories are TR5*, the above theorem covers a \textit{derived Morita theory for Grothendieck categories}. Indeed, if the unbounded derived categories of two Grothendieck categories are equivalent via a restrictable equivalence, then one of them is the heart of a t-structure associated to a bounded cotilting object in the derived category of the other. Moreover, the realisation functor associated to this bounded cotilting object yields an equivalence of bounded derived categories.
\end{enumerate}
\end{exam}

\begin{rem}
Let $\A$ and $\B$ be as in Theorem~\ref{eq tilt cotilti shape}. 
\begin{enumerate}
\item If $\A$ has exact coproducts, then there is a triangle equivalence $\phi\colon\mathsf{D}^b(\B)\lxr\mathsf{D}^b(\A)$ if and only if there is a tilting object $M$ in $\mathsf{D}^b(\A)$ such that $\Hcal_M\cong\B$. One direction is clear from the proof of (i)$\Longrightarrow$(ii) above. For the converse, note that since $\A$ has exact coproducts, $\Add(M)$ is contained in $\mathsf{D}^b(\A)$ and, therefore, the associated realisation functor yields an equivalence as in the proof of Proposition~\ref{der equival}. The dual result follows analogously. Note, however, that the relation between (co)silting objects in $\mathsf{D}^b(\A)$ and bounded silting objects in $\mathsf{D}(\A)$ is not clear. The problem here lies in extending t-structures from bounded to unbounded derived categories.
\item As discussed at the end of Section \ref{sec 3}, the question of whether (iii) implies (i) remains in general open. We will see, however, that in some cases we can guarantee the extendability of realisation functors (see Theorem \ref{standard f-lifts} and Remark \ref{rem extendable k-proj}).
\end{enumerate}
\end{rem}

We now briefly discuss an application of the above theorem to representation theory of infinite quivers, proving a version of APR-tilting in this setting. Recall the notation set up before Example \ref{unb silting}. The intuition from the theory of finite dimensional algebras leads us to think that the BGP-reflection functors on sources and sinks should provide derived equivalences. For infinite quivers, this cannot be achieved through the endomorphism ring of a tilting object (the reflected category cannot be regarded as a unital ring), but rather through the heart of a tilting object. We refer to \cite{AHV} for a detailed discussion of reflection functors and derived equivalences in the setting of infinite quivers.

Let $Q$ be a quiver (possibly infinite) with no loops nor cycles. We assume that $Q$ is locally finite, i.e. that each vertex has only finitely many incoming and outgoing arrows. For a source $k$ in $Q$, define $\mu_k(Q)$ to be the quiver obtained from $Q$ by reversing the direction of every arrow starting in $k$ and keeping the remaining vertices and arrows as in $Q$. We show the following fact (compare with \cite[Theorem 3.19]{AHV}).

\begin{prop}
If $k$ is a source of a locally finite quiver $Q$, then there is a triangle equivalence between $\mathsf{D}^b(\mathcal{M}(\mathbb{K}Q))$ and $\mathsf{D}^b(\mathcal{M}(\mathbb{K}\mu_k(Q)))$.
\end{prop}
\begin{proof}
Let $Q_0$ be the set of vertices in $Q$ and let $I:=\{i\in Q_0 \ | \ \Hom_{\mathbb{K}Q}(k,i)\neq 0\}$. Let $R_k$ denote the set of arrows from $k$ to some vertex in $I$. Since $Q$ is locally finite, $R_k$ is finite. For an arrow $\alpha$ in $\mathbb{K}Q$, denote by $t(\alpha)$ the target of $\alpha$. Consider the naturally induced map 
$$\phi:P_k\lxr \bigoplus\limits_{\alpha\in R_k}P_{t(\alpha)}$$
and let $C$ denote its cokernel in $\mathcal{M}(\mathbb{K}Q)$. Note that $\phi$ is a left $\Add(\oplus_{j\neq k}P_j)$-approximation of $P_k$, i.e. any map from $P_k$ to an object in $\Add(\oplus_{j\neq k}P_j)$ must factor through $\phi$. We will check that $T:=C\oplus (\oplus_{j\neq k}P_j)$ is a bounded tilting object in $\mathsf{D}(\mathcal{M}(\mathbb{K}Q))$. Since $\phi$ is a monomorphism and the sum of all indecomposable projectives is a generator in $\mathsf{D}(\mathcal{M}(\mathbb{K}Q))$, it is easy to check that also $T$ is a generator of $\mathsf{D}(\mathcal{M}(\mathbb{K}Q))$. Since $T$ is a a direct sum of finitely presented objects, it is clear that $T^{\perp_{>0}}$ is closed under coproducts. Furthermore, since $T$ has projective dimension 1, it only remains to show that $\Ext_{\mathcal{M}(\mathbb{K}Q)}^1(T,T)=0$, i.e. to check that 
$${\Ext}^1_{\mathcal{M}(\mathbb{K}Q)}(C,C)=0={\Ext}^1_{\mathcal{M}(\mathbb{K}Q)}(C,\bigoplus\limits_{j\neq k}P_j).$$
The first equality follows from applying $\Hom_{\mathcal{M}(\mathbb{K}Q)}(-,C)$ to the short exact sequence defined by $\phi$, using the projectivity of $P_k$ and the fact that $\phi$ is a left $\Add(\oplus_{j\neq k}P_j)$-approximation. The second one follows from applying $\Hom_{\mathcal{M}(\mathbb{K}Q)}(-,\oplus_{j\neq k}P_j)$ to the same sequence and using, once again, the approximation properties of $\phi$. It can also be checked that the object $T$ is a bounded tilting object (the associated t-structure is the HRS-tilt  with respect to the torsion pair $(T^{\perp_1},T^\circ)$ in $\mathcal{M}(\mathbb{K}Q)$ - see Example \ref{easyexample}(iv)). Then the realisation functor
$$\mathsf{real}_T\colon \mathsf{D}^b(\mathcal{H}_T)\lxr \mathsf{D}^b(\mathcal{M}(\mathbb{K}Q))$$
is an equivalence. It remains to show that $\mathcal{M}(\mathbb{K}\mu_k(Q))$ is equivalent to $\mathcal{H}_T$. An equivalence $\psi$ from $\mathbb{K}\mu_k(Q)$ to $\Hcal_T$ can be defined by setting $\psi(P_j)=P_j$, for all $j\neq k$ and $\psi(P_k)=C$. By definition of $C$, the $\Hom$-spaces are preserved and $\psi$ extends to the whole category since it is defined on a projective generator. Since $T$ is a projective generator in the heart, the functor so defined is an equivalence, as wanted.
\end{proof} 

Theorem \ref{eq tilt cotilti shape} leads us to discuss a derived invariant which is well-understood for rings: the finiteness of global dimension. This invariant generalises to the setting of abelian categories with a projective generator or an injective cogenerator. Recall that an abelian category $\A$ has \textbf{finite global dimension} if there is a positive integer $n$ such that the Yoneda Ext functor $\Ext^n_\A(-,-)$ is identically zero. Whenever $\A$ has a projective generator or an injective cogenerator, the following is a well-known lemma, which we prove for convenience of the reader.

\begin{lem}\label{fte gldim general}
Let $\A$ be a cocomplete (respectively, complete) abelian category with a projective generator $P$ (respectively, an injective cogenerator $E$). The following statements are equivalent.
\begin{enumerate}
\item $\A$ has finite global dimension;
\item The smallest thick subcategory of $\mathsf{D}(\A)$ containing $\Add(P)$ (respectively, $\Prod(E)$) is $\mathsf{D}^b(\A)$.
\end{enumerate}
\end{lem}
\begin{proof}
We discuss the case of $\A$ with a projective generator; the injective cogenerator case is dual.

(i)$\Longrightarrow$(ii): Note that $\mathsf{D}^b(\A)$ is the smallest thick subcategory containing $\A$. So it suffices to show that any object in $\A$ lies in the smallest thick subcategory containing $\Add(P)$, which is $\mathsf{K}^b(\Add(P))$. Let $X$ be an object of $\A$ and consider a projective resolution of $X$: $(Q^i,d^i)_{i\leq 0}$. Let $f$ be an epimorphism $P^{(I)}\rightarrow \Ker{d^{-n}}$, yielding an exact sequence of projective objects
\[
\xymatrix{P^{(I)}\ar[r]^f&Q^{-n}\ar[r]^{d^{-n}}&\cdots\ar[r]& Q^{0}\ar[r]^{d^{0}}&X}
\]
which must then split at some point by (1). Thus, $X$ admits a finite projective resolution, as wanted.

(ii)$\Longrightarrow$(i): Since $\mathsf{K}^b(\Add(P))=\mathsf{D}^b(\A)$, it follows that every object of $\A$ admits a finite projective resolution. Since Yoneda Ext-groups can be computed by projective resolutions in the first component (see, for example, \cite[III.6.14]{GM}) it only remains to show that there is a uniform choice of integer $n$ for all objects in $\A$. Suppose that this is not the case, i.e. that for any $n$ in $\mathbb{N}$, there is an object $X_n$ in $\A$ with projective dimension greater or equal than $n$. Since $\A$ is cocomplete, considering the coproduct of the family $(X_n)_{n\in\mathbb{N}}$ would yield an object of infinite projective dimension, contradicting our assumption.
\end{proof}

\begin{prop}\label{fte gldim}
Let $\A$ be an abelian category with a projective generator (respectively, injective cogenerator) and suppose that $\mathsf{D}(\A)$ is TR5 (respectively, TR5*). Let $\mathbb{T}$ be a t-structure in $\mathsf{D}(\A)$ satsifying the equivalent conditions of Lemma \ref{eq bdd} and suppose that the realisation functor $\mathsf{real}_\mathbb{T}\colon \mathsf{D}^b(\Hcal(\mathbb{T}))\lxr \mathsf{D}(\A)$ is fully faithful.
\begin{enumerate}
\item If $\A$ has finite global dimension, then so does $\Hcal(\mathbb{T})$.
\item If $\mathbb{T}=(M^{\perp_{>0}},M^{\perp_{<0}})$ for a bounded tilting (respectively, cotilting) object $M$ in $\mathsf{D}(\A)$, then the following are equivalent: 
\begin{enumerate}
\item $\A$ has finite global dimension;
\item $\Hcal_M$ has finite global dimension;
\item The smallest thick subcategory of $\mathsf{D}(\A)$ containing $\Add(M)$ (respectively, $\Prod(M)$) is $\mathsf{D}^b(\A)$.
\end{enumerate}
\end{enumerate}
\end{prop}
\begin{proof}
First note that if $\mathsf{D}(\A)$ is TR5 (respectively, TR5*), then $\A$ is cocomplete (respectively, complete), by Lemma \ref{facts silting t-structures}. Assume now that $\mathsf{D}(\A)$ is TR5 and that $\A$ has a projective generator $P$ (the other statement follows by dualising the arguments).

(i) Since $\mathbb{T}$ satisfies the equivalent conditions of Lemma \ref{eq bdd}, one may consider the induced functor $\mathsf{real}_\mathbb{T}\colon \mathsf{D}^b(\Hcal(\mathbb{T}))\lxr \mathsf{D}^b(\A)$ which is, by assumption, a triangle equivalence (see also Theorem \ref{real}). Since $\Hcal(\mathbb{T})$ lies in $\mathsf{D}^b(\A)$ and $\A$ has finite global dimension, $\Hcal(\mathbb{T})$ is contained in the smallest thick subcategory generated by $P$, i.e. in $\mathsf{K}^b(\Add(P))$ (see Lemma \ref{fte gldim general}). Hence, for any objects $X$ and $Y$ in $\Hcal(\mathbb{T})$, $\Ext_{\Hcal(\mathbb{T})}^n(X,Y)\cong\Hom_{\mathsf{D}(\A)}(X,Y[n])$ must vanish for $n\gg0$ (as before, we identify $X$ and $Y$ with their images under $\mathsf{real}_\mathbb{T}$, since this functor acts as the identity on $\Hcal(\mathbb{T})$ by definition).

Suppose now that $\Hcal(\mathbb{T})$ does not have finite global dimension, i.e. that there are sequences of objects $(X_i)_{i\in\mathbb{N}}$ and $(Y_i)_{i\in\mathbb{N}}$ in $\Hcal(\mathbb{T})$ such that $\Ext_{\Hcal(\mathbb{T})}^{\geq i}(X_i,Y_i)\neq 0$.  By Lemma \ref{facts silting t-structures}, the heart $\Hcal(\mathbb{T})$ is a cocomplete abelian category and the coproducts in $\mathcal{H}(\mathbb{T})$, here denoted with the symbol $\oplus$, are computed by using the truncation $\tau^{\geq 0}_\mathbb{T}$ on the coproduct available in $\mathsf{D}(\A)$, here denoted with the symbol $\coprod$. We will show that $\Ext_{\Hcal(\mathbb{T})}^{\geq n}(\bigoplus\limits_{i\in\mathbb{N}}X_i,\bigoplus\limits_{i\in\mathbb{N}}Y_i)\neq 0$, for all $n\in\mathbb{N}$, thus reaching a contradiction with the previous paragraph. Note that, since $\tau^{\geq 0}_\mathbb{T}$ is left adjoint to the inclusion of the coaisle in $\mathsf{D}(\A)$, we have that 
$$\mathsf{Ext}_{\Hcal(\mathbb{T})}^{\geq n}(\bigoplus\limits_{i\in\mathbb{N}}X_i,\bigoplus\limits_{i\in\mathbb{N}}Y_i)\cong \Hom_{\mathsf{D}(\A)}(\coprod_{i\in\mathbb{N}}X_i,\bigoplus\limits_{i\in\mathbb{N}}Y_i[\geq n])\cong \prod_{i\in\mathbb{N}}\Hom_{\mathsf{D}(\A)}(X_i,\bigoplus\limits_{i\in\mathbb{N}}Y_i[\geq n])$$
for any $n$ in $\mathbb{N}$. The heart $\Hcal(\mathbb{T})$ admits finite products and they coincide with finite coproducts - and these biproducts indeed coincide with those of $\mathsf{D}(\A)$ (since both $\mathbb{T}^{\leq 0}$ and $\mathbb{T}^{\geq 0}$ are closed under finite biproducts). Thus, if we write, for some integer $n$, 
$$\bigoplus\limits_{i\in\mathbb{N}}Y_i=Y_n\times \bigoplus\limits_{i\in\mathbb{N}\setminus\{n\}}Y_i,$$ 
since $\Hom_{\mathsf{D}(\A)}(X_n,-[\geq n])$ commutes with products and since by assumption $\Hom_{\mathsf{D}(\A)}(X_n,Y_n[\geq n])\neq 0$, we get that $\Hom_{\mathsf{D}(\A)}(X_n,\bigoplus\limits_{i\in\mathbb{N}}Y_i[\geq n])\neq 0$.

(ii) (a)$\Longrightarrow$(b): This follows directly from (i).

(b)$\Longrightarrow$(a): Let $N=\mathsf{real}_M^{-1}(P)$. Then $N$ is a tilting object in $\mathsf{D}(\Hcal_M)$ whose heart is equivalent to $\A$. Since, by assumption, $\Hcal_M$ has finite global dimension, we can apply the statement (a)$\Longrightarrow$ (b) exchanging the roles of $\A$ and $\Hcal_M$.

(b)$\Longleftrightarrow$(c): This statement follows from a combined application of (i) and Lemma \ref{fte gldim general}, since $M$ is a projective generator of $\Hcal_M$ (see Proposition \ref{gen cogen heart}). 
\end{proof}

\subsection{Standard forms} 
So far, our discussion of derived equivalences has mostly been concerned with their \textit{existence}. We would like now to discuss their \textit{shape}, i.e. their explicit description as functors. Our approach is in part motivated by Proposition \ref{everything is real}. In the context of derived equivalences of rings, this problem was addressed in \cite{Rick2} for algebras over a commutative ring $\mathbb{K}$, which are projective as $\mathbb{K}$-modules (for simplicity, we will call such algebras \textit{projective $\mathbb{K}$-algebras}). Therein, a partial answer to the problem is presented through the concept of \textit{equivalences of standard type}. In this subsection we will often refer to \textit{tilting complexes} in the original version of the concept, as defined in \cite{Rick}. Note, however, that these are precisely the compact tilting objects in the derived category of a ring, following Definition \ref{defncosiltcotilt}.

\begin{rem}\label{identification compact}
As seen in Corollary \ref{corheartmodulecat}, given a compact silting object $M$, there is an equivalence of abelian categories $\Hom_{\Hcal_M}(\textsf{H}^0_M(M),-)\colon\Hcal_M\lxr \Mod{\End_{\T}(\textsf{H}^{ \, 0}_M(M))}$. In what follows, we identify these two categories without mention to the equivalence functor. In particular, given a tilting complex $T$ in the derived category of a ring $A$, the realisation functor $\mathsf{real}_T$ will be regarded, via this identification, as a functor $\mathsf{D}^b(\End_{\mathsf{D}(A)}(T))\lxr \mathsf{D}^b(A)$.
\end{rem}

\begin{defn}
Let $\mathbb{K}$ be a commutative ring and let $A$ and $B$ be projective $\mathbb{K}$-algebras. We say that an equivalence $\phi\colon \mathsf{D}^b(B)\lxr \mathsf{D}^b(A)$ is \textbf{of standard type} if there is a complex of $B$-$A$-bimodules such that $\phi$ is naturally equivalent to $-\otimes_B^\mathbb{L}X$. Such an object $X$ is called a \textbf{two-sided tilting complex}.
\end{defn}

\begin{exam}\label{exam stand type}
Let $\mathbb{K}$ be a commutative ring and $A$ and $B$ two projective $\mathbb{K}$-algebras. Any exact equivalence $F\colon\Mod{B}\lxr \Mod{A}$ is a tensor product with a bimodule. Hence, its derived functor is a standard equivalence of derived categories. 
\end{exam}

The two-sided tilting complex $X$ in the above definition, when seen as an object in $\mathsf{D}^b(A)$, is a tilting complex $T:=X_A$.  It is known that, in the derived category of $B$-$A$-bimodules, $X$ can be chosen such that both $X_A$ and ${}_BX$ are complexes of projective modules and such that $X_A$ is still isomorphic to $T$ in $\mathsf{D}^b(A)$ (\cite[Proposition 6.4.4]{Z}). Moreover, Rickard proved in \cite{Rick2} the following result on the existence of equivalences of standard type (see also \cite[Theorem]{Keller}).

\begin{thm}\textnormal{\cite[Corollary 3.5]{Rick2}}\label{Rick stand type} 
Let $\mathbb{K}$ be a commutative ring and $A$ and $B$ two projective $\mathbb{K}$-algebras. If $\mathsf{D}^b(B)$ is triangle equivalent to $\mathsf{D}^b(A)$, then there is a two-sided tilting complex $_{B}X_A$ and a derived equivalence of standard type $-\otimes_B^{\mathbb{L}}X\colon\mathsf{D}^b(B)\lxr\mathsf{D}^b(A)$.
\end{thm}

Motivated by Proposition \ref{everything is real}, we have the following result (Theorem B in the introduction).

\begin{thm}\label{standard f-lifts}
Let $\mathbb{K}$ be a commutative ring, $A$ and $B$ projective $\mathbb{K}$-algebras and $T$ a compact tilting object in $\mathsf{D}(A)$ such that $\End_{\mathsf{D}(A)}(T)\cong B$. Then the functor $\mathsf{real}_T$ is an equivalence of standard type. Moreover, a triangle equivalence $\psi\colon \mathsf{D}^b(B)\lxr \mathsf{D}^b(A)$ is of standard type if and only if $\psi$ admits an f-lifting $\Psi\colon \mathsf{DF}^b(B)\lxr \mathsf{DF}^b(A)$ to the filtered bounded derived categories.
\end{thm}
\begin{proof}
Since $T$ is a compact tilting object with endomorphism ring $B$, it follows from classical derived Morita theory (\cite{Rick}) that $\mathsf{D}^b(B)$ is equivalent to $\mathsf{D}^b(A)$. From Theorem \ref{Rick stand type}, there is a two-sided tilting complex $_BX_A$ (such that $X_A\cong T$) and an equivalence of standard type ${\phi:=-\otimes^\mathbb{L}_B X\colon \mathsf{D}^b(B)\lxr \mathsf{D}^b(A)}$. We first show that $\phi$ admits an f-lifting to the filtered bounded derived categories and then argue with Theorem \ref{lift implies diagram} to prove that $\mathsf{real}_T$ is of standard type.

As mentioned above, the two-sided tilting complex $X$ can be chosen such that ${}_BX$ is a complex of projective $B$-modules. It is then easy to check that given a monomorphism $f\colon Y\lxr Z$ in the category of complexes $\mathsf{C}^b(B)$, $f\otimes_BX$ is still a monomorphism. Then, for an object $(Y,F)$ in $\mathsf{CF}^b(B)$, there is an obvious filtration $F\otimes_BX$ on the complex $Y\otimes_BX$ induced by $F$, i.e. $(F\otimes_BX)_nY\otimes_BX:=F_nY\otimes_BX$. It follows that there is a functor $\psi\colon \mathsf{CF}^b(B)\lxr \mathsf{CF}^b(A)$ such that $\psi(Y,F)=(Y\otimes_BX,F\otimes_BX)$. Again, since $-\otimes_BX$ is an exact functor between the categories of complexes, the functor $\psi$ sends filtered quasi-isomorphisms in $\mathsf{CF}^b(B)$ to filtered quasi-isomorphisms in $\mathsf{CF}^b(A)$, thus inducing a functor $\Phi\colon \mathsf{D}\mathsf{F}^b(B)\lxr \mathsf{D}\mathsf{F}^b(A)$. The functor $\Phi$ is clearly an f-functor and it is an f-lifting of $\phi$ to the filtered bounded derived categories.

Now, from Theorem \ref{lift implies diagram} and Proposition \ref{standard identity}, it follows that $\phi\cong\mathsf{real}_{T}\circ \mathsf{D}^b(\phi^0)$, where $\phi^0$ is the induced exact functor of abelian categories $\phi^0\colon \Mod{B}\lxr \phi(\Mod{B})$. Identifying $\phi(\Mod{B})$ with a module category $\Mod{B^\prime}$, Example \ref{exam stand type} yields that $\mathsf{D}^b(\phi^0)$ is a derived equivalence of standard type. Since the inverse of a derived equivalence of standard type and the composition of two derived equivalences of stardard type are both of standard type (\cite[Proposition 4.1]{Rick2}) the claim follows.

It remains to show that any triangle equivalence $\psi\colon\mathsf{D}^b(B)\lxr \mathsf{D}^b(A)$ admitting an f-lifting to the filtered bounded derived categories is of standard type. Indeed, under this assumption, from Theorem \ref{lift implies diagram} and Proposition \ref{standard identity}, we have that $\psi\cong \mathsf{real}_{\psi(B)}\circ \mathsf{D}^b(\psi^0)$. Since, as shown in the above paragraphs, both $\mathsf{real}_{\psi(B)}$ and $\mathsf{D}^b(\psi^0)$ are equivalences of standard type, it follows from \cite[Proposition 4.1]{Rick2} that the composition $\psi$ is of standard type.
\end{proof}

\begin{rem}\label{rem extendable k-proj}
Note that if the conditions in the above theorem are satisfied, then $\mathsf{real}_T$ is in fact an extendable equivalence (see Definition \ref{defn restrict extend}). This follows from the fact that derived equivalences of standard type are extendable (\cite{Keller:dg}).
\end{rem}

There is another setting where a suitable notion of \textit{standard derived equivalences} exists: derived categories of coherent sheaves. We recall the definition of a Fourier-Mukai transform. For the rest of this section, let $\mathbb{K}$ be an algebraically closed field of characteristic zero. Given a smooth projective variety $X$ over $\mathbb{K}$, we denote by $\mathsf{D}^b(\mathsf{coh}(X))$ the bounded derived category of coherent sheaves over $X$.

\begin{defn}
Let $X$ and $Y$ be smooth projective varieties over $\mathbb{K}$ and let $q\colon X\times Y\lxr X$ and $p\colon X\times Y\lxr Y$ be the natural projections. For any object $P$ in $\mathsf{D}^b(\mathsf{coh}(X\times Y))$ we define the \textbf{Fourier-Mukai transform with kernel P} as the functor $\Theta_P\colon \mathsf{D}^b(\mathsf{coh}(X))\lxr\mathsf{D}^b(\mathsf{coh}(Y))$ which sends an object $Z$ to $\theta_P(Z):=\mathbb{R}p_*(q^*Z\otimes_{X\times Y}^{\mathbb{L}}P)$, where $p_*$ and $q^*$ represent, respectively, the (left exact) direct image and the (exact) pullback functors defined on coherent sheaves.
\end{defn}

Note that the functor $q^*$ is exact since $q$ is a flat morphism (see also \cite[Chapter 5]{Huy}). The following theorem of Orlov provides a standard form for equivalences between derived categories of coherent sheaves.

\begin{thm}\textnormal{\cite{Or}}
Let $X$ and $Y$ be two smooth projective varieties over $\mathbb{K}$. Any fully faithful triangle functor $F\colon \mathsf{D}^b(\mathsf{coh}(X))\lxr \mathsf{D}^b(\mathsf{coh}(Y))$ admiting left and right adjoints is naturally equivalent to a Fourier-Mukai transform $\Theta_P$, for an object $P$ in $\mathsf{D}^b(\mathsf{coh}(X\times Y))$, unique up to isomorphism.
\end{thm}

\begin{prop}\label{FM is real}
Let $\theta_P\colon \mathsf{D}^b(\mathsf{coh}(X))\lxr\mathsf{D}^b(\mathsf{coh}(Y))$ be a Fourier-Mukai transform between two smooth projective varieties $X$ and $Y$ over $\mathbb{K}$. Then $\theta_P$ is naturally equivalent to $\mathsf{real}_{\mathbb{T}}\circ\mathsf{D}^b(\theta_P^0)$, where $\mathbb{T}=(\theta_P(\mathbb{D}^{\leq 0}),\theta_P(\mathbb{D}^{\geq 0}))$.
\end{prop}
\begin{proof}
The technique employed here is analogous to the ones before. First we observe that $\theta_P$ lifts to the filtered bounded derived categories of $X$ and $Y$. This follows from the fact that $\theta_P$ is the composition of functors that admit f-liftings. In fact, since $\mathbb{R}p_*$ and $q^*$ are derived functors of functors between abelian categories, if follows that they admit f-liftings to the filtered bounded derived categories of coherent sheaves of $X$ and $Y$ (see \cite[Example A2]{B}). Moreover, also the derived tensor functor admits an f-lifting. Indeed, since $X$ is projective and smooth, any object in $\mathsf{D}^b(\mathsf{coh}(X))$ is isomorphic to a \textit{perfect} complex, i.e. to a complex of locally free sheaves. In particular, we may assume without loss of generality that $P$ is such a complex. It is then easy to check that tensoring with $P$ is an exact functor in the category of complexes of coherent sheaves and, hence, it induces a functor between categories of filtered complexes. The argument then follows analogously to the proof of Theorem \ref{standard f-lifts}, using Theorem \ref{lift implies diagram} and Proposition \ref{standard identity}.
\end{proof}

Finally, we remark that the t-structure $\mathbb{T}$ in the above proposition is, in fact, the restriction to $\mathsf{D}^b(\mathsf{coh}(X))$ of a cotilting t-structure in $\mathsf{D}(\mathsf{Qcoh}(X))$, the unbounded derived category of quasicoherent sheaves over $X$. The functor $\theta_P$ can be extended to a functor $\Theta_P\colon \mathsf{D}(\mathsf{Qcoh}(X))\lxr \mathsf{D}(\mathsf{Qcoh}(Y))$, by using the same formula and this extension still has both left and right adjoints (which are again Fourier-Mukai transforms, see \cite{Muk} and \cite[Section 2.2 and 3.1]{AL}). Hence, $\Theta_P$ preserves coproducts. Since the varieties are smooth, the compact objects in $\mathsf{D}(\mathsf{Qcoh}(X))$ are precisely those in $\mathsf{D}^b(\mathsf{coh}(X))$. Thus, $\Theta_P$ is a coproduct preserving triangle functor between compactly generated triangulated categories which restricts to an equivalence on compact objects. By \cite[Lemma 3.3]{Schwede}, this means that also $\Theta_P$ is a triangle equivalence. If $E$ is an injective cogenerator of $\mathsf{Qcoh}(X)$, then the t-structure $\mathbb{T}$ in the proof of the above proposition is the restriction to $\mathsf{D}^b(\mathsf{coh}(X))$ of the t-structure associated to the cotilting object $\Theta_P(E)$. Informally, one may then say that \textit{Fourier-Mukai transforms are cotilting equivalences}.

\section{Recollements of derived categories}
\label{sectionrecollements}
Let $R$ be a ring and $e$ an idempotent element of $R$. As observed in \cite{CPSmemoirs}, the recollement of the module category $\Mod{R}$ induced by the idempotent element $e$ (see Example \ref{examrecol}) always induces a recollement of triangulated categories $\mathsf{R}_{\mathsf{tr}}(\Ker{\textsf{D}(e(-))}, \textsf{D}(\Mod{R}), \textsf{D}(\Mod{eRe}))$, where $\textsf{D}(e(-))\colon \textsf{D}(\Mod{R})\lxr \textsf{D}(\Mod{eRe})$ is the derived functor induced by the exact functor $e(-)\colon \Mod{R}\lxr \Mod{eRe}$. It also follows from \cite{CPSmemoirs} that $\Ker{\textsf{D}(e(-))}$ is triangle equivalent with $\textsf{D}(\Mod{R/ReR})$ if and only if the natural map $R\lxr R/ReR$ is a homological ring epimorphism. In this section we generalise this result, investigating when do recollements of abelian categories $\mathsf{R}_{\mathsf{ab}}(\B,\A,\C)$ induce recollements of the associated unbounded derived categories $\mathsf{R}_{\mathsf{tr}}({\mathsf{D}}(\B),{\mathsf{D}}(\A),{\mathsf{D}}(\C))$. Moreover, we will discuss when is a recollement of derived categories equivalent to a derived version of a recollement of abelian categories. This is intimately related with glueing tilting objects. Here is an informal overview of this section.\\

\textbf{Subsection 6.1: Homological embeddings}
\begin{itemize}
\item We prove in a rather general context that fully faithful functors between abelian categories preserving Yoneda extensions yield fully faithful derived functors between unbounded derived categories (Theorem \ref{prophomolemb}). 
\item We present equivalent conditions for a recollement of abelian categories to induce a recollement of unbounded derived categories, generalising the analogous result for rings in \cite{CPSmemoirs} (Theorem \ref{thmrecolderivedcat}).
\end{itemize}

\textbf{Subsection 6.2: Glueing tilting and equivalences of recollements}
\begin{itemize}
\item  Under some technical assumptions, we give an equivalent condition (in terms of glueing tilting t-structures) for a recollement of derived categories to be equivalent to the derived version of a recollement of abelian categories (Theorem \ref{criterium stratifying}). 
\item We specialise Theorem \ref{criterium stratifying} to recollements of derived categories of certain $\mathbb{K}$-algebras, where all the technical conditions are automatically satisfied (Corollary \ref{rec strat k-alg}). This statement reflects the main idea of this section, relating equivalences of recollements with glueing of tilting t-structures. 
\item We apply the above results to recollements of derived module categories of finite dimensional hereditary $\mathbb{K}$-algebras, over a field $\mathbb{K}$, answering a question posed by Xi (Theorem \ref{rec hereditary}).
\end{itemize}

\subsection{Homological embeddings}

In this section we will study exact fully faithful functors between abelian categories whose derived functors are fully faithful. Examples of these are well-known in representation theory: a ring epimorphism induces a fully faithful restriction of scalars functor, and its derived functor is fully faithful if and only if some homological conditions are satisfied (see Theorem \ref{hom ring epi}).

To build recollements of triangulated categories, we will often need to ensure the existence of adjoint pairs.  A powerful tool for this purpose is Brown representability.  Recall that a TR5 triangulated category $\T$ \textbf{satisfies Brown representability} if every cohomological functor $H\colon \T^{op}\lxr \Mod{\mathbb{Z}}$ which sends coproducts to products is representable (i.e. $H\cong \Hom_\T(-,X)$ for some $X$ in $\T$). 

\begin{thm}\textnormal{\cite[Theorem 8.4.4]{Neeman}}
\label{adjoints from Brown}
Let $\T$ and $\V$ be TR5 triangulated categories and suppose that $\T$ satisfies Brown representability. Then any coproduct-preserving functor $F\colon\T\lxr \V$ has a right adjoint.
\end{thm}

There are many examples of triangulated categories satisfying Brown representability, in particular some derived categories of abelian categories. The following concept will be used to provide examples (we refer to \cite[Section 2]{KN} for the terminology \textit{Milnor colimit} and \textit{Milnor limit}).

\begin{defn}
The derived category $\mathsf{D}(\A)$ of an abelian category $\A$ is  \textbf{left-complete} if $\mathsf{D}(\A)$ is TR5* and any object $X$ is isomorphic to a Milnor limit of its standard truncations $(\tau^{\geq n}X)_{n\leq 0}$ with respect to the canonical maps $\tau^{\geq n-1}X\lxr \tau^{\geq n}X$, for $n\leq 0$.
\end{defn}

\begin{rem}\label{rem completeness}
\begin{enumerate}
\item We can dually define the notion of a \textbf{right-complete} derived category. It is, however, easy to see that any TR5 derived category $\mathsf{D}(\A)$ is right-complete. This follows from \cite[Remark 24 and Lemma 64(iii)]{Mur} since, for any complex $X$, there is a quasi-isomorphism in $\mathsf{K}(\A)$ between the Milnor colimit of $(\tau^{\leq n}X)_{n\geq 0}$ and the direct limit of the same system in $\mathsf{C}(\A)$ (which is isomorphic to $X$).
\item Note that right and left-completeness can analogously be defined in any triangulated category $\T$ endowed with a t-structure $\mathbb{T}$, in which case we say that $\T$ is right or left-complete with respect to $\mathbb{T}$. Given a triangle equivalence $\phi\colon\mathsf{D}(\A)\lxr \mathsf{D}(\B)$, then $\mathsf{D}(\A)$ is right or left-complete if and only if $\mathsf{D}(\B)$ is left or right-complete with respect to the t-structure $\phi(\mathbb{D}_\A)=(\phi(\mathbb{D}_\A^{\leq 0}),\phi(\mathbb{D}_\A^{\geq 0}))$.
\end{enumerate}
\end{rem}

\begin{exam}\label{ex right left}
Let $\A$ be an abelian category.
\begin{enumerate}
\item If $\A$ has a projective generator and $\mathsf{D}(\A)$ is TR5*, then $\mathsf{D}(\A)$ is left-complete. In fact, it follows from \cite[Remark 27 and Lemma 67(i)]{Mur} that, for any complex $X$, there is a quasi-isomorphism in $\mathsf{K}(\A)$ between the inverse limit of $(\tau^{\geq n}X)_{n\leq 0}$ in the category of complexes (which is isomorphic to $X$) and the Milnor limit of the same system.
\item In particular, as a consequence of Lemma \ref{facts silting t-structures}(i) and \ref{gen cogen heart}, if $\mathsf{D}(\A)$ is TR5 (respectively, TR5*) and $M$ is a silting object in $\mathsf{D}(\A)$, then $\mathsf{D}(\Hcal_M)$ is right-complete (respectively, left-complete).
\item Assume that $\A$ has an injective cogenerator and finite global dimension. If $\mathsf{D}(\A)$ is TR5* and $M$ is a cotilting object in $\mathsf{D}(\A)$, then $\mathsf{D}(\Hcal_M)$ is left-complete (in particular, so is $\mathsf{D}(\A)$). This follows from $\Hcal_M$ having finite global dimension (Proposition \ref{fte gldim}(ii)) and from \cite[Theorem 1.3]{HX}. 
\item If $\A$ is a Grothendieck category, then $\mathsf{D}(\A)$ is right-complete but not always left-complete (\cite{Neeman2}). However, from item (ii) above, $\mathsf{D}(R)$ is left-complete for any ring $R$, and from \cite[Remark 3.3]{HX}, $\mathsf{D}(\mathsf{Qcoh}(X))$ is left-complete for any separated quasi-compact scheme $X$.
\end{enumerate}
\end{exam}

The next theorem recalls two examples of categories satisfying Brown representability.

\begin{thm}\label{ex Brown}
Let $\A$ be an abelian category.
\begin{enumerate}
\item\textnormal{\cite[Theorem 3.1]{Franke}} If $\A$ is Grothendieck, then $\mathsf{D}(\A)$ satisfies Brown representability.
\item\textnormal{\cite[Theorem 1.1]{Modoi}} If $\A$ has an injective cogenerator and $\mathsf{D}(\A)$ is left-complete, then $\mathsf{D}(\A)^{op}$ satisfies Brown representability.
\end{enumerate}
\end{thm}

Given a TR5 (respectively, TR5*) triangulated category, we say that a triangulated subcategory is \textbf{localising} (respectively, \textbf{colocalising}) if it is coproduct-closed (respectively, product-closed), i.e. if the inclusion functor commutes with coproducts (respectively, products). The following result is well-known, still we present a proof for convenience of the reader. 

\begin{lem}\label{lemconsBR} 
Let $\A$ be a Grothendieck abelian category.
\begin{enumerate}
\item\textnormal{\cite[Proposition 3.3]{Franke}}  If $\X$ is a localising subcategory of $\textsf{D}(\A)$ and $\A\subseteq \X$, then $\X=\textsf{D}(\A)$.
\item Assume that $\mathsf{D}(\A)$ is left-complete. If $\Y$ is a colocalising subcategory of $\textsf{D}(\A)$ and $\A\subseteq \Y$, then 
$\Y=\textsf{D}(\A)$.
\end{enumerate}
\begin{proof}
(i) This statement can also be shown using \cite[Theorem 5.7]{AJS}. Let $G$ be a generator of $\A$. It suffices to prove that the smallest localising subcategory $\mathcal{L}$ of $\mathsf{D}(\A)$ containing $G$ is the whole derived category. By \cite[Theorem 5.7]{AJS}, the inclusion of $\mathcal{L}$ in $\mathsf{D}(\A)$ has a right adjoint and, thus, the pair $(\mathcal{L},\mathcal{L}^{\perp})$ is a (stable) t-structure in $\mathsf{D}(\A)$. It remains to show that $\mathcal{L}^{\perp}=0$. Now, if $X$ lies in $\mathcal{L}^\perp$, then we have $\Hom_{\mathsf{D}(\A)}(G,X[i])=0$ for all $i$ in $\mathbb{Z}$, which from Lemma \ref{Lemma generator1}(i) implies that $X=0$.

(ii) This proof is essentially contained in \cite[Applications 2.4 and 2.4']{BN}. We prove that the smallest colocalising subcategory $\mathcal{C}$ containing the injective objects of $\A$ is $\mathsf{D}(\A)$. Clearly any bounded complex of injective objects lies in $\mathcal{C}$. Now every bounded below complex of injective objects lies in $\mathcal{C}$ is an inverse limit of bounded complexes in the category $\mathsf{C}(\A)$, where the maps in the inverse system are identities componentwise. This implies that every bounded below complex of injective objects is a Milnor limit of bounded ones (see for example \cite[Proposition 85]{Mur}) and, therefore, every such complex lies in $\mathcal{C}$. Finally, since every bounded below complex in $\mathsf{D}(\A)$ is quasi-isomorphic to a bounded below complex of injective objects, and since $\mathsf{D}(\A)$ is left-complete, we conclude that $\mathcal{C}=\mathsf{D}(\A)$.
\end{proof}
\end{lem}

\begin{rem}\label{prods}
By \cite[Proposition 8.4.6]{Neeman}, a triangulated category satisfying Brown representability is TR5*. Hence, as mentioned before, the derived category of a Grothendieck category is TR5*. Moreover, from \cite[Theorem 1.3]{HX}, if products in a Grothendieck category $\A$ have finite homological dimension (i.e. the $n$-th right derived functor vanishes for some $n>0$), then $\mathsf{D}(\A)$ is left-complete. In particular, this is satisfied if the Grothendieck category has finite global dimension (see also Example \ref{ex right left}).
\end{rem}

We now recall the definition of a homological embedding of abelian categories (see \cite{Psaroud:homolrecol}).

\begin{defn}
An exact functor $i\colon\B\lxr \A$ between abelian categories is called a {\bf homological embedding}, if the map $i_{X,Y}^n\colon\Ext_{\B}^n(X,Y) \lxr \Ext_{\A}^n(i(X),i(Y))$ is an isomorphism of the abelian groups of Yoneda extensions, for all $X, Y$ in $\B$ and for all $n\geq 0$.
\end{defn}

As mentioned earlier, homological ring epimorphisms are examples of homological embeddings. In that case, by Theorem \ref{hom ring epi}, the derived functor $f_*\colon \mathsf{D}(B)\lxr\mathsf{D}(A)$ is fully faithful. Our next theorem generalises this statement for homological embeddings $i\colon \B\lxr \A$ of abelian categories. Note that we do not assume $\B$ to be a Serre subcategory of $\A$ (compare \cite[Theorem 2.1]{Y} and \cite[Theorem 1.5]{HX}).

\begin{thm}
\label{prophomolemb}
Let $\A$ be an abelian category such that $\mathsf{D}(\A)$ is TR5 and TR5* and let $\B$ be a Grothendieck category such that $\mathsf{D}(\B)$ is left-complete. If $i\colon\B\lxr \A$ is an exact full embedding such that the derived functor $i\colon \mathsf{D}(\B)\lxr \mathsf{D}(\A)$ preserves products and coproducts, then the following statements are equivalent.
\begin{enumerate}
\item The functor $i\colon \B\lxr \A$ is a homological embedding.

\item The derived functor $i\colon \textsf{D}(\B)\lxr \textsf{D}(\A)$ is fully faithful.
\end{enumerate}
If, in addition, $\mathsf{D}(\A)$ is left-complete, then the statements \textnormal{(i)} and \textnormal{(ii)} are also equivalent to:
\begin{enumerate}
\item[(iii)] The derived functor $i$ induces a triangle equivalence between $\mathsf{D}(\B)$ and the full subcategory $\mathsf{D}_\B(\A)$ of $\mathsf{D}(\A)$, whose objects are complexes of objects in $\A$ with cohomologies in $i(\B)$.
\end{enumerate}
\begin{proof}
(i)$\Longrightarrow$(ii): Let $Y\in \textsf{D}(\B)$. We define the following full subcategory of $\textsf{D}(\B)$$\colon$
\[
\M_{Y} = \big\{X\in \textsf{D}(\B) \ | \ \alpha_{X,Y}\colon \Hom_{\textsf{D}(\B)}(X[k],Y) \stackrel{\iso}{\lxr} \Hom_{\textsf{D}(\A)}(i(X[k]),i(Y)), \forall k\in\mathbb Z \big\}
\]
where $\alpha_{X,Y}$ is the natural map of abelian groups induced by the functor $i\colon\mathsf{D}(\B)\lxr\mathsf{D}(\A)$. We show that $\M_{Y}$ is a localising subcategory of $\textsf{D}(\B)$. First, it is clear that $\M_{Y}$ is closed under shifts. For $K$ and $M$ in $\M_{Y}$ and a triangle $\Delta$ in $\mathsf{D}(\B)$ of the form 
\[
\xymatrix{
(\Delta): & K \ar[r] & L \ar[r]^{} & M \ar[r] & K[1] },
\]
consider the triangle $i(\Delta)$ in $\textsf{D}(\A)$. Applying to $\Delta$ and $i(\Delta)$ the cohomological functors $\Hom_{\textsf{D}(\B)}(-,Y)$ and $\Hom_{\textsf{D}(\A)}(-,i(Y))$, respectively, we obtain a commutative diagram with exact rows as follows$\colon$
\[
\xymatrix@C=0.3cm{
  \cdots \ar[r]^{} & \Hom_{\textsf{D}(\B)}(M[k],Y) \ar[d]_{\iso}^{\alpha_{M,Y}} \ar[r]^{} & \Hom_{\textsf{D}(\B)}(L[k],Y) \ar[d]_{}^{\alpha_{L,Y}} \ar[r]^{} & \Hom_{\textsf{D}(\B)}(K[k],Y) \ar[d]_{\iso}^{\alpha_{K,Y}} \ar[r] & \cdots  \\
   \cdots \ar[r]^{} & \Hom_{\textsf{D}(\A)}(i(M[k]),i(Y)) \ar[r]^{ }
  & \Hom_{\textsf{D}(\A)}(i(L[k]),i(Y)) \ar[r]^{} & \Hom_{\textsf{D}(\A)}(i(K[k]),i(Y)) \ar[r] & \cdots  } 
\]  
Hence, the map $\alpha_{L,Y}$ is an isomorphism and the object $L$ lies in $\M_{Y}$. 
Let $(X_i)_{i\in I}$, be a family of objects in $\M_{Y}$. We show that $\coprod_{i\in I} X_i$ lies in $\M_{Y}$, concluding that $\M_Y$ is localising in $\mathsf{\B}$. Since, by assumption, the derived functor $i\colon \textsf{D}(\B)\lxr \textsf{D}(\A)$ preserves coproducts, the following diagram commutes.
\[
\xymatrix@C=0.5cm{
   \Hom_{\textsf{D}(\B)}(\coprod_{i\in I}X_i,Y) \ar[d]_{\iso}^{} \ar[rrr]^{\alpha_{\coprod X_i,Y} \ \  \ \ \ \ \ } &&& \Hom_{\textsf{D}(\A)}(\coprod_{i\in I}\mi(X_i),\mi(Y)) \ar[d]_{\iso}^{} \\
   \prod_{i\in I} \Hom_{\textsf{D}(\B)}(X_i,Y) \ar[rrr]^{\prod \alpha_{X_i,Y} \ \ \ \ \ \ \ \ }_{\iso \ \ \ \ \ \ }
  &&& \prod_{i\in I} \Hom_{\textsf{D}(\A)}(\mi(X_i),\mi(Y))   }
\]
Then the map $\alpha_{\coprod X_i,Y}$ is an isomorphism and $\M_{Y}$ is a localising subcategory of $\textsf{D}(\B)$. In particular, for every $B$ in $\B$ the subcategory  $\M_{B}$ is localising in $\textsf{D}(\B)$ and $\B\subseteq \M_{B}$ since the functor $i\colon \B\lxr \A$ is a homological embedding. Then from Lemma \ref{lemconsBR}(i) we get that $\M_{B}=\textsf{D}(\B)$ for every $B$ in $\B$. 

Let $X$ be an object in $\textsf{D}(\B)$ and consider the full subcategory of $\textsf{D}(\B)\colon$
\[
{_{X}\M}=\big\{Y\in \textsf{D}(\B) \ | \ \alpha_{X, Y}\colon \Hom_{\textsf{D}(\B)}(X, Y[k]) \stackrel{\iso}{\lxr} \Hom_{\textsf{D}(\A)}(i(X),i(Y[k])), \forall k\in\mathbb Z \big\}.
\]
Then, dually to the argument above, it follows that ${_{X}\M}$ is a colocalising subcategory of $ \textsf{D}(\B)$. Since, for any $B$ in $\B$ we have $\M_{B}=\textsf{D}(\B)$ it follows that $\B\subseteq {_{X}\M}$ for any $X$ in $\mathsf{D}(\B)$. Since $\mathsf{D}(\B)$ is left-complete, Lemma \ref{lemconsBR}(ii) shows that ${_{X}\M} = \textsf{D}(\B)$. Hence, the derived functor $i\colon \textsf{D}(\B)\lxr \textsf{D}(\A)$ is fully faithful.

(ii)$\Longrightarrow$(i): Suppose that $i\colon \mathsf{D}(\B)\lxr\mathsf{D}(\A)$ is fully faithful and let $X$ and $Y$ be objects in $\B$ and $n\geq 0$. Then we have the following chain of natural isomorphisms$\colon$
\[{\Ext}_{\B}^n(X,Y) \iso \Hom_{\textsf{D}(\B)}(X,Y[n]) \iso \Hom_{\textsf{D}(\A)}(i(X),i(Y)[n]) \iso {\Ext}_{\A}^n(i(X),i(Y)) \]
thus showing that the functor $i\colon \A\lxr \B$ is a homological embedding. 

Assume now that also $\mathsf{D}(\A)$ is left-complete. Note that it is also right-complete by Remark~\ref{rem completeness}(i).

(iii)$\Longrightarrow$(ii): This is clear. 

(ii)$\Longrightarrow$(iii): Assume that the derived functor $i$ is fully faithful, i.e. that $i$ induces a triangle equivalence between $\mathsf{D}(\B)$ and $\X:=\Image{i}$. It is clear that $\X$ is a full extension-closed subcategory of $\mathsf{D}_\B(\A)$. First observe that $i(\B)$ lies in $\X$ and, hence, so does $\mathsf{D}_\B^b(\A)$, since every object  in $\mathsf{D}_\B^b(\A)$ can be obtained as a finite extension of objects in $i(\B)$. Since $\mathsf{D}(\A)$ is both left and right-complete, every object in $\mathsf{D}_\B(\A)$, $X$ can be expressed as a Milnor limit of a Milnor colimit of bounded complexes with cohomologies in $i(\B)$ (since standard truncations respect the cohomologies). Since $\mathsf{D}(\B)$ is TR5 and TR5* and $i$ preserves products and coproducts, $\X$ is closed under Milnor limits and Milnor colimits, thus finishing the proof. 
\end{proof}
\end{thm}

Motivated by Example~\ref{examrecol}, the following result gives necessary and sufficient conditions for certain recollements of abelian categories to induce recollements of unbounded derived categories. The key ingredient of the proof is the equivalence between $\mathsf{D}_\B(\A)$ and $\mathsf{D}(\B)$ established in Theorem~\ref{prophomolemb}. 

\begin{thm}
\label{thmrecolderivedcat}
Let $\textsf{R}_{\mathsf{ab}}(\B,\A,\C)$ be a recollement of abelian categories as in $(\ref{rec})$. Suppose that $\B$ is a Grothendieck category such that $\mathsf{D}(\B)$ is left-complete and suppose that $\mathsf{D}(\A)$ is TR5 and left-complete. The following statements are equivalent.
\begin{enumerate}
\item The functor $i_{*}\colon \B\lxr \A$ is a homological embedding and $i_*\colon\mathsf{D}(\B)\lxr\mathsf{D}(\A)$ commutes with products and coproducts.
\item There is a recollement of triangulated categories
\begin{equation}
\label{derivedrecollement}
\xymatrix@C=0.5cm{
{\mathsf{D}}(\B) \ar[rrr]^{i_*} &&&
{\mathsf{D}}(\A) \ar[rrr]^{j^*} \ar @/_1.5pc/[lll]_{i^*} \ar
 @/^1.5pc/[lll]_{i^{!}} &&& {\mathsf{D}}(\C).
\ar @/_1.5pc/[lll]_{j_{!}} \ar
 @/^1.5pc/[lll]_{j_{*}}
 }
\end{equation}
\end{enumerate}
\begin{proof}
(i)$\Longrightarrow$(ii): From the recollement $\textsf{R}_{\textsf{ab}}(\B,\A,\C)$ we have the exact sequence of abelian categories $0 \lxr \B \lxr \A \lxr \C \lxr 0$. Since the quotient functor $\A\lxr \C$ has a right (or left) adjoint, it follows that  
$0 \lxr {\textsf{D}}_{\B}(\A) \lxr {\textsf{D}}(\A) \lxr {\textsf{D}}(\C) \lxr 0$ is an exact sequence of triangulated categories, where ${\textsf{D}}_{\B}(\A)$ is the full subcategory of ${\textsf{D}}(\A)$ consisting of complexes whose cohomology lie in $i_*(\B)$ (see, for instance, \cite[Lemma 5.9]{Krause}). The functor $i_*\colon \B\lxr \A$ is a homological embedding, thus from Theorem~\ref{prophomolemb} the derived functor $i_*\colon \textsf{D}(\B)\lxr \textsf{D}(\A)$ is fully faithful and $\mathsf{D}_\B(\A)$ is equivalent to $\mathsf{D}(\B)$. This implies that $0 \lxr {\textsf{D}}(\B) \lxr {\textsf{D}}(\A) \lxr {\textsf{D}}(\C) \lxr 0$ is an exact sequence of triangulated categories. Since both $\mathsf{D}(\B)$ and $\mathsf{D}(\B)^{op}$ satisfy Brown representability (Theorem \ref{ex Brown}) and $i_*$ preserves both products and coproducts, it follows from Theorem \ref{adjoints from Brown} that $i_*\colon\mathsf{D}(\B)\lxr\mathsf{D}(\A)$ has both a left and a right adjoint. On the other hand, by \cite[Theorem $2.1$]{CPSstratif}, we obtain that the derived functor $j^*\colon {\mathsf{D}}(\A)\lxr {\mathsf{D}}(\C)$ admits a left and right adjoint and, therefore, we get a recollement of triangulated categories as wanted.

(ii)$\Longrightarrow$(i): Since $({\textsf{D}}(\B), {\textsf{D}}(\A), {\textsf{D}}(\C))$ is a recollement, the functor $i_*\colon \mathsf{D}(\B)\lxr \mathsf{D}(\A)$ is fully faithful. Then, by Theorem~\ref{prophomolemb} we infer that $i_*\colon \B\lxr \A$ is a homological embedding. Since the abelian category $\B$ is Grothendieck, it follows that the derived category $\mathsf{D}(\B)$ is TR5 and TR5*. Hence, the functor $i_*\colon \mathsf{D}(\B)\lxr \mathsf{D}(\A)$ preserves coproducts and products since it is both a left and a right adjoint.
\end{proof}
\end{thm}

\begin{rem}
Note that, in the above theorem, we do not describe explicitly how to obtain the adjoints of $i_*$ and $j^*$ in the recollement of derived categories. If, however, we assume some further conditions for the abelian categories $\A$ and $\C$, we can say more about these functors. In particular: 
\begin{itemize}
\item If $\A$ is Grothendieck, then $i^{!}$ is the right derived functor of the right adjoint of the inclusion functor $\B\lxr\A$. If, furthermore, $\A$ has enough projectives, then $i^*$ is the left derived functor of the left adjoint of the inclusion functor $\B\lxr\A$.
\item If $\C$ is Grothendieck, then $j_*$ is the right derived functor of the right adjoint of the quotient functor $\A\lxr \C$. If, furthermore, $\C$ has enough projectives, then $j_{!}$ is the left derived functor of the left adjoint of the quotient functor $\A\lxr\C$.
\end{itemize}
\end{rem}

Note also that, as a consequence of Theorem \ref{thmrecolderivedcat}, we obtain the result of \cite{CPSmemoirs} for recollements of derived module categories, already mentioned in Example \ref{examrecol}. More concretely, we have that given a ring $A$ and an idempotent element $e$ of $A$, the ring epimorphism $f\colon A\lxr A/AeA$ is homological if and only if there is a  recollement of triangulated categories of the form (\ref{stratifyingrecol}). 

\subsection{Glueing tilting and equivalences of recollements}
\label{subsectiongluequivrecol}
Our aim in this subsection is to identify which recollements of derived categories are equivalent to derived versions of recollements of abelian categories. We will provide an answer to this question in terms of the glueing of (co)tilting t-structures. 

We begin with two useful properties of derived recollements. On one hand, we discuss the glueing of standard t-structures along such a recollement, motivating a necessary condition towards our answer (Theorem \ref{criterium stratifying}) to the proposed problem. On the other hand, we restate in this context the exact sequence (\ref{exseqfcat}) of f-categories (see Proposition~\ref{f-categories over exact seq}) for filtered derived categories, which we use to prove a corollary of the main theorem (Corollary \ref{cor 2 strat}).
 
\begin{lem}
\label{standard glue to standard}
Let $\textsf{R}_{\mathsf{ab}}(\B,\A,\C)$ be a recollement of abelian categories as in $(\ref{rec})$. Suppose that $\B$ is a Grothendieck category such that $\mathsf{D}(\B)$ is left-complete and suppose that $\mathsf{D}(\A)$ is TR5 and left-complete. If the functor $i_*\colon \B\lxr\A$ is a homological embedding and  $i_*\colon\mathsf{D}(\B)\lxr\mathsf{D}(\A)$ commutes with products and coproducts, then the following statements hold.
\begin{enumerate}
\item The standard t-structures in $\mathsf{D}(\B)$ and $\mathsf{D}(\C)$ glue to the standard t-structure in $\mathsf{D}(\A)$ along the recollement (\ref{derivedrecollement}).
 
\item The f-categories over $\mathsf{D}(\B)$ and $\mathsf{D}(\C)$ induced by the filtered derived category $\mathsf{D}\mathsf{F}(\A)$ over $\mathsf{D}(\A)$ coincide, respectively, with the filtered derived categories $\mathsf{D}\mathsf{F}(\B)$ and $\mathsf{D}\mathsf{F}(\C)$. In particular,  there is an exact sequence of filtered derived categories:
\[
\xymatrix{
0 \ar[r] & \mathsf{D}\mathsf{F}(\B) \ar[r]^{} & \mathsf{D}\mathsf{F}(\A) \ar[r]^{} & \mathsf{D}\mathsf{F}(\C) \ar[r] & 0.
}
\] 
\end{enumerate} 
\end{lem}
\begin{proof}
(i) From Theorem~\ref{thmrecolderivedcat} we get a recollement of derived categories $\textsf{R}_{\mathsf{tr}}(\mathsf{D}(\B),\mathsf{D}(\A),\mathsf{D}(\B))$, see diagram $(\ref{derivedrecollement})$. Since $i_*\colon \mathsf{D}(\B)\lxr \mathsf{D}(\A)$ and $j^*\colon \mathsf{D}(\A)\lxr \mathsf{D}(\C)$ are the derived functors of the underlying exact functors $i_*\colon \B\lxr \A$ and $j^*\colon \A\lxr \C$, respectively, it follows that the derived functors $i_*$ and $j^*$ are t-exact with respect to the standard t-structures. Hence, the triangle functor $i^*\colon \mathsf{D}(\A)\lxr \mathsf{D}(\B)$ in the recollement diagram $(\ref{derivedrecollement})$ is right t-exact with respect to the standard t-structure. Then clearly we have $\mathbb{D}^{\leq 0}_\A\subseteq \{X\in \mathsf{D}(\A) \ | \ j^*(X)\in \mathbb{D}^{\leq 0}_\C \ \text{and} \ i^*(X)\in \mathbb{D}^{\leq 0}_\B \}$. Similarly, we get that 
$\mathbb{D}^{\geq 0}_\A\subseteq \{X\in \mathsf{D}(\A) \ | \ j^*(X)\in \mathbb{D}^{\geq 0}_\C \ \text{and} \ i^!(X)\in \mathbb{D}^{\geq 0}_\B \}.$ We infer that these t-structures coincide.

(ii) From Theorem \ref{thmrecolderivedcat}, there is an exact sequence of derived categories
\begin{equation}
\label{shortexactseqderlevel}
\xymatrix{0\ar[r]& \mathsf{D}(\B)\ar[r]^{i_*}& \mathsf{D}(\A)\ar[r]^{j^*}&\mathsf{D}(\C)\ar[r]& 0.}
\end{equation}
For an object $(X,F)$ in $\mathsf{D}\mathsf{F}(\B)$, where $F$ is a finite filtration of $X$, it follows from the exactness of $i_*\colon\B\lxr\A$ that $i_*(X)$ has an induced filtration $i_*(F)$ and $(i_*(X),i_*(F))$ lies in $\mathsf{D}\mathsf{F}(\A)$. Also, it is easy to see that $\textsf{gr}_{i_*(F)}^n(i_*(X))=i_*(\textsf{gr}^n_F(X))$ and, thus, it lies in $i_*(\mathsf{D}(\B))$. Hence, $i_*$ induces a functor 
\[
 i_*^\mathsf{F}\colon \mathsf{D}\mathsf{F}(\B)\lxr \Y:=\big\{(X,F)\in \mathsf{D}\mathsf{F}(\A) \ | \ \mathsf{gr}_F^n(X)\in i_*(\mathsf{D}(\B)) \ \text{for all} \ n\in \mathbb{Z} \big\}
\]
and one can check, using the exactness of $i_*\colon \B\lxr \A$, that $ i_*^\mathsf{F}$ is an f-functor (see Definition~\ref{defnffunctor}). 

We show that $i_*^\mathsf{F}$ is an equivalence. An easy induction argument shows that $ i_*^\mathsf{F}$ is essentially surjective. In fact, since $i_*$ is a fully faithful triangle functor, $i_*(\mathsf{D}(\B))$ is a triangulated subcategory. Hence, given $(Z,L)$ in $\Ycal$, each $L_nZ$ lies in $i_*(\mathsf{D}(\B))$, since $L_nZ$ is a finite extension of its $\textsf{gr}_L^k$-components, with $k\geq n$. From the fact that $i_*\colon\B\lxr \A$ induces a fully faithful exact functor between the categories of complexes, it then follows that $(Z,L)$ can be identified with an object in $\mathsf{DF}(\B)$. We show that $i_*^{\mathsf{F}}$ is faithful. Let $f\colon (X,F)\lxr (Y,G)$ be a morphism in $\mathsf{DF}(\B)$ such that $i_*^{\mathsf{F}}(f)=0$. Let $f$ be represented by a roof:
\begin{equation}\nonumber
\xymatrix{ & (Z,L) \ar[ld]_{c} \ar[rd]^{d} \\ 
(X,F) && (Y,G) }
\end{equation}
where $c$ is a filtered quasi-isomorphism in $\mathsf{KF}(\B)$. It is easy to see that the morphism $f=d\circ c^{-1}$ can be written as a sequence of roofs $(\dots, L_ad\circ(L_{a}c)^{-1},L_{a+1}d\circ(L_{a+1}c)^{-1},\dots)$ in $\mathsf{K}(\B)$ (see also Example~\ref{exam filtered derived cat}). Hence, it follows from the faithfulness of $i_*\colon\mathsf{D}(\B)\lxr \mathsf{D}(\A)$ that $f=0$. It remains to show that $ i_*^\mathsf{F}$ is full. Let $f\colon i_*^\mathsf{F}(X,F)\lxr i_*^\mathsf{F}(Y,G)$ be a morphism in $\Y$. The map $f$ can be represented by a roof similar to the one above, where the maps $c$ and $d$ now lie in $\mathsf{KF}(\A)$. We claim that the middle object $(Z,L)$ lies in the image of the functor $i_*^{\mathsf{F}}$. The map $c$ is a filtered quasi-isomorphism, i.e, for all $k$ in $\mathbb{Z}$ the maps $L_kc\colon L_kZ\lxr (i_*F)_ki_*(X)$ are quasi-isomorphisms in $\mathsf{K}(\A)$. Equivalently, the complexes $\mathsf{gr}_L(Z)$ and $\mathsf{gr}_{i_*F}(i_*X)$ are quasi-isomorphic (\cite[V.1.2]{Ill}) and, thus, $(Z,L)$ also lies in the image of $i_*^{\mathsf{F}}$. Writing $f$ as a sequence of roofs as before, the fullness of $i_*\colon\mathsf{D}(\B)\lxr \mathsf{D}(\A)$ guarantees that each component of $f$ lies in $i_*(\mathsf{D}(\B))$. Finally, it can be checked (using the faithfulness of $i_*\colon\mathsf{D}(\B)\lxr \mathsf{D}(\A)$) that the preimages of each component of $f$ under $i_*$ form a compatible sequence of morphisms in $\mathsf{D}(\B)$, i.e. there is a morphism in $\mathsf{DF}(\B)$ which is a preimage of $f$ under $ i_*^\mathsf{F}$.

We now show that the quotient f-category $\mathsf{DF}(\A)/\Y$ is equivalent to $\mathsf{DF}(\C)$. The exact functor $j^*\colon\A\lxr\C$ clearly defines a functor $j^*_\mathsf{F}\colon\mathsf{DF}(\A)\lxr\mathsf{DF}(\C)$ which factors uniquely through $\mathsf{DF}(\A)/\Ycal$. Hence, we get a functor $\phi\colon\mathsf{DF}(\A)/\Ycal\lxr \mathsf{DF}(\C)$. Conversely, since $j_*\colon\C\lxr \A$ is left exact, it induces a functor $\mathsf{KF}(\C)\lxr \mathsf{KF}(\A)$, which we also denote by $j_*$. Consider the composition
\[
\xymatrix{
\mathsf{KF}(\C) \ar[r]^{j_* } & \mathsf{KF}(\A) \ar[r] & \mathsf{DF}(\A) \ar[r] & \mathsf{DF}(\A)/\Y},
\]
where the last two functors are the obvious localisation functors. We claim that it sends filtered acyclic complexes to zero, hence yielding a functor $\psi\colon \mathsf{DF}(\C)\lxr\mathsf{DF}(\A)/\Y$. Let $(j_*(X),j_*(F))$ be an object in $\mathsf{D}\mathsf{F}(\A)$ where $(X,F)$ is acyclic in $\mathsf{KF}(\C)$. It suffices to show that $\mathsf{gr}_{j_*(F)}^n(j_*(X))$ lies in $i_*(\mathsf{D}(\B))$ for all integers $n$. From the exact sequence (\ref{shortexactseqderlevel}) this is equivalent to show that $j^*(\mathsf{gr}_{j_*(F)}^n(j_*(X)))=0$ for all integers $n$. Using the adjoint pair $(j^*,j_*)$ at the level of homotopy categories and since the counit $j^*j_*\lxr \iden_{\mathsf{K}(\C)}$ is a natural equivalence we derive that $j^*(\mathsf{gr}_{j_*(F)}^n(j_*(X)))\cong \mathsf{gr}^n_F(X)$ which is acyclic, thus zero in $\mathsf{D}(\C)$, proving the claim. Finally, using the counit of the adjunction induced by $(j^*,j_*)$ at the level of filtered complexes, one can check that $\phi$ and $\psi$ are quasi-inverse f-functors, thus finishing the proof.
\end{proof}

We now prove the main theorem of this section. 

\begin{thm}\label{criterium stratifying}
Let $\A$, $\B$ and $\C$ be abelian categories whose derived categories are TR5 and TR5*. 
Suppose furthermore that $\B$ is a Grothendieck category.  Let $\R$ be a recollement of the form
\begin{equation}\label{recollement}
\R\colon\ \ \ \ \ \ \xymatrix{\mathsf{D}(\B)  \ar[rrr]^{i_*} &&& \mathsf{D}(\A)  \ar[rrr]^{j^*}  \ar @/_1.5pc/[lll]_{i^*}  \ar @/^1.5pc/[lll]^{i^!} &&& \mathsf{D}(\C). \ar @/_1.5pc/[lll]_{j_!} \ar @/^1.5pc/[lll]^{j_*} }
\end{equation}
Then the following statements are equivalent.
\begin{enumerate}
\item There are abelian categories $\U$, $\V$ and $\W$ with a projective generator (respectively, an injective cogenerator) and a recollement of abelian categories
$$\xymatrix{\U \ar[rrr]^{I_*} &&& \V  \ar[rrr]^{J^*}  \ar @/_1.5pc/[lll]_{I^*}  \ar @/^1.5pc/[lll]^{I^!} &&& \W \ar @/_1.5pc/[lll]_{J_!} \ar @/^1.5pc/[lll]^{J_*} }$$
such that
\begin{itemize}
\item $\U$ is a Grothendieck category and $\mathsf{D}(\U)$ is left-complete;

\item The derived category $\mathsf{D}(\V)$ is TR5 and left-complete;

\item $I_*$ is a homological embedding and $I_*\colon\mathsf{D}(\U)\lxr\mathsf{D}(\V)$ preserve products and coproducts;

\item The associated derived recollement is  equivalent to $\R$ via restrictable equivalences.
\end{itemize}

\item There are bounded tilting (respectively, cotilting) objects $V$ in $\mathsf{D}(\A)$, $U$ in $\mathsf{D}(\B)$ and  $W$ in $\mathsf{D}(\C)$ such that
\begin{itemize}
\item $\Hcal_U$ is a Grothendieck category and $\mathsf{D}(\Hcal_U)$ is left-complete;
\item The derived category $\mathsf{D}(\Hcal_V)$ and left-complete;
\item there is an f-category $(\X,\theta)$ over $\mathsf{D}(\A)$ such that the realisation functor $\mathsf{real}^\X_V$ is an extendable equivalence;
\item the associated tilting t-structures in $\mathsf{D}(\B)$ and $\mathsf{D}(\C)$ glue along $\R$ to the associated tilting t-structure in $\mathsf{D}(\A)$.
\end{itemize}
\end{enumerate} 
\end{thm}
\begin{proof}
Once again, the tilting and cotilting cases are dual to each other. We prove the tilting case. 

(i)$\Longrightarrow$(ii): From Theorem \ref{thmrecolderivedcat}, we conclude that the recollement of abelian categories $\mathsf{R}_{\mathsf{ab}}(\U,\V,\W)$ can be derived. By assumption, there is an equivalence of recollements as follows:
\begin{equation}
\label{equivalenceofrecoldiagram1}
\xymatrix@C=0.5cm{ \mathsf{D}(\U) \ar[dd]_{\simeq}^{\Phi} \ar[rrr]^{I_*} &&& \mathsf{D}(\V) \ar[dd]_{\simeq}^{\Psi} \ar[rrr]^{J^*}  \ar @/_1.5pc/[lll]_{}  \ar @/^1.5pc/[lll]^{} &&& \mathsf{D}(\W) \ar[dd]_{\simeq}^{\Theta} \ar @/_1.5pc/[lll]_{} \ar @/^1.5pc/[lll]^{} \\ &&& &&& \\
 \mathsf{D}(\B) \ar[rrr]^{i_*} &&& \mathsf{D}(\A) \ar[rrr]^{j^*}  \ar @/_1.5pc/[lll]_{i^*}  \ar @/^1.5pc/[lll]^{i^!} &&& \mathsf{D}(\C) \ar @/_1.5pc/[lll]_{j_!} \ar @/^1.5pc/[lll]^{j_*} } 
\end{equation} 
where $\Phi$, $\Psi$ and $\Theta$ are restrictable equivalences. Therefore, by Theorem \ref{eq tilt cotilti shape}, there are bounded tilting objects $V$, $U$ and $W$ in $\mathsf{D}(\A)$, $\mathsf{D}(\B)$ and $\mathsf{D}(\C)$, respectively, such that $\Hcal_V\cong \V$, $\Hcal_U\cong \U$ and $\Hcal_W\cong \W$.

Since the top recollement is derived from an abelian recollement, it follows from Lemma~\ref{standard glue to standard}(i) that the standard t-structures in $\mathsf{D}(\U)$ and $\mathsf{D}(\W)$ glue to the standard t-structure in $\mathsf{D}(\W)$. Furthermore, the standard t-structures in the top recollement are sent to the the tilting t-structures generated by $V$, $U$ and $W$ in the bottom recollement. Hence, it follows from the commutativity of the diagram $(\ref{equivalenceofrecoldiagram1})$ that the glueing of the t-structures generated by $U$ and $W$ is the t-structure generated by $V$. The assumption that $\U$ is Grothendieck then translates into the fact that $\Hcal_U$ is Grothendieck (since they are equivalent abelian categories). Also the left-completeness properties required in (i) imply the left-completeness properties of (ii). Finally, observe from Proposition \ref{everything is real}(ii) that there is a choice of an f-category $(\X,\theta)$ over $\mathsf{D}(\A)$ such that $\Psi^b\cong \mathsf{real}^\X_V\circ \mathsf{D}^b(\Psi^0)$, where $\Psi^b$ is the restriction of $\Psi$ to $\mathsf{D}^b(\V)$ and $\Psi^0\colon \V\lxr \Psi(\V)=\Hcal_V$ is the naturally induced exact functor. Hence, $\mathsf{real}^\X_V$ is an extendable equivalence. 

(ii)$\Longrightarrow$(i): First, from Theorem \ref{glueing rec heart}, there is a recollement of hearts of the form
\[
\xymatrix{\Hcal_U \ar[rrr]^{\mathsf{H}^0_V i_*} &&& \Hcal_V  \ar[rrr]^{\mathsf{H}^0_W j^*}  \ar @/_1.5pc/[lll]_{\mathsf{H}^0_U i^*}  \ar @/^1.5pc/[lll]^{\mathsf{H}^0_U i^!} &&& \Hcal_W. \ar @/_1.5pc/[lll]_{\mathsf{H}^0_V j_!} \ar @/^1.5pc/[lll]^{\mathsf{H}^0_V j_*} }
\]
Note that there is a slight abuse of notation here: each of the functors in the recollement is in fact a triple composition - we are omitting the embedding of each heart in the corresponding triangulated category. Set $I_*:=\mathsf{H}^0_V i_*$ and $J^*:=\mathsf{H}^0_W j^*$ and keep the same notations for the corresponding derived functors. 

Consider the f-category $(\X,\theta)$ over $\mathsf{D}(\A)$ and the realisation functor $\mathsf{real}_V^\X$ which, by assumption, is an extendable equivalence. Let $(\Ycal,\xi)$ and $(\Zcal,\eta)$ be the induced f-categories over $\mathsf{D}(\B)$ and $\mathsf{D}(\C)$, respectively, so that $i_*$ and $j^*$ admit f-liftings (Proposition~\ref{f-categories over exact seq} and Corollary~\ref{ffadmitsflifting}). Since $i_*$ and $j^*$ are t-exact functors for the fixed t-structures (Theorem \ref{glueing rec heart}), it follows from Theorem~\ref{lift implies diagram} that we have commutative diagrams
\begin{equation}\label{commutative diagrams}
\xymatrix{\mathsf{D}^b(\Hcal_U)\ar[r]^{I_*}\ar[d]^{\mathsf{real}^\Ycal_U}&\mathsf{D}^b(\Hcal_V)\ar[d]^{\mathsf{real}_V^\Xcal}&& \mathsf{D}^b(\Hcal_V)\ar[d]^{\mathsf{real}_V^\Xcal}\ar[r]^{J^*}&\mathsf{D}^b(\Hcal_W)\ar[d]^{\mathsf{real}^\Zcal_W}\\ \mathsf{D}(\B)\ar[r]^{i_*}&\mathsf{D}(\A)&&\mathsf{D}(\A)\ar[r]^{j^*}&\mathsf{D}(\C).}
\end{equation}

From the left diagram we show that the functor $I_*$ is a homological embedding. Let $X$ and $Y$ be objects in $\Hcal_U$. Since $\mathsf{real}^\Y_U$ is fully faithful (see Proposition~\ref{der equival}) and acts as the identity on $X$ and $Y[n]$ for any $n$, it follows that $\Ext^n_{\Hcal_U}(X,Y)\cong\Hom_{\mathsf{D}(\Hcal_U)}(X,Y[n])\cong\Hom_{\mathsf{D}(\B)}(X,Y[n])$. Now, since $i_*$ is fully faithful, we get that $\Ext^n_{\Hcal_U}(X,Y)\cong\Hom_{\mathsf{D}(\A)}(i_*X,i_*Y[n])$. On the other hand, since $\mathsf{real}^\X_V$ is fully faithful (again by Proposition~\ref{der equival}) and acts as the identity on $\mathsf{H}^0_V i_*X=i_*X$ and $(\mathsf{H}^0_V i_*Y)[n]=i_*Y[n]$, it follows that 
\[
{\Ext}_{\Hcal_V}^n(\mathsf{H}^0_V i_*X,\mathsf{H}^0_V i_*Y)\cong\Hom_{\mathsf{D}(\Hcal_V)}(i_*X,i_*Y[n])\cong\Hom_{\mathsf{D}(\A)}(i_*X,i_*Y[n]),
\] 
thus showing that $I_*$ is a homological embedding.

By assumption the functor $\mathsf{real}^\X_V$ is extendable, that is, there is a restrictable equivalence, denoted by $\textsf{Real}^\X_V$, between $\mathsf{D}(\Hcal_{V})$ and $\mathsf{D}(\A)$ that restricts to $\mathsf{real}^\X_V$. In particular, $\mathsf{D}(\Hcal_V)$ is TR5. Since all the other completeness properties required in Theorem \ref{thmrecolderivedcat} are also satisfied by assumption, the proof of that same theorem yields an exact sequence of triangulated categories:
\[
\xymatrix{0\ar[r]&\mathsf{D}(\Hcal_U)\ar[r]^{I_*}&\mathsf{D}(\Hcal_V)\ar[r]^{J^*}&\mathsf{D}(\Hcal_W)\ar[r] &0.}
\]
Consider now the composition $F:= j^*\textsf{Real}^\X_VI_*\colon \mathsf{D}(\Hcal_U)\lxr \mathsf{D}(\C)$. Note that by the commutativity of $(\ref{commutative diagrams})$, the image of $\mathsf{D}^b(\Hcal_U)$ under $F$ is zero. Since $F$ is, by construction, t-exact with respect to the standard t-structure in $\mathsf{D}(\Hcal_U)$ and the tilting t-structure generated by $W$ in $\mathsf{D}(\C)$, we have $F(\mathsf{H}_0^i(X))=\mathsf{H}_W^i(F(X))$ for any object $X$ in $\mathsf{D}(\Hcal_U)$. But $\mathsf{H}_0^i(X)$ lies in $\mathsf{D}^b(\Hcal_U)$ and, hence, we conclude that $\mathsf{H}_W^i(F(X))=0$ for all $i\in\mathbb{Z}$. Since tilting t-structures are nondegenerate (see Remark \ref{nondeg}) it follows that $F\cong 0$. Hence the functor $\textsf{Real}^\X_V$ induces a functor $\Phi\colon \mathsf{D}(\Hcal_U)\lxr\mathsf{D}(\A)$ such that $i_*\Phi=\textsf{Real}^\X_VI_*$. As a consequence, $\textsf{Real}^\X_V$ also induces a functor $\Theta\colon \mathsf{D}(\Hcal_W)\lxr \mathsf{D}(\C)$ such that $\Theta J^*=j^*\textsf{Real}^\X_V$.  

It remains to show that $\Phi$ (and, thus, $\Theta$) are (restrictable) triangle equivalences. Let $\Phi^b$ denote the restriction of $\Phi$ to $\mathsf{D}^b(\Hcal_U)$. Since, by diagram $(\ref{commutative diagrams})$, in the bounded setting we have $i_*\Phi^b\cong \mathsf{real}^\X_VI_*\cong i_*\mathsf{real}^\Y_U$, it follows from the fully faithfulness of $i_*$ that $\Phi^b\cong \mathsf{real}^\Y_U$. Therefore, since $U$ is a bounded tilting object, we get from Corollary~\ref{der equival2} that the essential image of $\Phi^b$ is $\mathsf{D}^b(\B)$. By the relation $i_*\Phi=\textsf{Real}^\X_VI_*$ we get that $\Phi$ is fully faithful and therefore we can consider the essential image $\Image{\Phi}$ as a full subcategory of $\mathsf{D}(\B)$. Then, since by assumption $\mathsf{D}(\Hcal_U)$ is TR5, it follows that $\Image{\Phi}$ is a localising subcategory of $\mathsf{D}(\B)$. Moreover, $\Image{\Phi}$ contains $\mathsf{D}^b(\B)$ and thus, from Lemma~\ref{lemconsBR}(i), since $\B$ is Grothendieck, it follows that $\Image{\Phi}=\mathsf{D}(\B)$. We infer that the functor $\Phi$ is an equivalence, as wanted. Note that, similarly to the arguments above, it can also be seen that $\Theta^b\cong \mathsf{real}_W^{\Z}$.

Finally, observe that $\Phi$ and $\textsf{Real}^\X_V$ preserve products and coproducts since they are equivalences. Since also $i_*\colon\mathsf{D}(\B)\lxr\mathsf{D}(\A)$ preserves products and coproducts, it follows that so does $I_*\colon\mathsf{D}(\Hcal_U)\lxr\mathsf{D}(\Hcal_V)$. Hence, by Theorem \ref{thmrecolderivedcat}, the functor $I_*$ induces indeed a recollement of unbounded derived categories.
\end{proof}
 
The following result (Theorem C in the introduction) is a consequence of the above theorem and it provides necessary and sufficient conditions for a recollement of derived module categories to be equivalent to a stratifying recollement (recall Definition~\ref{defnstratifyingrecollement}). In this case almost all technical assumptions of Theorem~\ref{criterium stratifying} are automatically satisfied. In order to guarantee the extendability of a realisation functor, we assume the ring in the middle of the recollement to be a projective $\mathbb{K}$-algebra over a commutative ring $\mathbb{K}$. The statement reads then as follows.

\begin{cor}\label{rec strat k-alg}
Let $A$, $B$ and $C$ be rings. Assume that $A$ is a projective $\mathbb{K}$-algebra over a commutative ring $\mathbb{K}$. Suppose there is a recollement $\R$ of the form
\begin{equation}\label{recollement}
\R\colon\ \ \ \ \ \ \xymatrix{\mathsf{D}(B)  \ar[rrr]^{i_*} &&& \mathsf{D}(A)  \ar[rrr]^{j^*}  \ar @/_1.5pc/[lll]_{i^*}  \ar @/^1.5pc/[lll]^{i^!} &&& \mathsf{D}(C). \ar @/_1.5pc/[lll]_{j_!} \ar @/^1.5pc/[lll]^{j_*} }
\end{equation}
Then the following are equivalent.
\begin{enumerate}
\item The recollement $\R$ is equivalent to a stratifying recollement of a projective $\mathbb{K}$-algebra $S$.

\item There are compact tilting objects $V$ in $\mathsf{D}(A)$, $U$ in $\mathsf{D}(B)$ and  $W$ in $\mathsf{D}(C)$ such that the associated tilting t-structures in $\mathsf{D}(B)$ and $\mathsf{D}(C)$ glue along $\R$ to the associated tilting t-structure in $\mathsf{D}(A)$ and such that $\End_{\mathsf{D}(A)}(V)$ is a projective $\mathbb{K}$-algebra.
\end{enumerate} 
\end{cor}
\begin{proof}
We use the fact that a recollement of module categories is equivalent to a recollement induced by an idempotent element (\cite[Theorem 5.3]{PsaroudVitoria}). This corollary then becomes a direct application of Theorem \ref{criterium stratifying}, provided we show that in this setting the additional technical assumptions of the theorem are automatically satisfied. First note that both derived module categories and their duals satisfy Brown representability  (they are left-complete derived categories of a Grothendieck category, see Example \ref{ex right left}(v) and Theorem \ref{ex Brown}). Note, furthermore, that every equivalence between derived module categories is restrictable (this follows from 
Example~\ref{example restrict extend}(i)). Since we assume that the algebra $A$ is projective over $\mathbb{K}$, it also follows that the realisation functor of the compact tilting object $V$ with respect to $\mathsf{DF}(A)$ is an extendable equivalence since it is an equivalence of standard type (see Theorem \ref{standard f-lifts}). Finally, observe that the compactness of the tilting objects is used to produce hearts which are module categories (see Corollary~\ref{corheartmodulecat}).
\end{proof}

In the next result we show, using Lemma~\ref{standard glue to standard}(ii), that we can be more specific about the shape of equivalences between two stratifying recollements (compare with \cite[Theorem $3.5$]{Miyachi}).

\begin{cor}\label{cor 2 strat}
Let $\mathbb{K}$ be a commutative ring, $A$ and $B$ projective $\mathbb{K}$-algebras and $e$ and $u$ idempotents in $A$ and $B$ respectively such that $A/AeA$, $eAe$, $B/BuB$ and $uBu$ are also projective $\mathbb{K}$-algebras. Suppose that $f\colon A\lxr A/AeA$ and $g\colon B\lxr B/BuB$ are homological ring epimorphisms, and denote by $\R_f$ and $\R_g$ the induced recollements of unbounded derived module categories. The following statements are equivalent.
\begin{enumerate}
\item There is an equivalence of recollements from $\R_f$ to $\R_g$ with all equivalences being extensions to unbounded derived categories of equivalences of standard~type between bounded derived categories.
\item There are compact tilting objects $V$ in $\mathsf{D}(B)$, $U$ in $\mathsf{D}(B/BuB)$ and  $W$ in $\mathsf{D}(uBu)$ such that 
$$\mathsf{End}_{\mathsf{D}(B)}(V)\cong A,\ \  \mathsf{End}_{\mathsf{D}(B/BuB)}(U)\cong A/AeA,\ \  \mathsf{End}_{\mathsf{D}(uBu)}(W)\cong eAe$$
and the associated tilting t-structures in $\mathsf{D}(B/BuB)$ and $\mathsf{D}(uBu)$ glue along $\R_g$ to the associated tilting t-structure in $\mathsf{D}(A)$.
\end{enumerate}
\end{cor}
\begin{proof}
Following the proof of the Theorem~\ref{criterium stratifying}, we see that the choice of the f-categories for the realisation functors that yield the equivalence of recollements is the one provided by Proposition~\ref{f-categories over exact seq}. We start with the recollement induced by $g$ and, thus, with a recollement of abelian categories coming from a homological embedding. In this setting, Lemma \ref{standard glue to standard}(ii) shows that if we chose the f-category over $\mathsf{D}(B)$ to be the filtered derived category $\mathsf{DF}(B)$, then the induced f-categories over $\mathsf{D}(B/BuB)$ and $\mathsf{D}(uBu)$ are, respectively, the filtered derived categories $\mathsf{DF}(B/BuB)$ and $\mathsf{DF}(uBu)$. The result then follows from the fact that the equivalences built in the proof of Theorem \ref{criterium stratifying} are extensions of realisation functors of compact tilting objects with respect to filtered derived categories. These realisation functors are, therefore, derived equivalences of standard type by Theorem \ref{standard f-lifts}, finishing the proof. 
\end{proof} 

At this point we cannot prove with our techniques that the simple formulation of Corollary~\ref{rec strat k-alg}  holds for arbitrary rings (compare with \cite{AKLY3}). The main obstacle there is the existence of an extension of the realisation functor. In Corollary \ref{cor 2 strat}, the problem occurring is that we do not know whether an extension of an equivalence of bounded derived categories to an equivalence of unbounded derived categories is unique. Although we know that equivalences of standard type are extendable, we do not know whether the extensions obtained in the proof of Theorem \ref{criterium stratifying} coincide with the expected derived tensor product.

If, however, we turn our attention to recollements of bounded derived categories, we can formulate an analogue of Corollary~\ref{rec strat k-alg} even with more general assumptions.

\begin{cor}
Let $\A$, $\B$ and $\C$ be abelian categories with a projective generator and such that their unbounded derived categories are TR5 and TR5*. Suppose that there is a recollement of bounded derived categories:
\begin{equation}\nonumber
\xymatrix{
\R^b\colon & \mathsf{D}^b(\B)  \ar[rrr]^{i_*} &&& \mathsf{D}^b(\A)  \ar[rrr]^{j^*}  \ar @/_1.5pc/[lll]_{i^*}  \ar @/^1.5pc/[lll]^{i^!} &&& \mathsf{D}^b(\C). \ar @/_1.5pc/[lll]_{j_!} \ar @/^1.5pc/[lll]^{j_*} }
\end{equation}
Assume that the global dimension of $\A$ or $\C$ is finite. The following statements are equivalent.
\begin{enumerate}
\item The recollement $\R^b$ is equivalent to the restriction of a stratifying recollement to bounded derived categories.

\item There are compact tilting objects $V$ in $\mathsf{D}(\A)$, $U$ in $\mathsf{D}(\B)$ and  $W$ in $\mathsf{D}(\C)$ such that the associated tilting t-structures in $\mathsf{D}(\B)$ and $\mathsf{D}(\C)$ glue along $\R^b$ to the associated tilting t-structure in $\mathsf{D}(A)$.
\end{enumerate}
\end{cor}
\begin{proof}
Note that (i)$\Longrightarrow$(ii) follows as in the proof of Theorem \ref{criterium stratifying}. Conversely, following the arguments in the proof of Theorem \ref{criterium stratifying}, it easily follows that the induced fully faithful functor $\Hcal_U\lxr \Hcal_V$ is homological. Since all the hearts are module categories (the tilting objects are compact, see Corollary~\ref{corheartmodulecat}), it follows from \cite{PsaroudVitoria} that the recollement of hearts is equivalent to one induced by a homological ring epimorphism $f\colon S\lxr S/SeS$, where $S$ is Morita equivalent to $\End_{\mathsf{D}(\A)}(V)$. Since $\A$ or $\C$ have finite global dimension, then so does $S$ or $eSe$ (see Proposition \ref{fte gldim}). In any of these two cases, it follows from \cite[Theorem 7.2]{Psaroud:homolrecol} that $f\colon S\lxr S/SeS$ induces a recollement of bounded derived categories. It remains to show that this recollement is equivalent to $\R^b$. However, this follows from the same arguments used in the proof of Theorem~\ref{criterium stratifying}, omitting the issues related to the extendability of the realisation functors.
\end{proof}

We finish with an application of the above results to recollements of derived module categories of finite dimensional hereditary $\mathbb{K}$-algebras, over a field $\mathbb{K}$.

\begin{thm}\label{rec hereditary}
Let $A$ be a finite dimensional hereditary $\mathbb{K}$-algebra over a field $\mathbb{K}$. Then any recollement of $\mathsf{D}(A)$ by derived module categories is equivalent to a stratifying one.
\end{thm}
\begin{proof}
Let $\R$ be a recollement of $\mathsf{D}(A)$ of the form
\[
 \xymatrix@C=0.5cm{ \mathsf{D}(B) \ar[rrr]^{i_*} &&& \mathsf{D}(A)\ar[rrr]^{{j^*}}  \ar @/_1.5pc/[lll]_{i^*}  \ar @/^1.5pc/[lll]^{i^!} &&& \mathsf{D}(C).\ar @/_1.5pc/[lll]_{j_!} \ar @/^1.5pc/[lll]^{j_*}}\]
It follows from the assignments in \cite[Theorem 8.1]{KrauseStovicek} (see also \cite[Corollary 3.3]{AKL2}) for hereditary rings that $\R$ is equivalent to a recollement induced by a homological ring epimorphism. Moreover, this equivalence changes only the functors between $\mathsf{D}(B)$ and $\mathsf{D}(A)$, leaving the functors between $\mathsf{D}(A)$ and $\mathsf{D}(C)$ unchanged. Thus, without loss of generality, we assume $i_*=f_*$ for some homological ring epimorphism $f\colon A\lxr B$.

Now, since $A$ is a finite dimensional algebra of finite global dimension, the recollement $\R$ fits in an infinite ladder (see \cite[Proposition $3.7$]{AKLY}). In particular, the functors $i^*$ and $j_!$ (respectively, $i^!$ and $j_*$) admit left (respectively, right) adjoints and there is a recollement 
\[\R_u\colon\ \ \ \ \ \ \xymatrix@C=0.5cm{ \mathsf{D}(C) \ar[rrr]^{j_!} &&& \mathsf{D}(A)\ar[rrr]^{{i^*}}  \ar @/_1.5pc/[lll]_{j^\#}  \ar @/^1.5pc/[lll]^{j^*} &&& \mathsf{D}(B)\ar @/_1.5pc/[lll]_{i^\#} \ar @/^1.5pc/[lll]^{f_*}}\]
Applying once again the result quoted in the first paragraph, there is a recollement of  $\mathsf{D}(A)$ equivalent to $\R_u$, which is induced by a homological ring epimorphism (and the functors between $\mathsf{D}(A)$ and $\mathsf{D}(B)$ remain unchanged). Thus, without loss of generality, once again we assume that $j_!=g_*$ for some homological ring epimorphism $g\colon A\lxr C$. 

Denote by $\mathbb{T}$ the glueing of the standard t-structures in $\mathsf{D}(C)$ and $\mathsf{D}(B)$ along $\R$. We check that $\mathbb{T}$ is a tilting t-structure in $\mathsf{D}(A)$. The standard t-structures in $\mathsf{D}(C)$ and $\mathsf{D}(B)$ admit left adjacent co-t-structures (see \cite{Bondarko} for the definition). These co-t-structures, when glued along $\R_u$, yield a left adjacent co-t-structure to $\mathbb{T}$ (see \cite[Remark 2.6]{LiuVitoriaYang}). Note that, since $A$ has finite global dimension, then so do $B$ and $C$ (\cite[Proposition 2.14]{AKLY}) and, hence, $\R$ restricts to a recollement of $\mathsf{D}^b(A)$ (\cite[Corollary 4.10]{AKLY}). Since the standard t-structures are bounded, then $\mathbb{T}$ restricts to a bounded t-structure in $\mathsf{D}^b(A)$ and therefore Lemma~\ref{eq bdd}(iii) implies that $\mathbb{T}$ is an intermediate t-structure. From \cite[Theorem 4.6]{AMV1}, any intermediate t-structure that admits a left adjacent co-t-structure is a (bounded) silting t-structure and, thus, there is a bounded silting object $T$ such that $\mathbb{T}=\mathbb{T}_T$. 

Since $A$ is a finite dimensional $\mathbb{K}$-algebra, then so are $B$ and $C$ (\cite[Lemma 2.10(b)]{AKLY}). It follows that the recollements $\R$ and $\R_u$ also restrict to the level of bounded derived categories of finitely generated modules (see, for example, \cite[Theorem 4.4]{AKLY}). From \cite{LiuVitoriaYang}, $T$ is then compact and it can be computed explicitly. More precisely, following the terminology of \cite{LiuVitoriaYang}, $T$ is the glued silting object of $B$ and $C$ along $\R_u$. From \cite[Theorem 4.5]{LiuVitoriaYang}, $T$ is tilting if and only if the following conditions hold: 
\begin{enumerate}
\item $\Hom_{\mathsf{D}(C)}(C,j^\# f_*B[k])=0$, for all $k<-1$;
\item $\Hom_{\mathsf{D}(C)}(j^\# f_*B,C[k])=0$, for all $k<0$;
\item $\Hom_{\mathsf{D}(C)}(j^\# f_*B,j^\# f_*B[k])=0$, for all $k<-1$.
\end{enumerate}
We start by analysing the object $j^\# f_*B$. Since $j^\#$ is the left adjoint of $j_!=g_*$, we conclude that $j^\#$ is naturally equivalent to $-\otimes^{\mathbb{L}}_AC$. Since $C$ has projective dimension at most 1 as an $A$-module and $f_*B$ is an $A$-module, it follows that $j^\# f_*B$ is a 2-term complex in $\mathsf{D}(C)$ with cohomologies concentrated in degrees $0$ and $-1$. From this property, it is then obvious that conditions (i), (ii) and (iii) are satisfied and, thus, the glued object from $C$ and $B$ along $\R_u$ is tilting. The result then follows from Corollary \ref{rec strat k-alg}.
\end{proof}

\begin{rem}
\begin{enumerate}
\item Starting with an arbitrary recollement $\R$ of $\mathsf{D}(A)$ by derived module categories as in Theorem \ref{rec hereditary}, in general one needs to change all three terms by non-trivial derived equivalences in order to obtain a stratifying recollement. In the above proof, we change the left hand side to a derived equivalent ring that induces $\R$ via a homological ring epimorphism of $A$. Then we change the right hand side (i.e. the left hand side of $\R_u$) in the same way. Finally, we change the middle term by considering the derived equivalence given by the tilting object obtained by glueing.
\item In the proof of the above theorem we use the fact that we can restrict the recollement $\R$ and $\R_u$ to recollements of $\mathsf{D}^b(A)$ and of $\mathsf{D}^b(\smod{A})$. It is then not difficult to see that the proof of the above theorem yields that, for $A$ a finite dimensional hereditary $\mathbb{K}$-algebra, any recollement of $\mathsf{D}^b(A)$ (or of $\mathsf{D}^b(\smod{A})$) by bounded derived categories $\mathsf{D}^b(B)$ and $\mathsf{D}^b(C)$ (respectively, by $\mathsf{D}^b(\smod{B})$ and $\mathsf{D}^b(\smod{C})$) is the restriction to $\mathsf{D}^b(S)$ (respectively $\mathsf{D}^b(\smod{S})$) of a stratifying recollement of a finite dimensional $\mathbb{K}$-algebra $S$ derived equivalent to $A$.
\end{enumerate}
\end{rem}

\appendix
\section{On Beilinson's realization functor (by Ester Cabezuelo Fern\'andez and Olaf M.~Schn{\"u}rer\except{toc}{\protect\footnote{Mathematisches Institut, Universit{\"a}t Bonn,  Endenicher Allee 60, 53115 Bonn, Germany. E-mail addresses: \email{ester.cabezuelo@gmail.com}, \email{olaf.schnuerer@math.uni-bonn.de}}})}
 
\label{sec:beil-real-funct}

Given a t-structure with heart $\heartsuit$ on a triangulated
category $\mathcal{T}$, a classical problem is the construction
of a t-exact, and hence in particular triangulated, functor $\mathsf{D}^b(\heartsuit) \rightarrow \mathcal{T}$
extending the inclusion $\heartsuit \rightarrow \mathcal{T}$ (up to
isomorphism). 
Such a functor is called a realization functor.

Beilinson claims in 
\cite[Appendix]{B}
that a realization functor exists if $\mathcal{T}$ admits a
filtered triangulated category. 
Following Beilinson's hints, it is
possible to construct a functor 
$\mathsf{D}^b(\heartsuit) \rightarrow \mathcal{T}$ which is (isomorphic to) the identity on $\heartsuit$. This functor is compatible
with the shift functor on the source category and the suspension
functor on its target. However, we could not prove that it is
triangulated. To us, it seems unavoidable to impose an additional
axiom on the filtered triangulated category to ensure this.
This 
axiom says that a certain family of morphisms of triangles
can be completed to a $3 \times 3$-diagram of triangles
(see Appendix \ref{sec:additional-axiom}).

Interestingly, the same additional axiom was used in the
context of weight structures (or co-t-structures) in 
\cite{Schn}
to ensure
the existence of a strong 
weight complex functor under suitable assumptions.
Motivated by these two instances and our failed efforts to do
without this axiom,  
we strongly
believe that this axiom 
should be added to the definition of a filtered triangulated
category. 

In this appendix we explain how the additional axiom is used in
order to show that the realization functor is triangulated, under
the assumption that it is already constructed as a functor.
Full details will appear in 
the master thesis \cite{CF}
and presumably in a subsequent publication.

\subsection{The additional axiom}
\label{sec:additional-axiom}

We restate the additional axiom (fcat7) from 
\cite[7.2]{Schn}. 
We use notation and basic facts about 
filtered triangulated categories from
this article.

Let     
$(\widetilde{\mathcal{T}}, \widetilde{\mathcal{T}}(\leq 0),
\widetilde{\mathcal{T}}(\geq 0), s, \alpha)$ 
be a filtered triangulated category. The additional axiom we need is the following condition which
expresses a certain compatibility between the $\sigma$-truncation
triangles and the morphism $\alpha \colon \id_{\widetilde{\mathcal{T}}} \rightarrow s$
of triangulated functors.

\begin{enumerate}[label=(fcat{\arabic*}),start=7]
\item 
  \label{enum:filt-tria-cat-3x3-diag}
  For any morphism $f: X \rightarrow Y$ in $\widetilde{\mathcal{T}}$ the 
  morphism 
  \begin{equation*}
    \xymatrix@C2cm{
      {s(\sigma_{\geq 1}Y)} \ar[r]^-{s(g^1_Y)} & 
      {s(Y)} \ar[r]^-{s(k_Y^0)} &
      {s(\sigma_{\leq 0}Y)} \ar[r]^-{s(v_Y^0)} &
      {[1]s(\sigma_{\geq 1}Y)}\\
      {\sigma_{\geq 1}X} \ar[r]^-{g^1_X} \ar[u]^{\alpha_{\sigma_{\geq
            1}(Y)} \circ \sigma_{\geq 1}(f)} & 
      {X} \ar[r]^-{k_X^0} \ar[u]^{\alpha_Y \circ f} &
      {\sigma_{\leq 0}X} \ar[r]^-{v_X^0} \ar[u]^{\alpha_{\sigma_{\leq
            0}(Y)} \circ \sigma_{\leq 0}(f)}&
      {[1]\sigma_{\geq 1}X.} \ar[u]^{[1](\alpha_{\sigma_{\geq 1}(Y)}
        \circ \sigma_{\geq 1}(f))}
    }
  \end{equation*}
  of triangles
  can be extended to a $3 \times 3$-diagram of triangles.
  In other words, this morphism of triangles is 
  \define{middling-good} in Neeman's terminology
  \cite[Def.~2.4]{Neeman0}.
\end{enumerate}

\subsection{The realization functor is triangulated}
\label{sec:real-funct-triang}

Let     
$(\widetilde{\mathcal{T}}, \widetilde{\mathcal{T}}(\leq 0),
\widetilde{\mathcal{T}}(\geq 0), s, \alpha)$ 
be a filtered triangulated category over a triangulated category 
$\mathcal{T}$, i.\,e.\ we are given an equivalence $\mathcal{T} \xrightarrow{\sim}
\widetilde{\mathcal{T}}([0]) =
\widetilde{\mathcal{T}}(\leq 0) \cap
\widetilde{\mathcal{T}}(\geq 0)$ of triangulated categories.

Recall from
\cite[Prop.~6.9]{Schn} the functor 
\begin{equation}
  \label{eq:functor-c-as-cat-with-translation}
  c: (\widetilde{\mathcal{T}}, [1]s^{-1}) \rightarrow (\mathsf{C}^b(\mathcal{T}), \Sigma)
\end{equation}
of additive categories with translation (we use the notation
$[1]$ for suspension in $\widetilde{T}$ and $\Sigma$ for the shift of a
  complex): on objects  
$X \in \widetilde{\mathcal{T}}$ there is a given isomorphism
\begin{equation}
  \label{eq:c-Sigma-commute}
  c([1]s^{-1}(X))\cong \Sigma c(X).
\end{equation}

Now assume that $\mathcal{T}$ is endowed with a t-structure
$(\mathcal{T}^{\leq 0}, \mathcal{T}^{\geq 0})$. Then
there is a compatible t-structure 
$(\widetilde{\mathcal{T}}^{\leq 0}, \widetilde{\mathcal{T}}^{\geq 0})$ on
$\widetilde{\mathcal{T}}$, as claimed in 
\cite{B}. It is uniquely
characterized by the fact 
that the given equivalence
$\mathcal{T} \xrightarrow{\sim}
\widetilde{\mathcal{T}}([0])$ is t-exact and that $s(\widetilde{\mathcal{T}}
^{\leq 0})= \widetilde{\mathcal{T}}^{\leq -1}$. This implies
that $s(\widetilde{\mathcal{T}}
^{\geq 0})= \widetilde{\mathcal{T}}^{\geq -1}$. 

If we denote the hearts of these t-structures by $\heartsuit$ and
$\widetilde{\heartsuit}$, respectively, the functor~\eqref{eq:functor-c-as-cat-with-translation} restricts to an equivalence
\begin{equation}
  \label{eq:functor-c-on-heart}
  c: \widetilde{\heartsuit} \xrightarrow{\sim} \mathsf{C}^b(\heartsuit)
\end{equation}
of abelian categories, 
as claimed in 
\cite{B}.
Note that the heart $\widetilde{\heartsuit}$ is stable 
under $[1]s^{-1}$. Moreover, it is stable under the functors
$\sigma_{\geq n}$ and $\sigma_{\leq  n}$, and these functors
(together with the transformations $\sigma_{\geq n} \rightarrow \id$ and
$\id \rightarrow \sigma_{\leq n}$) correspond to the brutal truncation
functors on $\mathsf{C}^b(\heartsuit)$.

Consider the diagram 
\begin{equation*}
  \xymatrix{
    {\mathsf{C}^b(\heartsuit)} \ar[r] \ar[d]_-{c^{-1}}^-{\sim} &
    {\mathsf{D}^b(\heartsuit)} \ar@{..>}[rd]^-{\real}\\
    {\widetilde{\heartsuit}} \ar[r] &
    {\widetilde{\mathcal{T}}} \ar[r]^-{\omega} & 
    {\mathcal{T}}
  }
\end{equation*}
of categories where $c^{-1}$ is a quasi-inverse to the equivalence $c$,
the upper horizontal arrow is the localization with respect to
all quasi-isomorphisms, $\mathsf{D}^b(\heartsuit)$ is the bounded
derived category 
of $\heartsuit$, and the arrows in the lower row are the
inclusion functor and the triangulated functor ``forget the
filtration'' 
$\omega$ (see \cite[Prop.~6.6]{Schn}). All functors are functors of categories with
translation. It can be shown that the dotted arrow exists
uniquely as a
functor of categories with translation such that the diagram is
commutative, as claimed in
\cite{B}. 
This
functor is the realization functor. Its restriction to
$\heartsuit$ is isomorphic to the inclusion $\heartsuit \rightarrow
\mathcal{T}$. 

\begin{rem}
  If we assume that $\mathcal{T}=\widetilde{\mathcal{T}}([0])$ and
  choose $\omega$ and $c^{-1}$ wisely, 
  then the restriction of the realization functor to $\heartsuit$
  is the inclusion of $\heartsuit$ into $\mathcal{T}$.
\end{rem}

If $f' \colon X' \rightarrow Y'$ is a morphism in $\mathsf{C}^b(\heartsuit)$, let
\begin{equation*}
  X' \xrightarrow{f'} Y' \xrightarrow{\svek 10} \cone(f') \xrightarrow{\zvek 01} \Sigma X'
\end{equation*}
be the usual mapping cone sequence
where the cone 
$\cone(f')$ is given by $\cone(f')^n = Y'^n \oplus X'^{n+1}$ with
differential $\tzmat df0{-d}$.

Our main result, Theorem~\ref{t:real-triangulated}, will be an easy
consequence of the following technical proposition. 

\begin{prop}
  \label{p:identify-image}
  Keep the above assumptions and assume that
  $\widetilde{\mathcal{T}}$ satisfies
  axiom~\ref{enum:filt-tria-cat-3x3-diag}. 
  Let $f \colon X \rightarrow Y$ be a morphism
  in $\widetilde{\heartsuit}$. 
  Complete the morphism $f \circ \alpha =f \circ \alpha_{s^{-1}X} 
  \colon s^{-1}X \rightarrow Y$ to a triangle
  \begin{equation}
    \label{eq:f-alpha-triangle}
    s^{-1}X
    \xrightarrow{f \circ \alpha} Y \xrightarrow{g} Z \xrightarrow{h} [1]s^{-1}X
  \end{equation}
  in $\widetilde{\mathcal{T}}$.
  Then 
  there is an isomorphism $m$ such that the
  diagram
  \begin{equation*}
    \xymatrix{
      {c(Y)} \ar[r]^-{c(g)} \ar@{=}[d]&
      {c(Z)} \ar[r]^-{c(h)} \ar[d]_{m}^-{\sim}&
      {c([1]s^{-1}X)} \ar[d]^-{\sim}_-{\eqref{eq:c-Sigma-commute}}\\
      {c(Y)} \ar[r]^-{\svek 10} &
      {\cone(c(f))} \ar[r]^-{\zvek 01} &
      {\Sigma c(X)}
    }
  \end{equation*}
  in $\mathsf{C}^b(\heartsuit)$ is commutative.
\end{prop}

Before giving the proof, we remind the reader of the
definition of the  
functor~\eqref{eq:functor-c-as-cat-with-translation}
on the subcategory $\widetilde{\mathcal{T}}([-1,0])$.
Given an object $A \in \widetilde{\mathcal{T}}([-1,0])$, consider
the diagram
\begin{equation}
  \label{eq:reminder-c}
  \xymatrix{
    {[-1]s\sigma_{\leq -1}A} \ar@{..>}[rd] \\
    {[-1]\sigma_{\leq -1}A} \ar[r] \ar[u]^-{\alpha} &
    {\sigma_{\geq 0}A} \ar[r] &
    {A} \ar[r] &
    {\sigma_{\leq -1}A.}
  }
\end{equation}
Its lower row is obtained by rotation from the
obvious $\sigma$-truncation triangle of 
$A$. The dotted arrow exists uniquely such that the diagram is commutative, and $c(A)$ is the
dotted arrow viewed as a complex concentrated in degrees $-1,
0$. This construction is clearly functorial.

\begin{proof}
  Let $a \leq b$ be integers such that $X$ and $Y$ are objects
  of $\widetilde{\mathcal{T}}([a,b])$. We prove the proposition
  by induction over $b-a$.

  \textbf{Base case $a=b$.} Without loss of generality we can assume
  that $a=b=0$.

  We claim that the conclusion of the proposition is even true
  for all morphisms $f \colon X \rightarrow Y$ in
  $\widetilde{\mathcal{T}}([0])$.
  Note that $[1]s^{-1}X \in \widetilde{\mathcal{T}}([-1])$. Hence
  $Z \in \widetilde{\mathcal{T}}([-1,0])$ and
  the triangle \eqref{eq:f-alpha-triangle}
  identifies in a unique way with the triangle in 
  \eqref{eq:reminder-c} for $A=Z$. Hence $f$ identifies with the
  dotted arrow
  and  $c(Z)$ with the complex $(\cdots \rightarrow 0 \rightarrow X \xrightarrow{f}
  Y \rightarrow 0 \rightarrow \cdots)$. 
  
  Similarly, the triangle $0 \rightarrow Y \xrightarrow{\id} Y \rightarrow 0$ identifies
  with the triangle  
  \eqref{eq:reminder-c} for $A=Y$, and hence $c(Y)$ with $Y$
  viewed as a complex in degree zero; the triangle
  $s^{-1}X \rightarrow 0 \rightarrow [1]s^{-1}X \xrightarrow{\id} [1]s^{-1}X$ identifies
  with the triangle \eqref{eq:reminder-c} for $A=[1]s^{-1}X$, and
  hence $c([1]s^{-1}X)$ with $X$ viewed as a complex in degree
  $-1$, and this identification coincides with
  \eqref{eq:c-Sigma-commute}. 
  
  It is easy to see that the morphisms $c(Y) \xrightarrow{c(g)} c(Z)
  \xrightarrow{c(h)}  
  c([1]s^{-1}X)$ correspond under these identifications to the
  obvious morphisms 
  \begin{equation*}
    (0 \rightarrow Y) \rightarrow (X \xrightarrow{f} Y) \rightarrow (X \rightarrow 0)
  \end{equation*}
  of complexes concentrated in degrees $-1$, $0$.
  This proves the base case of our induction.

  \textbf{Induction step.} 
  Assume now that $f \colon X \rightarrow Y$ is a morphism in
  $\widetilde{\mathcal{T}}([a,b]) \cap \widetilde{\heartsuit}$. Without
  loss of generality we 
  can assume that $a \leq 1 < b$. 
  By induction, we already know the conclusion of the
  proposition for the morphisms $\sigma_{\geq 2} f$
  and $\sigma_{\leq 1} f$.

  Axiom~\ref{enum:filt-tria-cat-3x3-diag} applied to the
  morphism $s^{-1}f \colon s^{-1}X \rightarrow s^{-1}Y$ shows that the
  following morphism of triangles is middling-good.
  \begin{equation}
    \label{eq:morphism-for-axiom}
    \hspace{-1.4cm}
    \xymatrix@C2cm{
      {s(\sigma_{\geq 1}{s^{-1}Y})} \ar[r]^-{s(g^1_{s^{-1}Y})} 
      & 
      {s({s^{-1}Y})} \ar[r]^-{s(k_{s^{-1}Y}^0)} 
      &
      {s(\sigma_{\leq 0}{s^{-1}Y})} \ar[r]^-{s(v_{s^{-1}Y}^0)} 
      &
      {[1]s(\sigma_{\geq 1}{s^{-1}Y})} 
      \\
      {\sigma_{\geq 1}{s^{-1}X}} \ar[r]^-{g^1_{s^{-1}X}}
      \ar[u]^{\alpha_{\sigma_{\geq 1}({s^{-1}Y})} \circ \sigma_{\geq 1}({s^{-1}f})} & 
      {{s^{-1}X}} \ar[r]^-{k_{s^{-1}X}^0} \ar[u]^{\alpha_{s^{-1}Y} \circ {s^{-1}f}} &
      {\sigma_{\leq 0}{s^{-1}X}} \ar[r]^-{v_{s^{-1}X}^0} \ar[u]^{\alpha_{\sigma_{\leq
            0}({s^{-1}Y})} \circ \sigma_{\leq 0}({s^{-1}f})} &
      {[1]\sigma_{\geq 1}{s^{-1}X}} \ar[u]^{[1](\alpha_{\sigma_{\geq 1}({s^{-1}Y})}
        \circ \sigma_{\geq 1}({s^{-1}f}))}
    }
  \end{equation}
  We claim that this morphism of triangles is isomorphic to the
  following morphism of triangles.
  \begin{equation}
    \label{eq:morphism-we-want}
    \hspace{-1.4cm}
    \xymatrix@C2cm{
      {\sigma_{\geq 2}Y} \ar[r]^-{g^2_Y} 
      & 
      {Y} \ar[r]^-{k_Y^1} 
      &
      {\sigma_{\leq 1}Y} \ar[r]^-{v_Y^1} 
      &
      {[1]\sigma_{\geq 2}Y} 
      \\
      {s^{-1}(\sigma_{\geq 2}X)} \ar[r]^-{s^{-1}(g^2_X)}
      \ar[u]^{\sigma_{\geq 2}(f) \circ \alpha_{s^{-1}(\sigma_{\geq 2}(X))}} & 
      {{s^{-1}X}} \ar[r]^-{s^{-1}(k_X^1)} 
      \ar[u]^{f \circ \alpha_{s^{-1}X}} &
      {s^{-1}(\sigma_{\leq 1}X)} \ar[r]^-{s^{-1}(v_X^1)} 
      \ar[u]^{\sigma_{\leq 1}(f) \circ \alpha_{s^{-1}(\sigma_{\leq 1}(Y))} } &
      {[1]s^{-1}(\sigma_{\geq 2}X)} 
      \ar[u]^{[1](\sigma_{\geq 2}(f) \circ \alpha_{s^{-1}(\sigma_{\geq 2}(X))})}
    }
  \end{equation}
  This is straightforward to prove: The two lower triangles in 
  \eqref{eq:morphism-for-axiom} and 
  \eqref{eq:morphism-we-want} 
  are uniquely isomorphic by a morphism of triangles whose
  middle component is the identity.
  This isomorphism of triangles is functorial in $X$. Hence $f
  \colon X \rightarrow Y$ yields a morphism between these isomorphisms of
  triangles; in other words, we obtain a commutative square whose
  vertices are triangles and whose
  arrows are  
  morphisms of triangles, two of them being isomorphisms.
  Applying the morphism $\alpha$ of triangulated functors to this
  square yields our claim.

  Hence also the morphism~\eqref{eq:morphism-we-want} of
  triangles is middling-good, say it fits into the following $3 \times
  3$-diagram whose right-most column we do not draw.
  \begin{equation}
    \xymatrix@C2cm{
      {[1]s^{-1}(\sigma_{\geq 2}X)} \ar[r]^-{[1]s^{-1}(g^2_X)}
      & 
      {[1]{s^{-1}X}} \ar[r]^-{[1]s^{-1}(k_X^1)} 
      &
      {[1]s^{-1}(\sigma_{\leq 1}X)} 
      \\
      {A} \ar[u]^-{b} \ar[r] & 
      {Z} \ar[u] \ar[r] &
      {B} \ar[u]
      \\
      {\sigma_{\geq 2}Y} \ar[r]^-{g^2_Y} 
      \ar[u]^-{a} 
      & 
      {Y} \ar[r]^-{k_Y^1} 
      \ar[u]
      &
      {\sigma_{\leq 1}Y} 
      \ar[u]
      \\
      {s^{-1}(\sigma_{\geq 2}X)} \ar[r]^-{s^{-1}(g^2_X)}
      \ar[u]^{\sigma_{\geq 2}(f) \circ \alpha_{s^{-1}(\sigma_{\geq 2}(X))}} & 
      {{s^{-1}X}} \ar[r]^-{s^{-1}(k_X^1)} 
      \ar[u]^{f \circ \alpha_{s^{-1}X}} &
      {s^{-1}(\sigma_{\leq 1}X)} 
      \ar[u]^{\sigma_{\leq 1}(f) \circ
        \alpha_{s^{-1}(\sigma_{\leq 1}(Y))} } 
    }
  \end{equation}
  In the
  following we will only need that all rows can be extended to
  triangles, that all columns are triangles, and that the diagram
  is commutative. 
  
  Now we use the assumption that $X$ and $Y$ are objects of the heart
  $\widetilde{\heartsuit}$. 
  If we ignore the bottom row of this diagram, we
  obtain a $3 \times 3$-diagram of nine objects that
  are easily 
  seen to be 
  objects of the heart $\widetilde{\heartsuit}$. 
  More precisly, we obtain a commutative $3 \times 3$-diagram in
  the abelian  
  category
  $\widetilde{\heartsuit}$ all of whose rows and columns are short
  exact sequences since they come from triangles.

  Using induction (and the
  isomorphism~\eqref{eq:c-Sigma-commute} and the fact that the
  functors $\sigma_{\geq n}$, $\sigma_{\leq n}$ correspond to the
  brutal truncations), the image of this
  diagram under the 
  equivalence $c$ in \eqref{eq:functor-c-on-heart}
  is
  isomorphic to the commutative diagram~\eqref{eq:3x3-abelian}
  below of short exact
  sequences in 
  $\mathsf{C}^b(\heartsuit)$ 
  where we use the following notation:
  $f' \colon X' \rightarrow Y'$ is the image of $f \colon X \rightarrow Y$
  under the functor $c$, and 
  $f'_{\geq 2} \colon X'_{\geq 2} \rightarrow Y'_{\geq 2}$ and
  $f'_{\leq 1} \colon X'_{\leq 1} \rightarrow Y'_{\leq 1}$ are the
  brutal truncations of $f'$.
  \begin{equation}
    \label{eq:3x3-abelian}
    \xymatrix@C2cm{
      {\Sigma X'_{\geq 2}} \ar[r]
      & 
      {\Sigma X'} \ar[r]
      &
      {\Sigma X'_{\leq 1}}
      \\
      {\cone(f'_{\geq 2})} \ar[u]^-{\zvek 01} \ar[r] & 
      {M} \ar[u] \ar[r] &
      {\cone(f'_{\leq 1})} \ar[u]^-{\zvek 01} & 
      \\
      {Y'_{\geq 2}} \ar[r]
      \ar[u]^-{\svek 10}
      & 
      {Y'} \ar[r]
      \ar[u]
      &
      {Y'_{\leq 1}}
      \ar[u]^-{\svek 10}
    }
  \end{equation}
  The horizontal arrows in top and bottom row
  are the usual morphisms from brutal truncation.
  
  We need to see that there is an isomorphism $M \xrightarrow{\sim} \cone(f')$
  such that the middle column of our diagram is identified with
  the obvious part of the usual mapping cone sequence of $f'$.
  
  In degree one, diagram~\eqref{eq:3x3-abelian} looks as follows.
  \begin{equation*}
    \xymatrix@C2cm{
      {X'^2} \ar[r]^-{\id}
      & 
      {X'^2} \ar[r]
      &
      {0}
      \\
      {X'^2} \ar[u]^-{\id} \ar[r] & 
      {M^1} \ar[u] \ar[r] &
      {Y'^1} \ar[u] & 
      \\
      {0} \ar[r]
      \ar[u]
      & 
      {Y'^1} \ar[r]^-{\id}
      \ar[u]
      &
      {Y'^1}
      \ar[u]^-{\id}
    }
  \end{equation*}
  The epimorphism in the middle column of this diagram is split by
  the monomorphism in its middle row. We obtain an isomorphism
  $M^1 \xrightarrow{\sim} Y'^1 
  \oplus X'^2$ such that the four arrows
  ending and starting at $M^1$ correspond to the obvious
  inclusions and projections.
  
  In degrees $n \not=1$ it is even simpler to find similar
  isomorphisms $M^n \xrightarrow{\sim} Y'^n \oplus X'^n$: 
  the right column of
  diagram~\eqref{eq:3x3-abelian} is zero
  in degrees $n>1$, and its left column is zero in degrees
  $n<1$.

  In this way we obtain an isomorphism 
  $M \xrightarrow{\sim} Y' \oplus \Sigma X'$
  of graded objects in $\heartsuit$ such that all morphisms starting
  and ending in $M$ correspond to the obvious inclusions and
  projections. 
  An easy computation shows that the differential of $M$ has
  the form $\tzmat df0{-d}$ modulo this isomorphism.
  This proves the proposition.
\end{proof}

\begin{thm}
  \label{t:real-triangulated}
  Keep the above assumptions and assume that
  $\widetilde{\mathcal{T}}$ satisfies
  axiom~\ref{enum:filt-tria-cat-3x3-diag}. Then the realization
  functor $\real \colon \mathsf{D}^b(\heartsuit) \rightarrow \mathcal{T}$ is
  triangulated.
\end{thm}

\begin{proof}
  Let $f \colon X \rightarrow Y$ be a morphism in $\widetilde{\mathcal{T}}$,
  and fit $f \circ \alpha$ into a triangle
  \eqref{eq:f-alpha-triangle}. 
  Consider 
  the sequence
  \begin{equation}
    \label{eq:lift-mapping-cone-sequence}
    X \xrightarrow{f} Y \xrightarrow{g} Z \xrightarrow{h} [1]s^{-1}X
  \end{equation}
  in $\widetilde{\heartsuit}$.
  Its image under 
  $c$ is the first row in the following commutative diagram
  in $\mathsf{C}^b(\heartsuit)$ obtained from
  Proposition~\ref{p:identify-image}.
  \begin{equation*}
    \xymatrix{
      {c(X)} \ar[r]^-{c(f)} \ar@{=}[d]&
      {c(Y)} \ar[r]^-{c(g)} \ar@{=}[d]&
      {c(Z)} \ar[r]^-{c(h)} \ar[d]_{m}^-{\sim}&
      {c([1]s^{-1}X)} \ar[d]^-{\sim}_-{\eqref{eq:c-Sigma-commute}}\\
      {c(X)} \ar[r]^-{c(f)} &
      {c(Y)} \ar[r]^-{\svek 10} &
      {\cone(c(f))} \ar[r]^-{\zvek 01} &
      {\Sigma c(X).}
    }
  \end{equation*}

  Since any triangle in $\mathsf{D}^b(\heartsuit)$ is isomorphic to the
  image of 
  mapping cone sequence in $\mathsf{C}^b(\heartsuit)$,
  the equivalence~\eqref{eq:functor-c-on-heart}
  and our description of the realization functor $\real$ show
  that it is enough to see that the functor ``forget the
  filtration'' $\omega$ maps
  the sequence \eqref{eq:lift-mapping-cone-sequence} in
  $\widetilde{\heartsuit}$ to a triangle in $\mathcal{T}$.
  But this is easy to see. First, the functor $\omega$ is triangulated
  and hence maps the triangle \eqref{eq:f-alpha-triangle} to a
  triangle in $\mathcal{T}$. Second, $\omega(\alpha)$ is an
  isomorphism, by the very definition of $\omega$.
  This proves the theorem.
\end{proof}

\end{document}